%% file: main_arxiv.tex
\documentclass[letterpaper,11pt]{article}
\usepackage[margin=1in]{geometry}

\usepackage{amsmath}
\usepackage{amssymb}
\usepackage{amsthm}
\usepackage{mathrsfs}
\usepackage{bm}
\usepackage{enumerate}
\usepackage{color}
\usepackage{graphicx}
\usepackage{subcaption}
\usepackage{booktabs}
\usepackage{ifthen}
\newboolean{arxiv}
\setboolean{arxiv}{TRUE}

\usepackage[
            CJKbookmarks=true,
            bookmarksnumbered=true,
            bookmarksopen=true,
            colorlinks=true,
            citecolor=red,
            linkcolor=blue,
            anchorcolor=red,
            urlcolor=blue
            ]{hyperref}

\newcommand{\reals}{\mathbb{R}}
\newcommand{\naturals}{\mathbb{N}}
\newcommand{\integers}{\mathbb{Z}}

\newcommand{\Expect}{\mathbb{E}}

\newcommand{\Poi}{\mathsf{Poi}}
\newcommand{\supp}{\ensuremath{\mathrm{supp}}}

\usepackage{prettyref}
\newrefformat{eq}{(\ref{#1})}
\newrefformat{ineq}{(\ref{#1})}
\newrefformat{thm}{Theorem~\ref{#1}}
\newrefformat{th}{Theorem~\ref{#1}}
\newrefformat{chap}{Chapter~\ref{#1}}
\newrefformat{sec}{Section~\ref{#1}}
\newrefformat{seca}{Section~\ref{#1}}
\newrefformat{subsec}{Section~\ref{#1}}
\newrefformat{algo}{Algorithm~\ref{#1}}
\newrefformat{fig}{Fig.~\ref{#1}}
\newrefformat{tab}{Table~\ref{#1}}
\newrefformat{rmk}{Remark~\ref{#1}}
\newrefformat{remark}{Remark~\ref{#1}}
\newrefformat{clm}{Claim~\ref{#1}}
\newrefformat{def}{Definition~\ref{#1}}
\newrefformat{cor}{Corollary~\ref{#1}}
\newrefformat{lmm}{Lemma~\ref{#1}}
\newrefformat{prop}{Proposition~\ref{#1}}
\newrefformat{pr}{Proposition~\ref{#1}}
\newrefformat{app}{Appendix~\ref{#1}}
\newrefformat{apx}{Appendix~\ref{#1}}
\newrefformat{ex}{Example~\ref{#1}}
\newrefformat{exer}{Exercise~\ref{#1}}
\newrefformat{soln}{Solution~\ref{#1}}

\newcommand{\PML}{\mathsf{PML}}

\newcommand{\NPMLE}{\mathsf{NPMLE}}

\newcommand{\pth}[1]{\left( #1 \right)}
\newcommand{\qth}[1]{\left[ #1 \right]}
\newcommand{\sth}[1]{\left\{ #1 \right\}}

\newcommand{\Bern}{\text{Bern}}

\newcommand{\indc}[1]{{\mathbf{1}\left\{{#1}\right\}}}
\newcommand{\Indc}{\mathbf{1}}

\newcommand{\calN}{{\mathcal{N}}}

\newcommand{\calP}{{\mathcal{P}}}

\def\R{\mathbb{R}}
\def\Z{\mathbb{Z}}
\def\P{\mathbb{P}}
\def\E{\mathbb{E}}

\def\cE{\mathcal{E}}

\def\eps{\varepsilon}

\def\d{\mathrm{d}}

\def\1{\mathbf{1}}

\newcommand{\equald}{\stackrel{d}{=}}

\newcommand{\pnorm}[2]{\lVert #1\rVert_{#2}}

\DeclareMathOperator{\Var}{Var}
\DeclareMathOperator{\var}{Var}

\DeclareMathOperator{\DKL}{D_{\mathrm{KL}}}

\DeclareMathOperator{\reg}{\mathsf{Reg}}

\DeclareMathOperator{\Sym}{\mathrm{S}}

\DeclareMathOperator{\PI}{\mathrm{PI}}
\DeclareMathOperator{\poi}{\mathsf{Poi}}

\DeclareMathOperator{\MGT}{MGT}
\newcommand{\pnatural}{\hat p^{\mathrm{nat}}}
\newcommand{\pNPMLE}{\widehat{p}^{\rm NPMLE}}
\newcommand{\UNPML}{\mathsf{unNPMLE}}
\def\argmin{\mathop{\rm argmin}}
\def\argmax{\mathop{\rm argmax}}

\def\Prob{\mathbb{P}}

\newtheorem{theorem}{Theorem}[section]
\newtheorem{lemma}[theorem]{Lemma}
\newtheorem{proposition}[theorem]{Proposition}

\theoremstyle{definition}

\newtheorem{remark}[theorem]{Remark}
\newtheorem{example}[theorem]{Example}

\usepackage{cleveref}
\crefname{lemma}{Lemma}{Lemmas}
\Crefname{lemma}{Lemma}{Lemmas}
\crefname{thm}{Theorem}{Theorems}
\Crefname{thm}{Theorem}{Theorems}

\newcommand{\stepa}[1]{\overset{\rm (a)}{#1}}
\newcommand{\stepb}[1]{\overset{\rm (b)}{#1}}
\newcommand{\stepc}[1]{\overset{\rm (c)}{#1}}
\newcommand{\stepd}[1]{\overset{\rm (d)}{#1}}

\def\hp{\tau} 
\newcommand{\eqbr}[1]{(\ref{#1})}

\begin{document}

\title{Besting Good--Turing: Optimality of Non-Parametric Maximum Likelihood for Distribution Estimation}
\author{Yanjun Han$^*$, Jonathan Niles-Weed\thanks{
Courant Institute of Mathematical Sciences and Center for Data Science, New York University, \texttt{yanjunhan@nyu.edu}, \texttt{jnw@cims.nyu.edu}.} , Yandi Shen\thanks{Department of Statistics and Data Science, Carnegie Mellon University, \texttt{yandis@andrew.cmu.edu}.} , and Yihong Wu\thanks{Department of Statistics and Data Science, Yale University, \texttt{yihong.wu@yale.edu}.}
}

\date{}

\maketitle

\begin{abstract}
When faced with a small sample from a large universe of possible outcomes, scientists often turn to the venerable Good--Turing estimator. Despite its pedigree, however, this estimator comes with considerable drawbacks, such as the need to hand-tune smoothing parameters and the lack of a precise optimality guarantee. We introduce a parameter-free estimator that bests Good--Turing in both theory and practice. Our method marries two classic ideas, namely Robbins's empirical Bayes and Kiefer--Wolfowitz non-parametric maximum likelihood estimation (NPMLE), to learn an implicit prior from data and then convert it into probability estimates. We prove that the resulting estimator attains the optimal instance-wise risk up to logarithmic factors in the competitive framework of Orlitsky and Suresh, and that the Good--Turing estimator is strictly suboptimal in the same framework. Our simulations on  synthetic data and experiments with English corpora and U.S. Census data show that our estimator consistently outperforms both the Good--Turing estimator and explicit Bayes procedures.
\end{abstract}

\setcounter{tocdepth}{3}
\tableofcontents


\input{intro.tex}

\input{proofs.tex}
\input{GT_lower_bound.tex}

\appendix 
\input{preliminary.tex}
\input{additional_results.tex}

\clearpage 

\ifthenelse{\boolean{arxiv}}{
\paragraph{Acknowledgements.} Y. Han is supported in part by a fund from Renaissance Philanthropy. J. Niles-Weed is supported in part by the National Science Foundation under Grant No.~DMS-2339829.
Part of the work of Y.~Wu was supported by the National Science Foundation under Grant No.~DMS-1928930, while Y.~Wu was in residence at the Simons Laufer Mathematical Sciences Institute in Berkeley, California, during the Spring 2025 semester.

Y.~Wu is grateful to Y.~Nie and Y.~Polyanskiy for helpful discussions at the onset of project. In particular, the method-of-type argument for bounding  \prettyref{eq:Gaussian_first_term}, previously used for $m=2$ in \cite{nie2023large} for a different problem, was initially suggested by Y.~Polyanskiy.
Y.~Wu thanks Yale Center for Research Computing for providing the computing resources for the experiments carried out in this paper.
}{}

\bibliographystyle{alpha}
\bibliography{mybib,newbib,reports}

	
\end{document}

%% file: intro.tex
\ifthenelse{\boolean{arxiv}}{\section{Introduction}}{}
\ifthenelse{\boolean{arxiv}}{E}{\dropcap{E}}stimating a probability distribution from a sample of data is a fundamental problem across the sciences.
An early version arose in ecology, where Fisher, Corbet, and Williams \cite{fisher1943relation} studied the distribution of species abundances in samples of Malayan butterflies.
In the subsequent decades, the same question has appeared in genomics, when estimating the distribution of k-mers in DNA sequences~\cite{marcais2011fast}; in immunology, when estimating the diversity of the human T-cell receptor repertoire~\cite{laydon2015estimating}; and in linguistics, when assessing the statistical properties of natural language~\cite{sparck1972statistical}.
A challenge common to these applications is that the set of possible observations---whether butterfly species or English words---is large compared with the sample size.

For example, natural language modeling requires accurately estimating the probability that bigrams such as \textit{new car} or \textit{gut feeling} appear in English text.
The most natural method of estimating a probability distribution is by empirical frequencies: counting the occurrence of each bigram in a large corpus, and dividing by the total number of bigrams.
In statistical terms, these empirical frequencies are maximum likelihood estimates (MLEs) of the unknown probabilities, and they clearly perform well when the corpus is large enough to cover all possible pairs of words.
However, the obvious drawback of this approach is that it assigns zero probability to rare but possible bigrams---like \textit{possible bigrams}---which happen not to occur in the corpus.
The enormous number of possible bigrams in English means that this situation is the norm rather than the exception.
To avoid assigning zero probability to unseen symbols, various techniques have been developed to adjust empirical frequencies by redistributing probability mass from high-count to low-count symbols, resulting in a smoother, less jagged distribution. Notable examples include the add-$c$ estimators that assign probabilities proportional to the empirical count plus a constant $c$, such as the Laplace's rule of succession ($c=1$) and the Krichevsky--Trofimov estimator ($c=\frac{1}{2}$), which have been widely used in natural language processing \cite{JurafskyMartin,chen1999empirical} and data compression \cite{krichevsky1981performance,cover2012elements}.

In 1953, Good \cite{good1953population}, reporting on a collaboration with Alan Turing at Bletchley Park motivated by natural language processing for Allied code breaking, proposed a more sophisticated method for estimating probability distributions.
This approach, since known as the Good--Turing estimator, evinces much better performance than the empirical frequency estimator in practice, and has become a standard tool~\cite{gale1995good}.
In fact, its widespread adoption preceded by several decades a satisfying theoretical explanation for its benefits.
Only in the last 25 years has a series of papers~\cite{mcallester2000convergence,orlitsky2003always,orlitsky2015competitive} obtained evidence for its theoretical superiority over other approaches as well.

The most convincing argument was given by Orlitsky and Suresh~\cite{orlitsky2015competitive}, who developed a \textit{competitive optimality} guarantee for the Good--Turing estimator.
This result shows that the Good--Turing estimator always performs nearly as well as an ``oracle'' estimator that has prior information about the true unknown probabilities.
As the sample size increases, Orlitsky and Suresh show that the performance of a version of the Good--Turing estimator approaches that of a hypothetical estimator that is given knowledge of the multiset of probabilities in the distribution, but not their assignment to individual items.
This argument establishes not only that the Good--Turing estimator successfully estimates the unknown probabilities, but also that it performs nearly as well \textit{on each problem instance} as an estimator that knows information about the unknown distribution in advance.

Despite these results, current theoretical analyses fall short of showing that the Good--Turing estimator is unimprovable.
For instance, though Orlitsky and Suresh show that the Good--Turing estimator approaches the performance of an oracle-aided estimator, they do not show that it approaches at an optimal \textit{rate}.
Their work leaves open the possibility that there exist other estimators that are much more competitive.

Moreover, from a practical perspective, the Good--Turing estimator suffers from several deficiencies.
First, as was already noted by Good, the benefits of the Good--Turing estimator are most pronounced for rare items, which occur infrequently in the data set.
In fact, the estimator \textit{underperforms} the empirical frequency estimate for items whose observed frequencies are large. 
In practice, therefore, use of the Good--Turing estimator requires the choice of an additional tuning parameter, which acts as a threshold: for items whose observed counts are below this threshold, the Good--Turing estimate is used, above it, the empirical frequency estimate.

More generally, as with the empirical frequency estimator, Good's original proposal involved ``smoothing'' the vector of observed frequencies prior to applying the estimator, to obtain a less jagged estimate.
Though Good proposes several strategies for smoothing, none of them have significant theoretical justification; moreover, the empirical and theoretical performance of the Good--Turing estimator is unfortunately sensitive to the choice of smoothing procedure~\cite{chen1999empirical,orlitsky2003always}.
The practitioner who wishes to employ the Good--Turing approach is faced with the unenviable task of experimenting with different, somewhat arbitrary variations of the estimator, none of which is guaranteed to perform well on the data at hand.

In this work, we propose a different approach to distribution estimation that bests the Good--Turing estimator. Our approach is based on two cornerstones of 20th century statistical theory: \emph{non-parametric} maximum likelihood estimation (NPMLE), introduced by Kiefer and Wolfowitz~\cite{kiefer1956consistency}, and empirical Bayes (EB), introduced by Robbins~\cite{Robbins56}.
Unlike the vanilla MLE, which estimates each probability individually by the empirical frequency, the NPMLE jointly estimates the probabilities simultaneously, allowing our estimator to adapt to global features of the unknown distribution to achieve better performance.
Moreover, unlike the Laplace and Krichevsky--Trofimov estimators, which can be viewed as Bayes estimators under the Dirichlet prior, 
our empirical Bayes approach learns a prior from data as opposed to adopting a pre-specified choice. 
Finally, unlike the Good--Turing estimator, our estimator does not require the adoption of \textit{ad hoc} smoothing techniques or optimization of tuning parameters.
Our main theoretical result is that the new estimator is, in the competitive sense of Orlitsky and Suresh~\cite{orlitsky2015competitive}, \textit{optimal} for the distribution estimation problem, and we show in experiments that it significantly outperforms the Good--Turing approach in a number of challenging settings.
Our approach offers a computationally efficient, parameter-free solution to the distribution estimation problem and a redemption of maximum likelihood estimation techniques for this important statistical task.

\ifthenelse{\boolean{arxiv}}{\subsection{Preliminaries}}{\section*{Preliminaries}}
To describe our estimator and theoretical results, we fix some notation.
Consider a vector $p^\star = (p_1^\star, \dots, p_k^\star) \in \Delta_k$, where $\Delta_k$ denotes the set of probability vectors  of length $k$.
We observe a sample consisting of independent draws from $p^\star$.
For $i \in [k]:=\{1, \dots, k\}$, we denote the number of times symbol $i$ is observed by $N_i$.
Following prior theoretical work on the probability estimation problem~\cite{AchJafOrl13,orlitsky2015competitive}, we adopt the Poisson sampling model where the sample size is a Poisson random variable with mean $n$, denoted by $\poi(n)$.
This choice simplifies the mathematical analysis and does not materially affect our practical conclusions, since when $n$ is large, a $\poi(n)$ random variable concentrates tightly around $n$.
With this choice,  the observed counts 
$N=(N_1,\ldots,N_k)$ are independent and distributed as $N_i \sim \poi(n p_i^\star)$.

We evaluate the performance of a candidate estimator $\widehat p \equiv \widehat{p}(N)$ of $p^\star$ by the Kullback--Leibler (KL) divergence~\cite{kullback1951information}, a commonly used loss function for probability distributions:
\begin{equation*}
	\DKL(p^\star \| \widehat{p}) := \sum_{i=1}^k p^\star_i \log \frac{p^\star_i}{\widehat p_i}\,.
\end{equation*}
The risk of an estimator $\widehat p$ based on a sample of size $\poi(n)$ drawn from $p^\star$ is given by
\begin{equation*}
	r_n(p^\star, \widehat p) = \E_{p^\star} \DKL(p^\star \| \widehat{p}).
\end{equation*}

The classical statistical approach to distribution estimation is to seek an estimator $\widehat p$ that minimizes 
its worst-case risk $\sup_{p^\star\in\Delta_k} r_n(p^\star, \widehat p)$. This minimax risk is well-known
to be on the order of $\frac{k-1}{n} \wedge \log k$, which any add-constant estimator can achieve. 
However, such a worst-case criterion is often too pessimistic. In practice (and as we will show in this paper), 
 data-driven estimators that adapt to the latent structure in the data can significantly outperform minimax strategies. As such, it is desirable to adopt an instance-dependent notion of optimality that goes beyond worst-case analysis.
 




To this end, we adopt the competitive framework pioneered by Orlitsky and Suresh~\cite{orlitsky2015competitive} and dating back to the compound decision theory framework of Robbins \cite{robbins1951asymptotically} by measuring the performance of our estimator against that of a \textit{permutation-invariant} (PI) oracle estimator. 
Let $\pi$ be a permutation on $[k]$. For any length-$k$ vector $y=(y_1,\ldots,y_k)$, we denote its permuted version by $\pi(y) := (y_{\pi(1)}, \ldots,y_{\pi(k)})$.
An estimator $\widehat{p}=\widehat{p}(N)$ is called permutation invariant if 
\begin{align*}
\widehat{p}(\pi(N)) = \pi(\widehat{p}(N))
\end{align*}
for any $N=(N_1,\dots,N_k)$ and any permutation $\pi$; in other words, 
the estimated probabilities only depend on the counts of each symbol rather than how they are labeled.
For a given $p^\star$, the \textit{PI oracle} $\widehat{p}^{\PI}$ is defined to be the PI estimator that minimizes the average KL risk $r_n(p^\star, \widehat p)$.
Note that $\widehat{p}^{\PI}$ is an ``oracle'' estimator because it is allowed to depend on the ground truth $p^\star$ provided it satisfies the symmetry requirement of permutation invariance. Equivalently, the PI oracle is the best estimator among all possible estimators (not necessarily permutation-invariant) which knows the values of $p^\star$ \emph{up to permutation} \cite{orlitsky2015competitive}. Furthermore, \cite{greenshtein2009asymptotic} shows that the PI oracle coincides with the Bayes estimator under a Bayesian setting where the symbols are randomly relabeled; we defer the details to \eqbr{eq:PI_oracle_intro}  and \prettyref{eq:PI-oracle-full} for the exact expression, and \ifthenelse{\boolean{arxiv}}{Appendix \ref{sec:PI_equivalence}}{SI Appendix, Sec. \ref{sec:PI_equivalence}} for detailed justifications.



Finally, for any estimator $\widehat{p}$,
define its \textit{regret} as the worst-case excess risk over the PI oracle risk:
\begin{align}\label{def:intro_regret}
	\reg(\widehat{p}) := \sup_{p^\star \in \Delta_k} [r_n(p^\star, \widehat p) - r_n(p^\star, \widehat p^{\PI})].
\end{align}
Our objective is to find an estimator $\widehat{p}$ with the smallest possible regret. While it is possible to adopt other benchmarks, in this work we focus on competing with the PI oracle, which is also the classical notion put forward by Robbins in the framework of compound decision theory \cite{robbins1951asymptotically,zhang2003compound,greenshtein2009asymptotic}.

\ifthenelse{\boolean{arxiv}}{\subsection{NPMLE-based estimator}}{\section*{NPMLE-based estimator}}
The starting point of our estimation procedure is the empirical Bayes approach~\cite{Robbins56}, which relies on the following thought experiment.
Write $\theta^\star_i = n p^\star_i$, so that $N_i \sim \poi(\theta_i^\star)$ independently for each $i$.
Suppose that, instead of being fixed, the coordinates of the vector $\theta^\star$ had been drawn i.i.d.\ from some prior distribution $G$ on $\R_+$.
Then the marginal distribution of each count $N_i$ would be a Poisson mixture with probability mass function $f_G$ on the nonnegative integers given by
\begin{equation*}
	f_G(y) = \int \frac{\theta^y e^{-\theta}}{y!} G(d\theta), \quad y\geq 0\,.
\end{equation*}
The Bayes estimator minimizing the expected squared error is the posterior mean, given by
\begin{equation}
	 \theta_G(y) := \E_G[\theta_i \mid N_i = y] = (y+1) \frac{f_G(y+1)}{f_G(y)},
\label{eq:bayes}
\end{equation}
where $\E_G$ indicates that the expectation is taken with respect to the prior $G$.

Following Robbins, we define an estimator by acting as though the above thought experiment were actually valid and estimating the fictitious prior $G$ from the data directly.
We propose to estimate $G$ via maximum likelihood:
\begin{align}\label{def:npmle}
	\widehat{G} := \argmax_G \sum_{i=1}^k \log f_G(N_i),
\end{align}
where the maximization is taken over all probability distributions on $\R_+$.
Since this space is infinite dimensional, $\widehat G$ is a \textit{non-parametric} maximum likelihood estimator~\cite{kiefer1956consistency}.
The theoretical properties of the NPMLE in the Poisson model are by now well understood. Notably, it is known that the convex optimization problem \eqbr{def:npmle} has a unique solution $\widehat{G}$ which is a discrete measure with at most $k$ atoms. Furthermore, with high probability its support size is at most $O(\sqrt{n})$,  independent of  the domain size $k$ (see \ifthenelse{\boolean{arxiv}}{Appendix \ref{sec:NPMLE-basics}}{SI Appendix, \prettyref{sec:NPMLE-basics}}).
Computationally, $\widehat G$ can be found using various efficient algorithms \cite{simar1976maximum,lindsay1995mixture,koenker2014convex}.
(See \ifthenelse{\boolean{arxiv}}{Appendix \ref{sec:NPMLE-computation}}{SI Appendix, \prettyref{sec:NPMLE-computation} for an overview}.)

Combining these two ideas gives rise to a fully data-driven estimation procedure: Having observed counts $N_1, \dots, N_k$, compute the NPMLE $\widehat G$ as in~\eqbr{def:npmle}, and 
apply the learned Bayes rule 
$\theta_{\widehat G}$ in \eqbr{eq:bayes} to estimate the probability vector $p^\star$ by 
\begin{align}\label{def:npmle_est}
	\pNPMLE_i \propto \theta_{\widehat G}(N_i) + \hp \indc{N_i = 0}
\end{align}
where $\hp > 0$ is a regularization parameter.
(Here and below, the notation $\propto$ means that the left-hand side equals the right-hand side scaled by a normalizing constant so that $\pNPMLE$ is a valid probability vector.)
We remark that the regularization parameter $\tau$ in \eqref{def:npmle_est} 
\textit{only} applies to unseen symbols, which require special treatment. 
As we shall see below, the choice of $\tau$ is rather flexible for achieving the optimal competitive regret ($\tau = k^{-C_0}$ for any $C_0\ge 1$ suffices) and has little effect on empirical performance. On the other hand, choosing a non-zero $\tau$ is also essential, for otherwise the estimated probability for unseen symbols can be exponentially small, leading to a significant regret (see \ifthenelse{\boolean{arxiv}}{Appendix \ref{sec:notau}}{SI Appendix, \prettyref{sec:notau}}).


Our main result is a competitive optimality guarantee for the proposed NPMLE estimator \eqbr{def:npmle_est}. The following theorem bounds its regret over the PI oracle.
\begin{theorem}\label{thm:main}
	With $\hp = k^{-C_0}$ for any $C_0\ge 1$, there exists an absolute constant $C > 0$ such that 
	\begin{align*}
		\reg(\pNPMLE) \le C\pth{\pth{n^{-2/3}\wedge \frac{k}{n}}\log^{14}(nk) }. 
	\end{align*}
\end{theorem}

A matching lower bound is given by \cite[Theorem 3]{orlitsky2015competitive}, which shows that the regret of \textit{any} estimator $\widehat p$ must satisfy
\[
\reg(\widehat{p}) \ge c\pth{\pth{n^{-2/3}\wedge \frac{k}{n}}\log^{-c'}(n) }
\]
for some positive constants $c$ and $c'$. 
Theorem~\ref{thm:main} therefore establishes the \textit{competitive optimality} of the NPMLE up to logarithmic factors.


In contrast, the Good--Turing estimator fails to achieve the optimal regret. As mentioned in the Introduction, the original Good--Turing estimator, defined by
\begin{equation}
\widehat{p}_{i}^{\mathrm{GT}} \propto
(N_i+1) \frac{\Phi_{N_i+1}}{\Phi_{N_i}},
    \label{eq:GT-original}
\end{equation}
requires additional modification (such as smoothing) to be useful; otherwise its KL risk is not even finite because the estimated probability can be zero. Here $\Phi_y := \sum_{i=1}^k \indc{N_i = y}$, known as the \textit{profile} \cite{PML,acharya2017unified}, is the number of symbols that appear exactly $y$ times. 
To prove a stronger negative result, we therefore allow the estimator the extra freedom to switch between Good--Turing and empirical frequency, which, as mentioned in the Introduction, is commonly used in practice. Specifically, we consider the following estimator analyzed by Orlitsky and Suresh~\cite{orlitsky2015competitive}, which uses Good--Turing for low counts and the empirical frequency for high counts by comparing with a threshold  
 $y_0$:
\begin{align}\label{eq:modified-Good--Turing}
    \widehat{p}_{i}^{\mathrm{MGT}} \propto
    \begin{cases}
		\frac{N_i+1}{n}\cdot \frac{\Phi_{N_i+1}+1}{\Phi_{N_i}} &\text{if } N_i \le y_0, \\
		\frac{N_i}{n} & \text{otherwise.}
	\end{cases}
\end{align}
Similar to the Laplace rule, \prettyref{eq:modified-Good--Turing} applies additive smoothing to the original Good-Turing \prettyref{eq:GT-original} by adding a constant $1$ to 
the numerator so that the estimated probabilities are never zero.


The following theorem (proved in \ifthenelse{\boolean{arxiv}}{Section \ref{sec:GT-proof}}{SI Appendix \prettyref{sec:GT-proof}}) shows that, regardless of the choice of the threshold $y_0$, the modified Good--Turing estimator\footnote{In comparison, \cite[Theorem 1]{orlitsky2015competitive} showed that the regret of the modified Good--Turing in \prettyref{eq:regGT-lb} is $\widetilde O(n^{-1/3})$ by choosing the threshold $y_0$ appropriately. Additionally, a much more involved modification of the Good--Turing estimator is shown to attain $\widetilde O(n^{-1/2})$
\cite[Theorem 2]{orlitsky2015competitive}.
Neither of these results achieves the optimal regret.
} 
fails to attain the optimal $\widetilde{O}(n^{-2/3})$ regret for large $k$.\footnote{Throughout the paper we use $\widetilde{O}(\cdot)$ and $\widetilde{\Omega}(\cdot)$ to denote upper and lower bound in order up to polynomial factors in $\log(nk)$.}
\begin{theorem}\label{thm:Good--Turing-LB}
	For any $C > 0$, there exists an absolute constant $c > 0$ such that for any $k \in [\sqrt n, n^C]$ and any $y_0\in\naturals$,
	\begin{align}
        \reg\pth{\widehat{p}^{\mathrm{MGT}}} \ge \frac{c}{\sqrt{n\log n}}. 
        \label{eq:regGT-lb}
	\end{align}
\end{theorem}

Together, the above two theorems establish a strong separation between the Good--Turing estimator and the NPMLE.
Despite its popularity in practice, the Good--Turing estimator cannot match the performance of the NPMLE.
Moreover, our numerical results show that this difference is not merely a theoretical result.
The NPMLE consistently and significantly outperforms the Good--Turing estimator in both simulations and real datasets including examples from corpus linguistics and census data.

\ifthenelse{\boolean{arxiv}}{\subsection{Experiments}\label{subsec:experiment}}{\section*{Experiments}}
In this section we demonstrate the efficacy of the NPMLE-based distribution estimator in a variety of synthetic and real datasets by comparing with existing approaches and oracle estimators.


To compute the NPMLE, we use the Frank--Wolfe algorithm \cite{frank1956algorithm,jaggi2013revisiting} to solve the convex optimization problem in \eqbr{def:npmle},   also known as the vertex direction method  in the literature on mixture models  \cite{simar1976maximum,lindsay1995mixture,jana2022poisson} -- see \ifthenelse{\boolean{arxiv}}{Appendix \ref{sec:NPMLE-computation}}{SI Appendix \prettyref{sec:NPMLE-computation}}. 
We also refrain from using the regularization parameter $\tau$ in \eqref{def:npmle_est} for  unseen symbols, so that the estimator is  \textit{completely tuning parameter-free}. While theoretically this regularization is necessary for attaining the optimal regret (see \ifthenelse{\boolean{arxiv}}{Appendix \ref{sec:notau}}{SI Appendix, \prettyref{sec:notau}}), empirically it has little effect on the performance.

We compare the NPMLE with the Laplace (add-$1$), Krichevsky--Trofimov (add-$\frac{1}{2}$), 
and Good--Turing estimators. Specifically, we consider the version of modified Good--Turing  used by \cite{orlitsky2015competitive} in their experiment, which replaces the pre-specified threshold $y_0$ by the profile $\Phi_{N_i+1}$:
\begin{align}\label{eq:MGT-exp}
	\widehat{p}_{i} \propto
    \begin{cases}
		\frac{N_i+1}{n}\cdot \frac{\Phi_{N_i+1}+1}{\Phi_{N_i}} &\text{if } N_i \le \Phi_{N_i+1}, \\
		\frac{N_i}{n} & \text{otherwise}. 
	\end{cases}
\end{align}
We also compare with two oracle estimators: the separable oracle (the normalized version of 
\prettyref{eq:separable_oracle_intro} below), and an even stronger oracle called the natural oracle introduced by \cite{orlitsky2015competitive}. We omit the PI oracle, the computation of which requires averaging over all permutations of symbols and is therefore intractable. Our theoretical analysis (see 
\prettyref{eq:reg_decomp_main} below) shows that the regret between the separable oracle and the PI oracle is uniformly bounded by the optimal rate $\widetilde O(n^{-2/3})$. As such we use the separable oracle as a proxy for the PI oracle in the experiment.

We adopt the Poisson sampling model. (The empirical performance is almost identical to that for fixed sample size.) All experiments are based on 200 independent trials, except for the 2010 Census surname experiment, for which we run 20 trials due to the sheer size of the dataset.

\paragraph{Synthetic data.}
We fix the domain size $k=10000$ and vary the sample size $n$ from $1000$ to $50000$, averaging over $200$ independent trials. For distributions drawn from a prior (such as Dirichlet), the true distribution is redrawn in each trial.

\begin{figure}
\centering
\centering
     \begin{subfigure}[b]{0.6\columnwidth}
         \includegraphics[width=\textwidth]{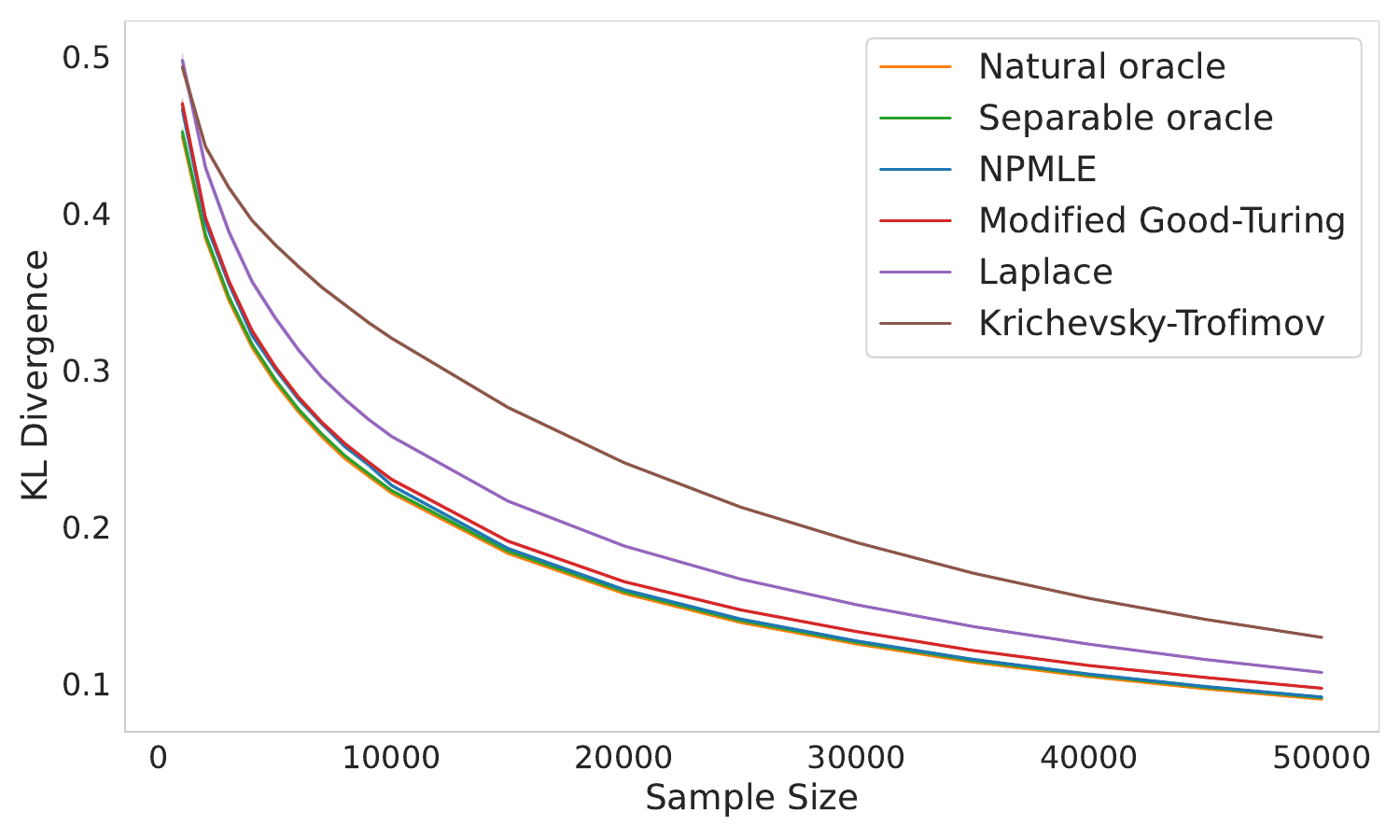}
         \caption{KL risk.}
         \label{fig:sqrtcauchy-risk}
     \end{subfigure}
     \begin{subfigure}[b]{0.6\columnwidth}
         \includegraphics[width=\textwidth]{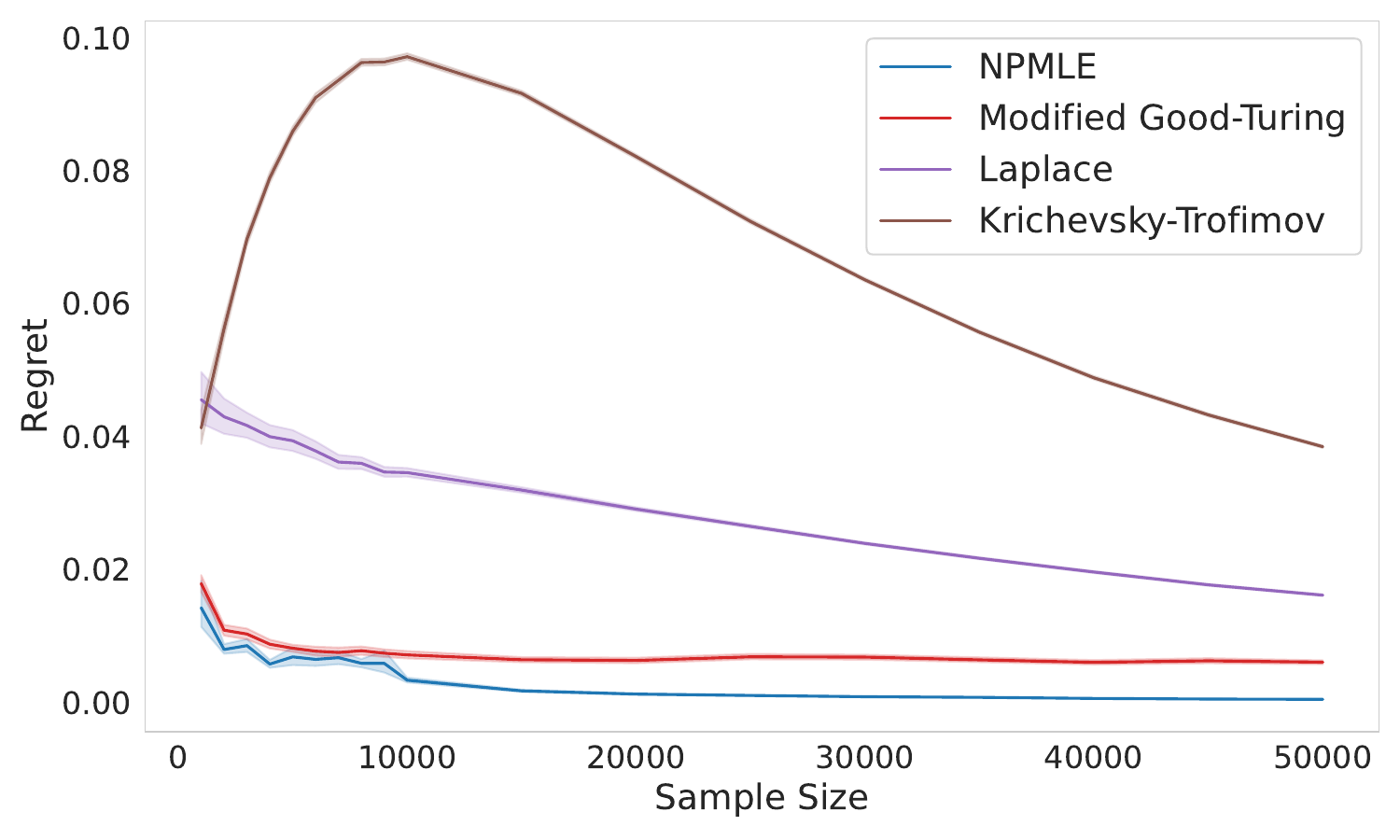}
         \caption{Regret over the separable oracle.}
         \label{fig:sqrtcauchy-regret}
     \end{subfigure}
\caption{KL risk and regret of various estimators and oracles for heavy-tailed distribution $p^\star$ (sqrt-Cauchy).}
\label{fig:sqrtcauchy}
\end{figure}

In \prettyref{fig:sqrtcauchy}, we plot the KL risks of various estimators and oracles as well as their regrets with respect to the separable oracle. The true distribution $p^\star$, referred to as sqrt-Cauchy, is randomly generated with $p^\star_i \propto \sqrt{|z_i|}$ and $z_i$ being iid Cauchy. This is a challenging setting because the heavy tail of the Cauchy distribution leads to an abundance of symbols that are relatively rare. As evident from \prettyref{fig:sqrtcauchy}, the NPMLE attains near-oracle performance, with lower regret than Good--Turing; both methods significantly outperform the Laplace or K--T estimators.
For clarity in subsequent plots, we focus on regret over the separable oracle and omit Laplace and K--T, which generally perform far worse. Full plots of the KL risks are provided in \ifthenelse{\boolean{arxiv}}{Appendix \ref{sec:all-figs}}{SI Appendix \prettyref{sec:all-figs}}.

\begin{figure}[t]
     \centering
     \begin{subfigure}[b]{0.38\columnwidth}
         \centering
         \includegraphics[width=\textwidth]{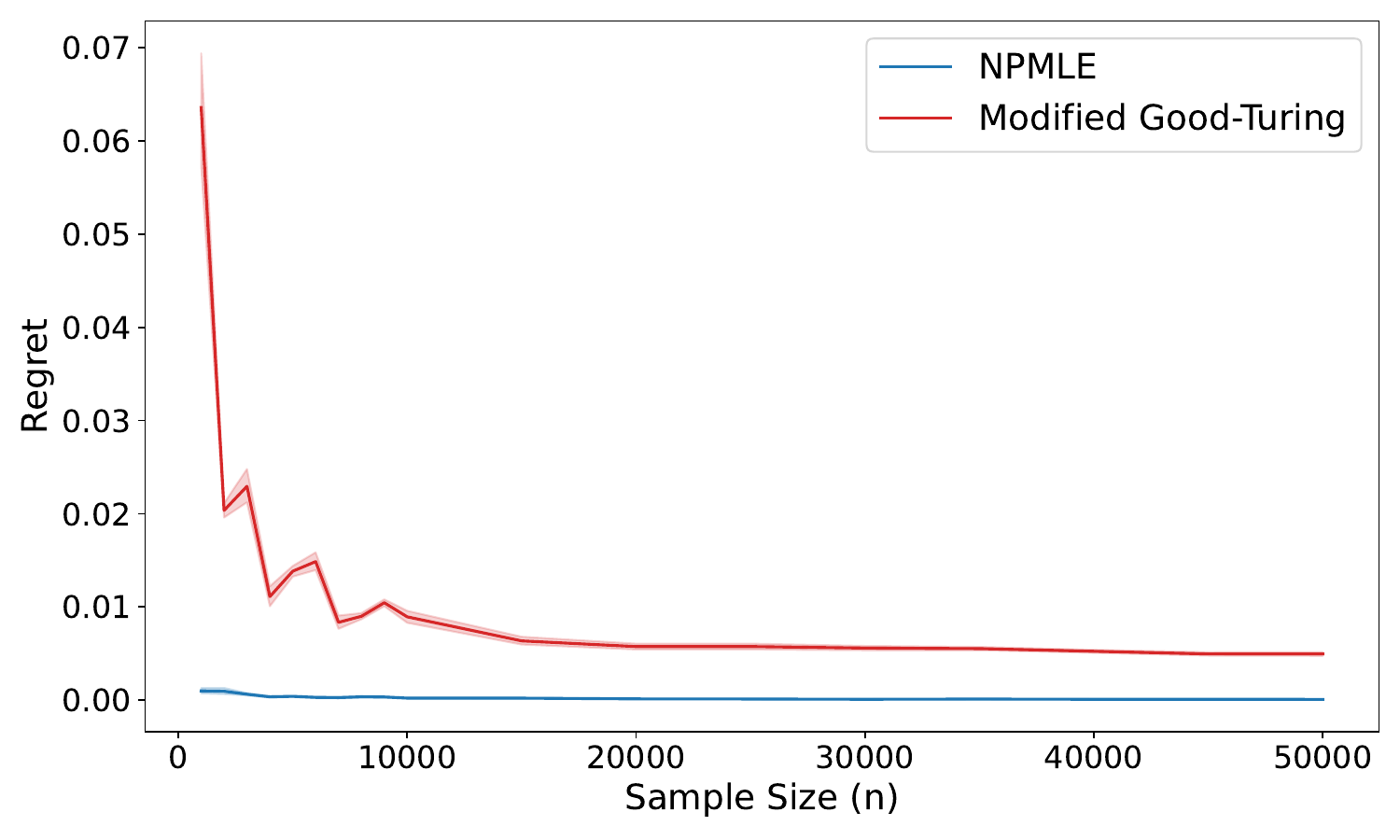}
         \caption{Uniform}
         \label{fig:uniform}
     \end{subfigure}
     \begin{subfigure}[b]{0.38\columnwidth}
         \centering
         \includegraphics[width=\textwidth]{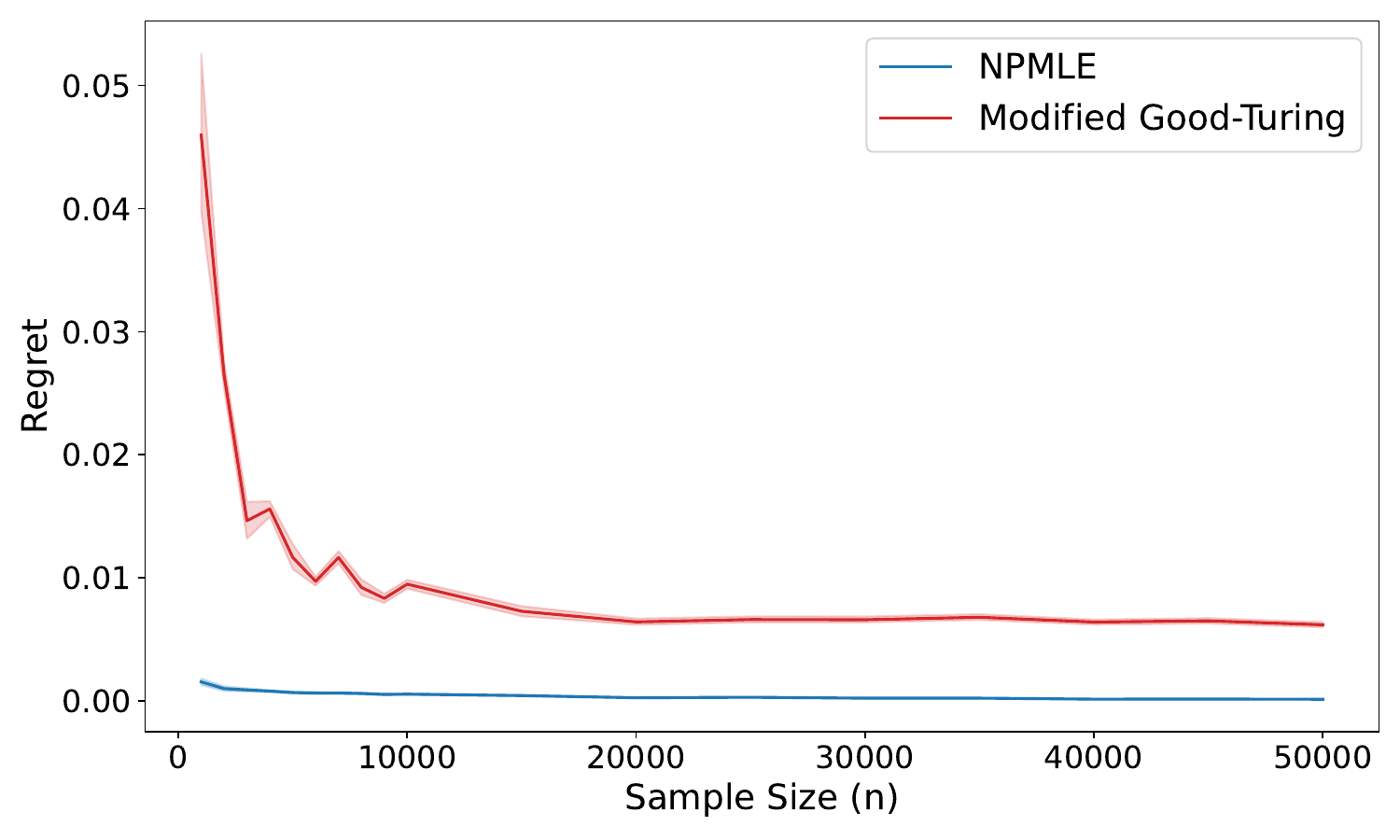}
         \caption{Step.}
         \label{fig:step}
     \end{subfigure}
     \begin{subfigure}[b]{0.38\columnwidth}
         \centering
         \includegraphics[width=\textwidth]{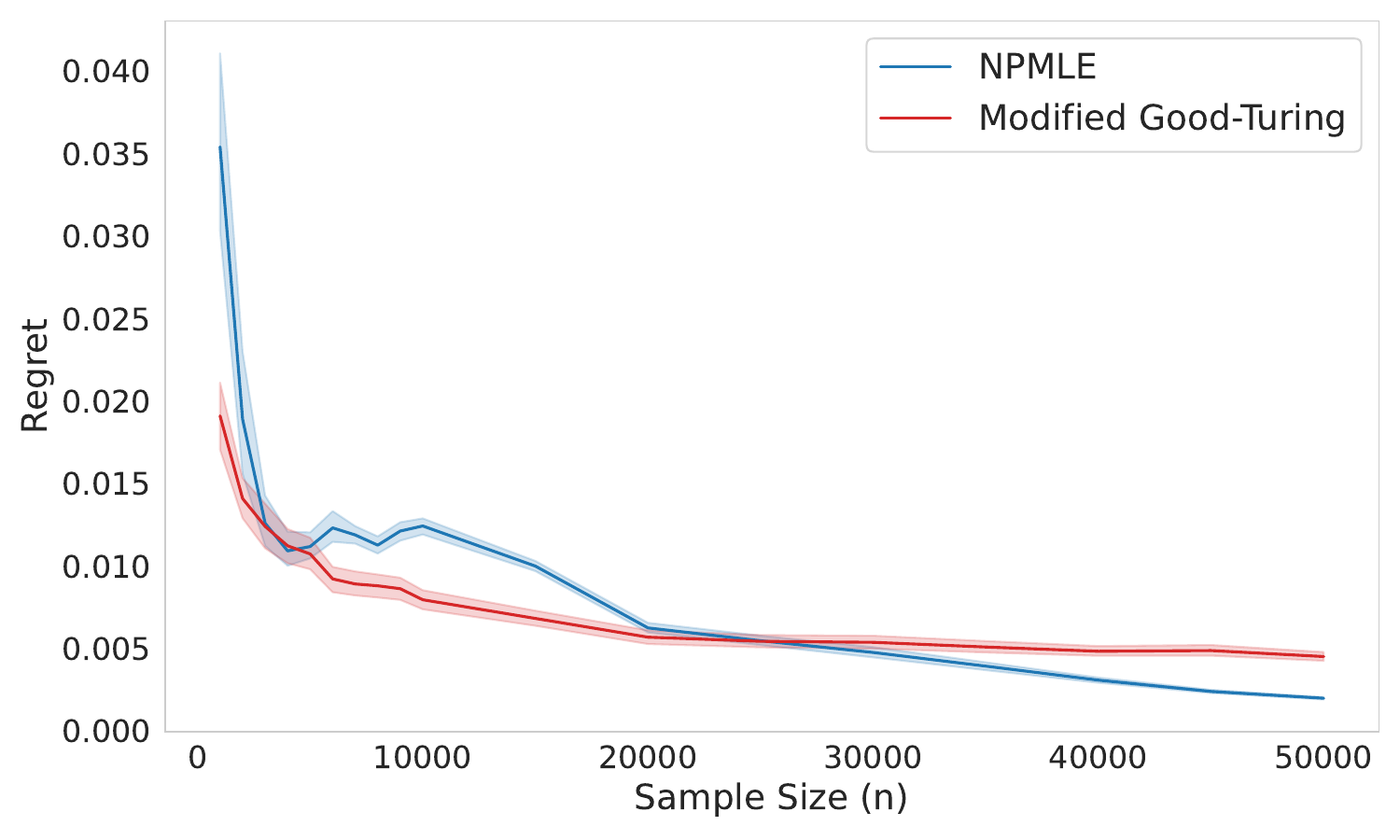}
         \caption{Zipf ($\alpha=1$).}
         \label{fig:zipf1}
     \end{subfigure}
     \begin{subfigure}[b]{0.38\columnwidth}
         \centering
         \includegraphics[width=\textwidth]{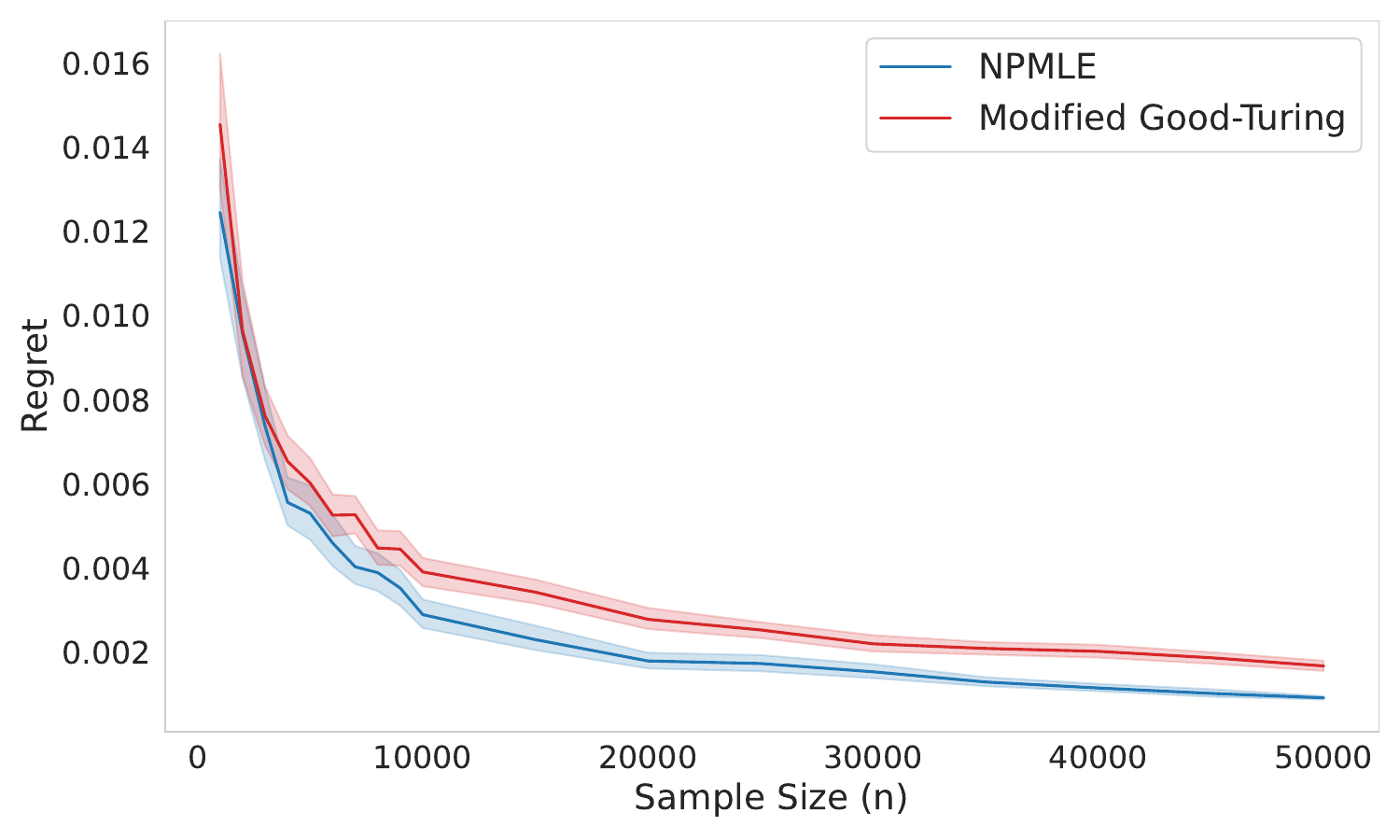}
         \caption{Zipf ($\alpha=1.5$).}
         \label{fig:zipf15}
     \end{subfigure}
     \begin{subfigure}[b]{0.38\columnwidth}
         \centering
         \includegraphics[width=\textwidth]{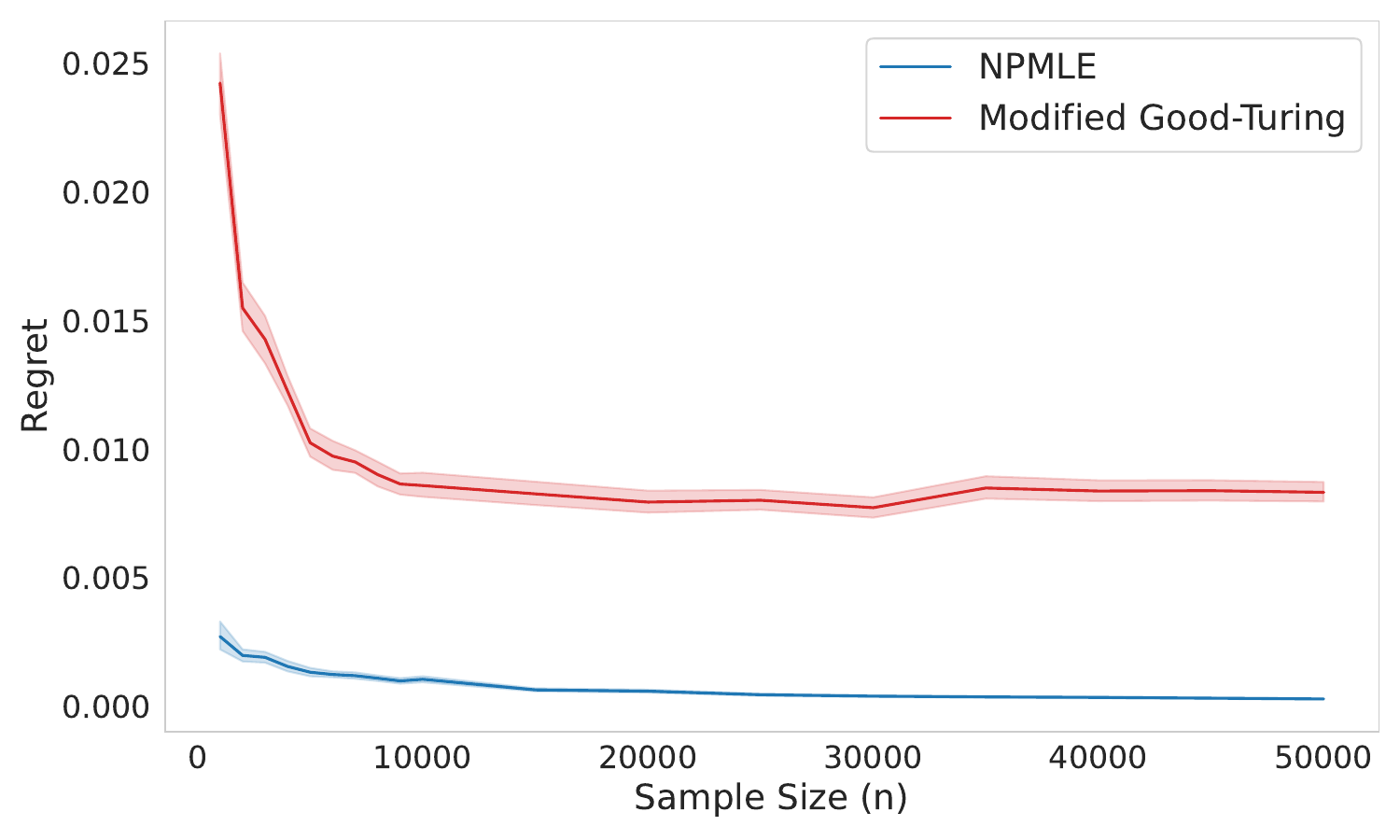}
         \caption{Dirichlet $(c=1)$}
         \label{fig:dirichlet1}
     \end{subfigure}
     \begin{subfigure}[b]{0.38\columnwidth}
         \centering
         \includegraphics[width=\textwidth]{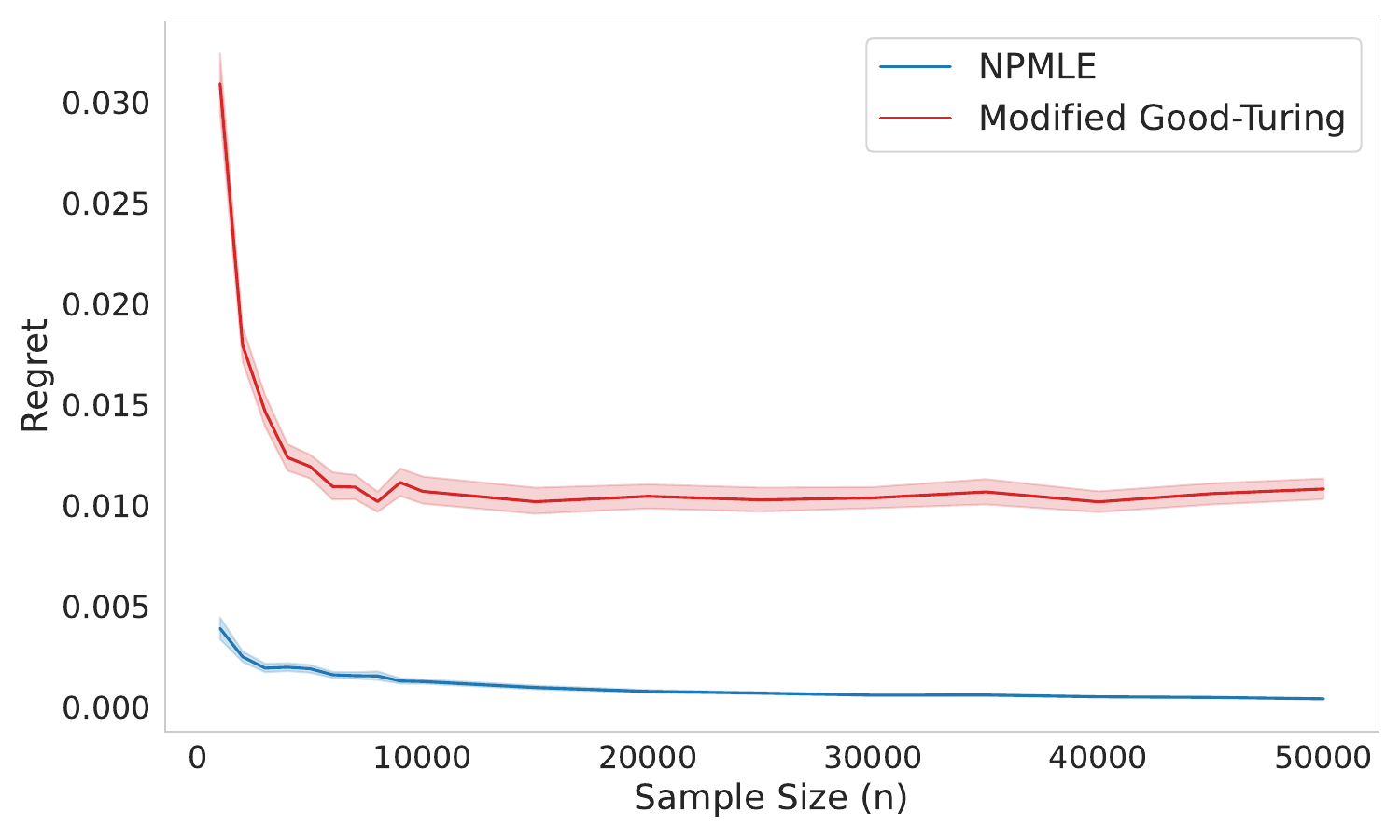}
         \caption{Dirichlet $(c=0.5)$}
         \label{fig:dirichlet05}
     \end{subfigure}
     
        \caption{KL regret over the separable oracle for various distributions over $k=10000$ elements.}
        \label{fig:KLregrets}
\end{figure}

\prettyref{fig:KLregrets} presents the regret of the NPMLE and Good--Turing estimators under several benchmark distributions used in the prior work \cite[Fig.~2]{orlitsky2015competitive}. These include the uniform distribution, step distribution ($\frac{1}{2k}$ for the first $\frac{k}{2}$ symbols and $\frac{3}{2k}$ for the remaining half), Zipf distribution with parameter $\alpha$ (given by the power law $p^\star_i \propto i^{-\alpha}, i=1,\ldots,k$), and Dirichlet prior with parameters $c$, with $c=1$ corresponding to the uniform prior over the probability simplex.
For uniform, step, and Dirchlet, the NPMLE consistently outperforms
Good--Turing. For the heavy-tailed Zipf distributions, the regret of Good--Turing starts off better but is eventually surpassed by the NPMLE as the sample size increases.

\begin{figure}[t]
     \centering
     \begin{subfigure}[b]{0.4\columnwidth}
         \centering
         \includegraphics[width=\textwidth]{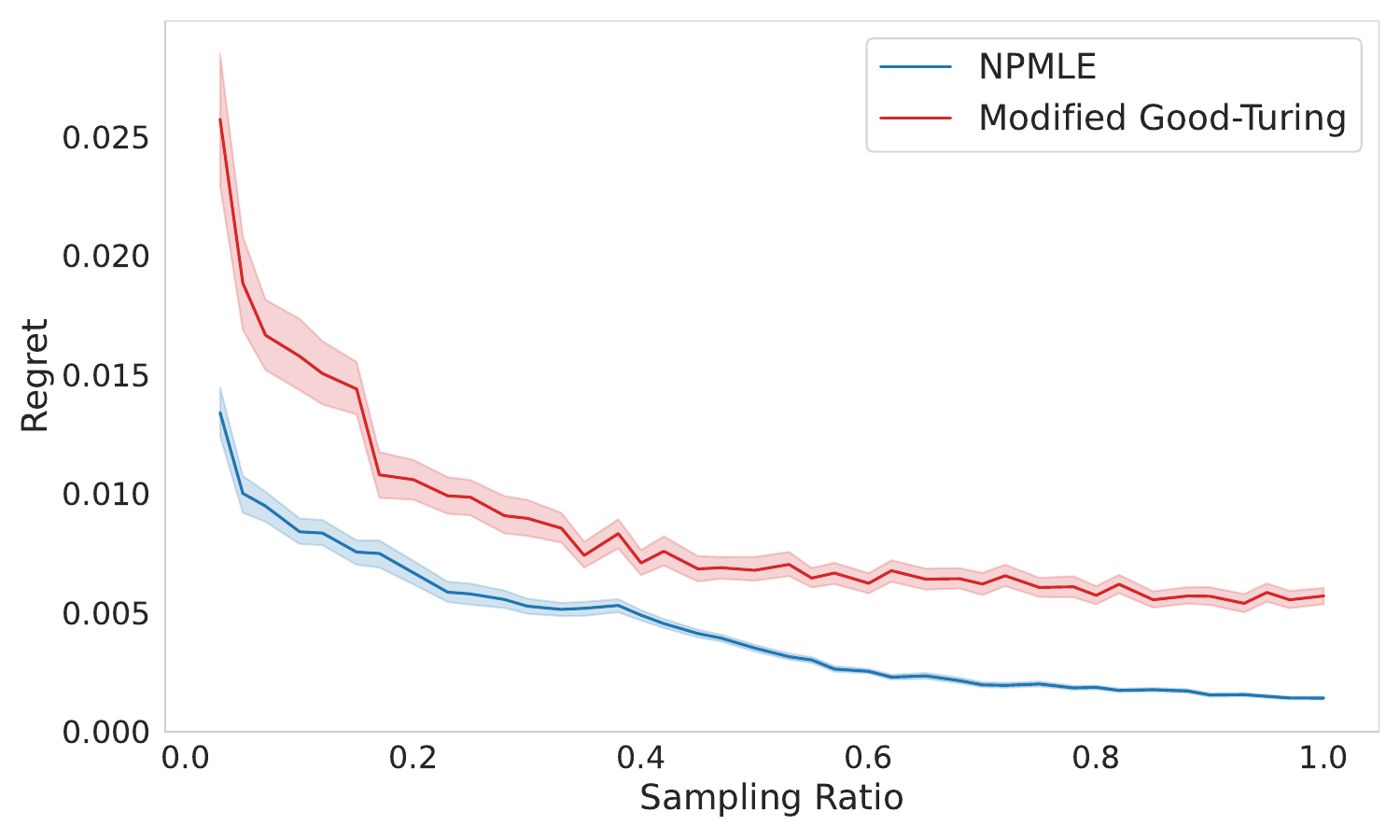}
         \caption{Hamlet (random).}       \label{fig:hamlet-random}
     \end{subfigure}
     \begin{subfigure}[b]{0.4\columnwidth}
         \centering
         \includegraphics[width=\textwidth]{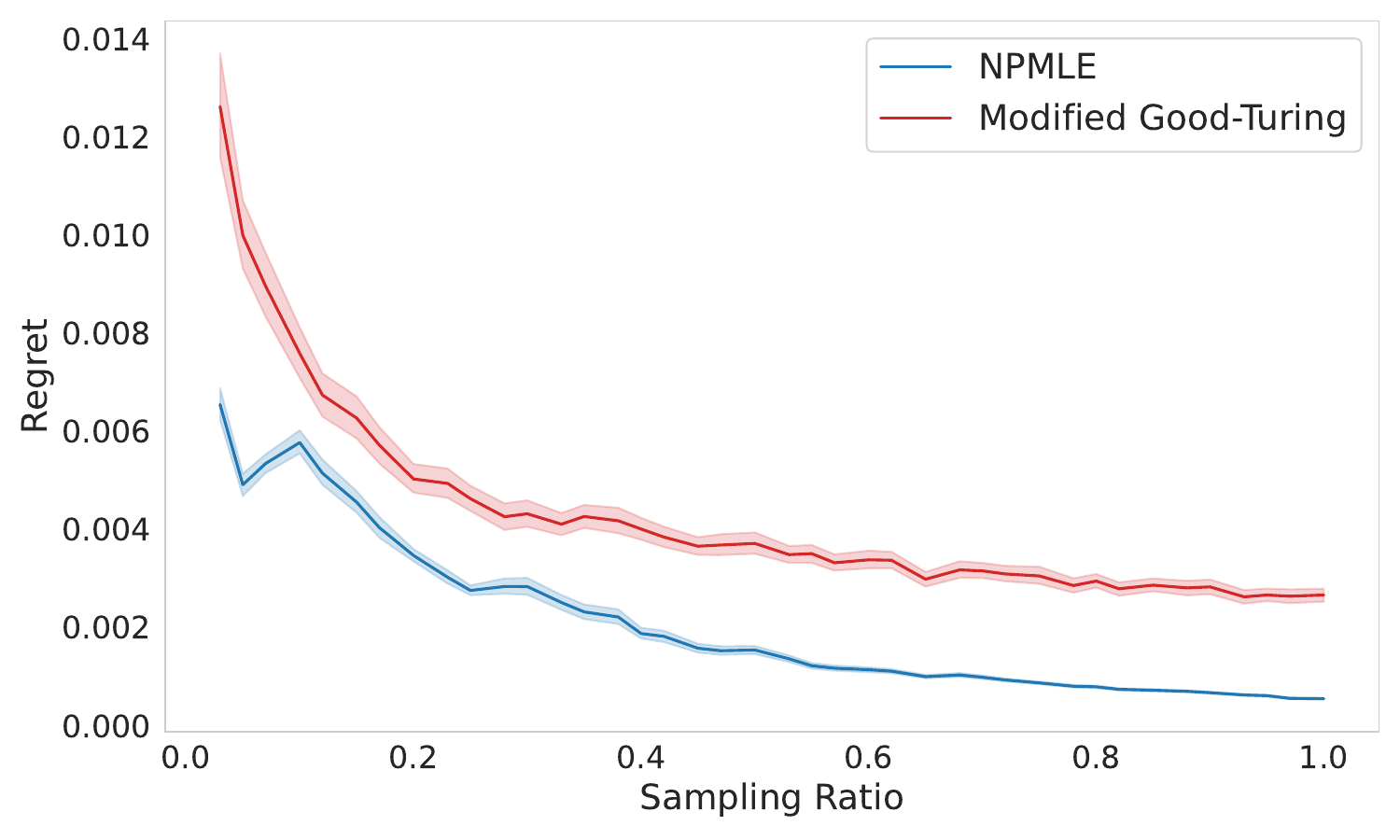}
        \caption{LOTR (random).}       \label{fig:lotr-random}
     \end{subfigure}     
     \begin{subfigure}[b]{0.4\columnwidth}
         \centering
         \includegraphics[width=\textwidth]{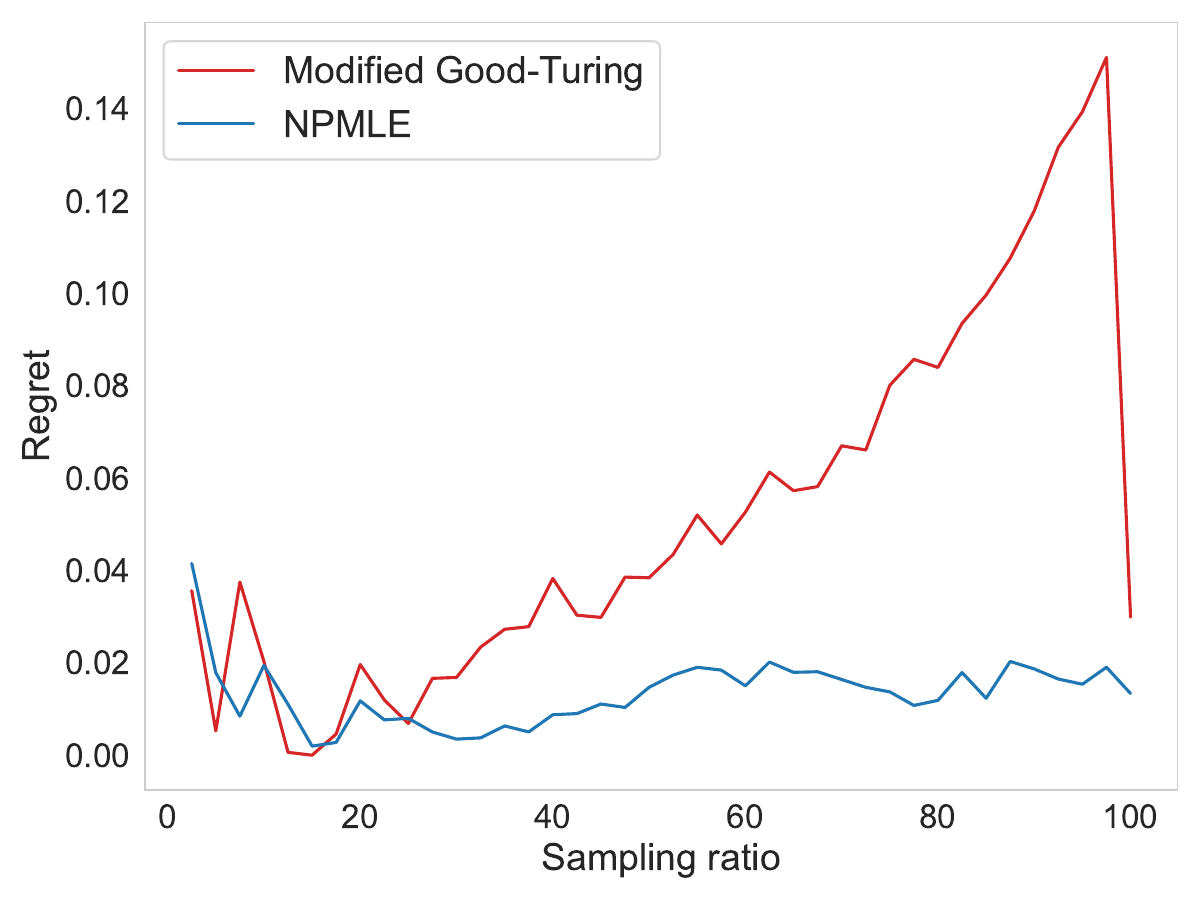}
         \caption{Hamlet (consecutive).}       \label{fig:hamlet-consec}
     \end{subfigure}
     \begin{subfigure}[b]{0.4\columnwidth}
         \centering
         \includegraphics[width=\textwidth]{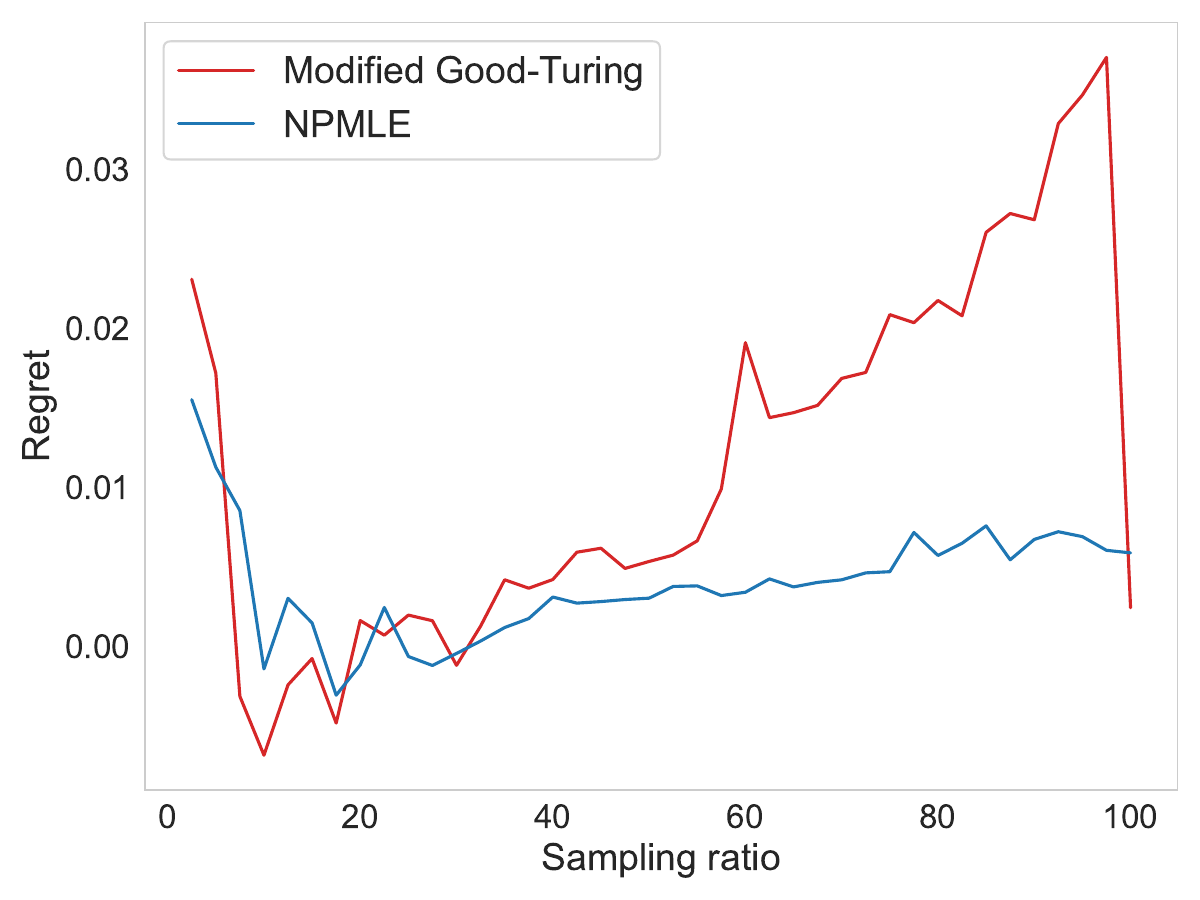}
         \caption{LOTR (consecutive).}       \label{fig:lotr-consec}
     \end{subfigure}     
     \begin{subfigure}[b]{0.4\columnwidth}
         \centering
         \includegraphics[width=\textwidth]{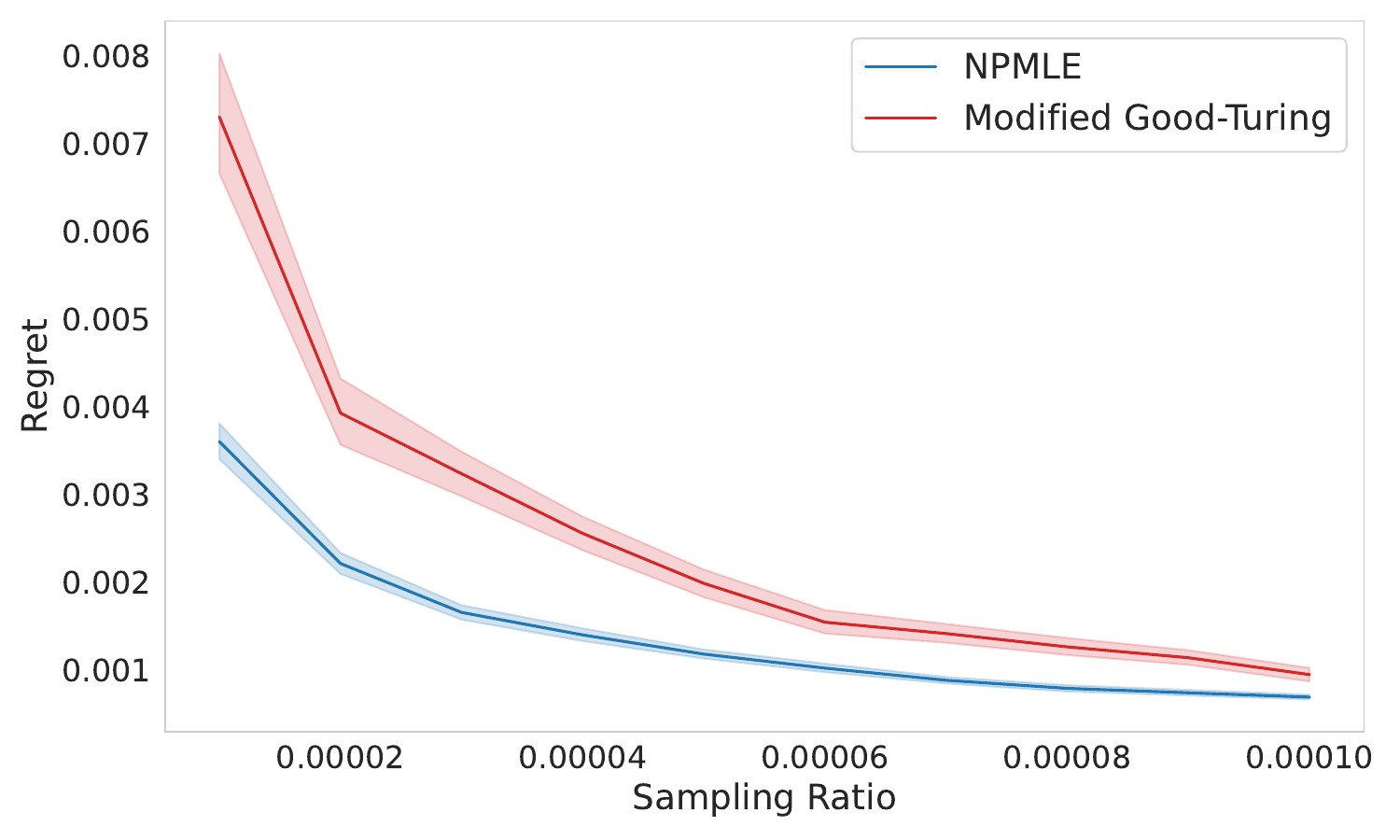}
         \caption{2020 Census Detailed DHC-A.}       \label{fig:census-group}
     \end{subfigure}
     \begin{subfigure}[b]{0.4\columnwidth}
         \centering         \includegraphics[width=\textwidth]{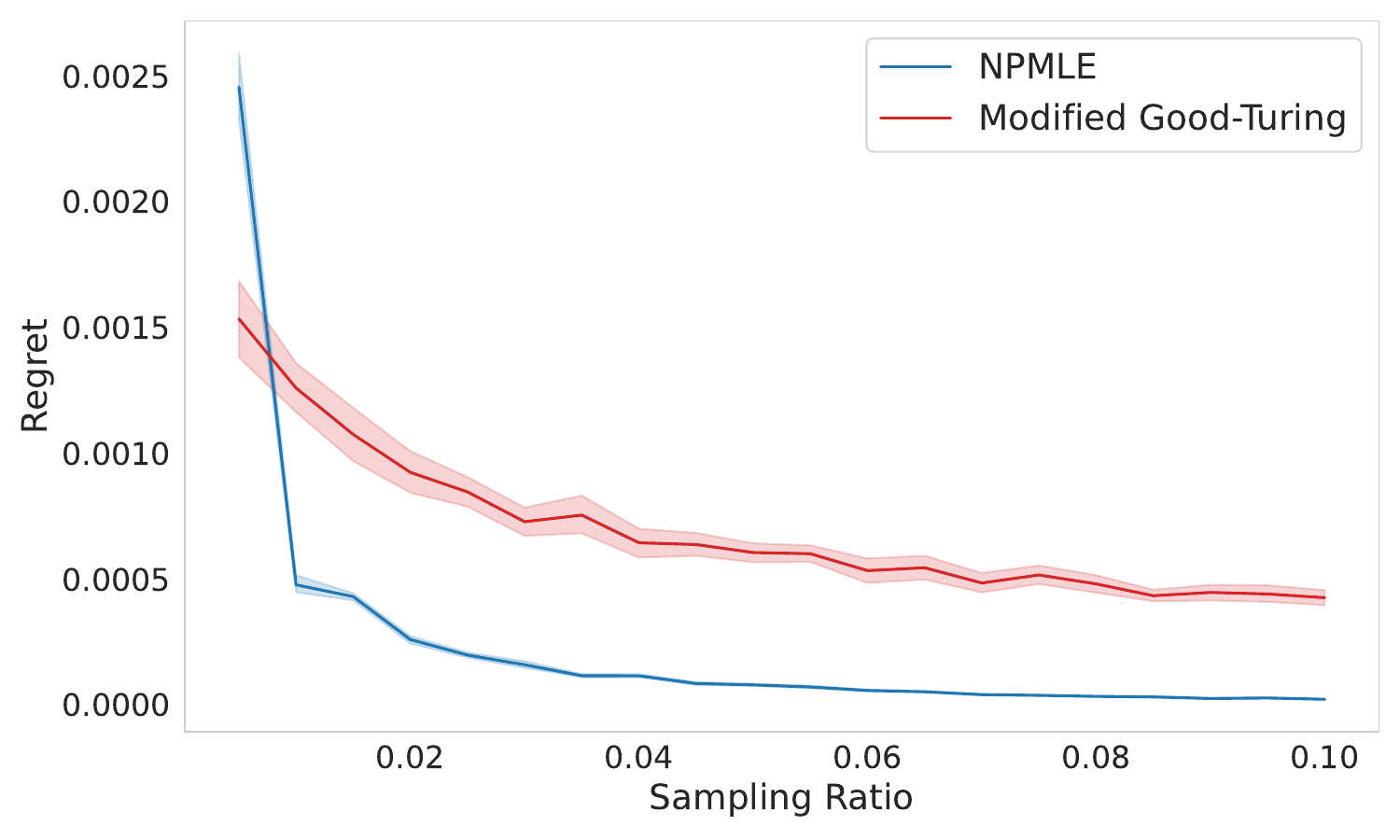}
        \caption{2010 Census surname.}       \label{fig:census-surname}
     \end{subfigure}
     \caption{KL regret of 
     NPMLE and modified Good-Turing on real data experiments.}
\end{figure}

\paragraph{Real data.}
A central task in corpus linguistics is estimating word frequencies from text.
To test the efficacy of the proposed methodology, we consider two corpora: 
\textit{Hamlet}  by William Shakespeare  (28,799 words in total and 4,804 distinct), and 
\textit{The Lord of the Rings: The Fellowship of the Ring} \cite{lotr} by J.~R.~R.~Tolkien (187,716 words in total and 10,552 distinct).
We randomly sample (with replacement) a fraction of the text and use it to estimate the word frequency in the entire corpus; Figs.~\ref{fig:hamlet-random}--\ref{fig:lotr-random} plot the KL regret of NPMLE and Good--Turing against the sampling ratio. Figs.~\ref{fig:hamlet-consec}--\ref{fig:lotr-consec} repeat the same experiment 
by taking consecutive (rather than randomly sampled) words 
from the beginning of the text as the input. 
In both sampling scenarios, the NPMLE outperforms Good--Turing over a wide range of sample ratios.\footnote{The steep decline of the regret for modified Good--Turing when the sampling ratio approaches 1 is an artifact of the form \eqbr{eq:MGT-exp}: 
When 
nearly 100\% of the text is consumed, \eqbr{eq:MGT-exp} is close to the empirical frequency which 
is defined as the ground truth in this experiment.}

Next, we evaluate the proposed methods on two large-scale U.S.~census datasets:
\begin{itemize}
    \item 2020 Census Detailed Demographic and Housing Characteristics File A (Detailed DHC-A) \cite{DHC}, 
    which, after pre-processing to remove overlapping categories (see \ifthenelse{\boolean{arxiv}}{Appendix \ref{sec:dataproc}}{SI Appendix \prettyref{sec:dataproc}}), contains the population size of 1,446 detailed race and ethnicity groups.
\item Frequently occurring surnames (appearing at least 100 times) in the 2010 census \cite{USCensus2010Surnames}.
    This represents the largest experiment carried out in this paper, with $162{,}254$ 
    distinct names and a total population size of $294$ million. For numerical stability and computational efficiency, we apply the NPMLE estimator \textit{conditionally} to counts below a certain threshold and combine it with the empirical frequency to produce  estimates of all surnames. (See \ifthenelse{\boolean{arxiv}}{Appendix \ref{sec:experiment-details-conditional}}{SI Appendix \prettyref{sec:experiment-details-conditional}} for details).
\end{itemize}
Figs.~\ref{fig:census-group} and \ref{fig:census-surname}
present the KL regret of estimating the proportion of population groups and surnames, respectively. 
In both experiments, the NPMLE proves superior to the Good--Turing estimator.

Finally, we conduct an \textit{out-of-sample} experiment in which a Bayes estimator is trained based on the \textit{Hamlet} and then applied to other Shakespearean plays to estimate word frequency; we call this strategy a \textit{pretrained Bayes estimator}.
Specifically, in the training stage, we apply the NPMLE estimator to the entire \textit{Hamlet} to obtain $\widehat G$. In the inference stage, we apply the Bayes estimator with prior $\widehat G$ to a sampled portion (20\%) of a new play, adjusting for sample size differences.

\prettyref{fig:shakespeare-oos} compares the performance of this pretrained Bayes estimator with the NPMLE and modified Good--Turing. The results are striking: Out of the 38 Shakespearean canonical plays other than the \textit{Hamlet}, the pretrained Bayes estimator outperforms the modified Good--Turing 32 times, and even surpasses the in-sample NPMLE 18 times.
One possible explanation is as follows: If we view the entire Shakespearean canon as being sampled from a common distribution, then the NPMLE trained on the full \textit{Hamlet} may better approximate the empirical distribution of the true probabilities than one learned on the fly from, say, 20\% of \textit{King Lear}. As a result, the pretrained estimator can be more competitive despite the domain mismatch. In contrast, the Bayes estimator trained on \textit{Hamlet} does not generalize to the \textit{Lord of the Rings} (see \ifthenelse{\boolean{arxiv}}{Appendix \ref{tab:oos}}{SI Appendix \prettyref{tab:oos}}), which suggests that the learned prior 
captures useful stylistic information (e.g.~the so-called \textit{vocabulary profile} \cite{stamatatos2009survey}) specific to a corpus.



\begin{figure}[t]
    \centering
    \ifthenelse{\boolean{arxiv}}{\includegraphics[width=0.8\linewidth]{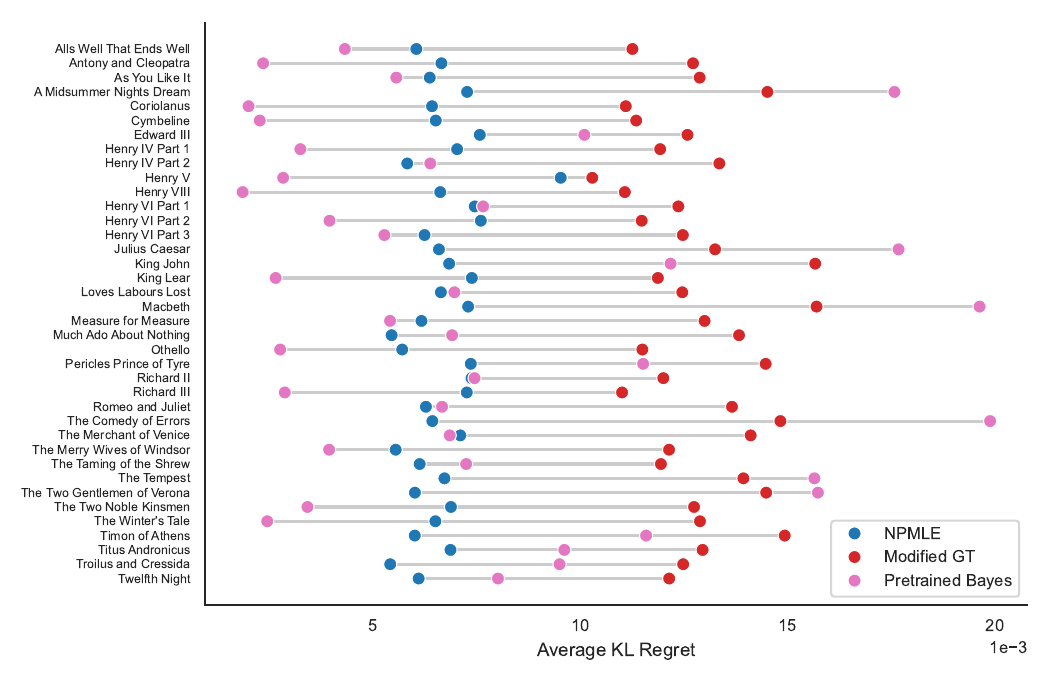}}{\includegraphics[width=1.0\linewidth]{figs_final/shakespeare_oos_regret.pdf}}
    \caption{Out of sample experiment on the Shakespearean canon (sampling ratio 20\%). The 
   NPMLE and the modified Good--Turing are the same as in \prettyref{fig:hamlet-random}. The pretrained Bayes estimator applies an estimated prior learned from \textit{Hamlet} only.}
    \label{fig:shakespeare-oos}
\end{figure}

\ifthenelse{\boolean{arxiv}}{\subsection{Discussion}}{\section*{Discussion}}
In this paper we proposed a computationally efficient  distribution estimator with both theoretical optimality and strong empirical performance. Compared with the Bayesian estimators of Laplace and Krichevsky--Trofimov using fixed priors, we take an empirical Bayesian approach with data-driven priors. Compared with Good--Turing which stems from the ``$f$-modeling'' approach to empirical Bayes \cite{efron2014two} focusing solely on the marginal mixture distribution, we adopt the ``$g$-modeling'' strategy by estimating the mixing distribution (prior) using the NPMLE.


The practical effectiveness of NPMLE owes much to its Bayesian structure, which yields estimates that are more stable, accurate, and interpretable than existing methodologies. By construction, Bayes (and hence empirical Bayes) estimators always assign strictly positive probabilities, eliminating the need for ad hoc adjustments such as additive smoothing. Furthermore, the NPMLE estimator has a desirable \textit{monotonicity} property: Symbols appearing more frequently are always assigned higher probabilities. This is inherited from the monotonicity of the Bayes estimator $y \mapsto \theta_G(y)$ for any prior $G$ \cite{koenker2014convex}. In contrast, as illustrated in  \prettyref{fig:smooth}, 
the modified Good--Turing estimator lacks this property and the unmodified version produces wildly oscillating estimates that are unusable in practice.

\begin{figure}[t]
\centering
\ifthenelse{\boolean{arxiv}}
{\includegraphics[width=0.8\linewidth]{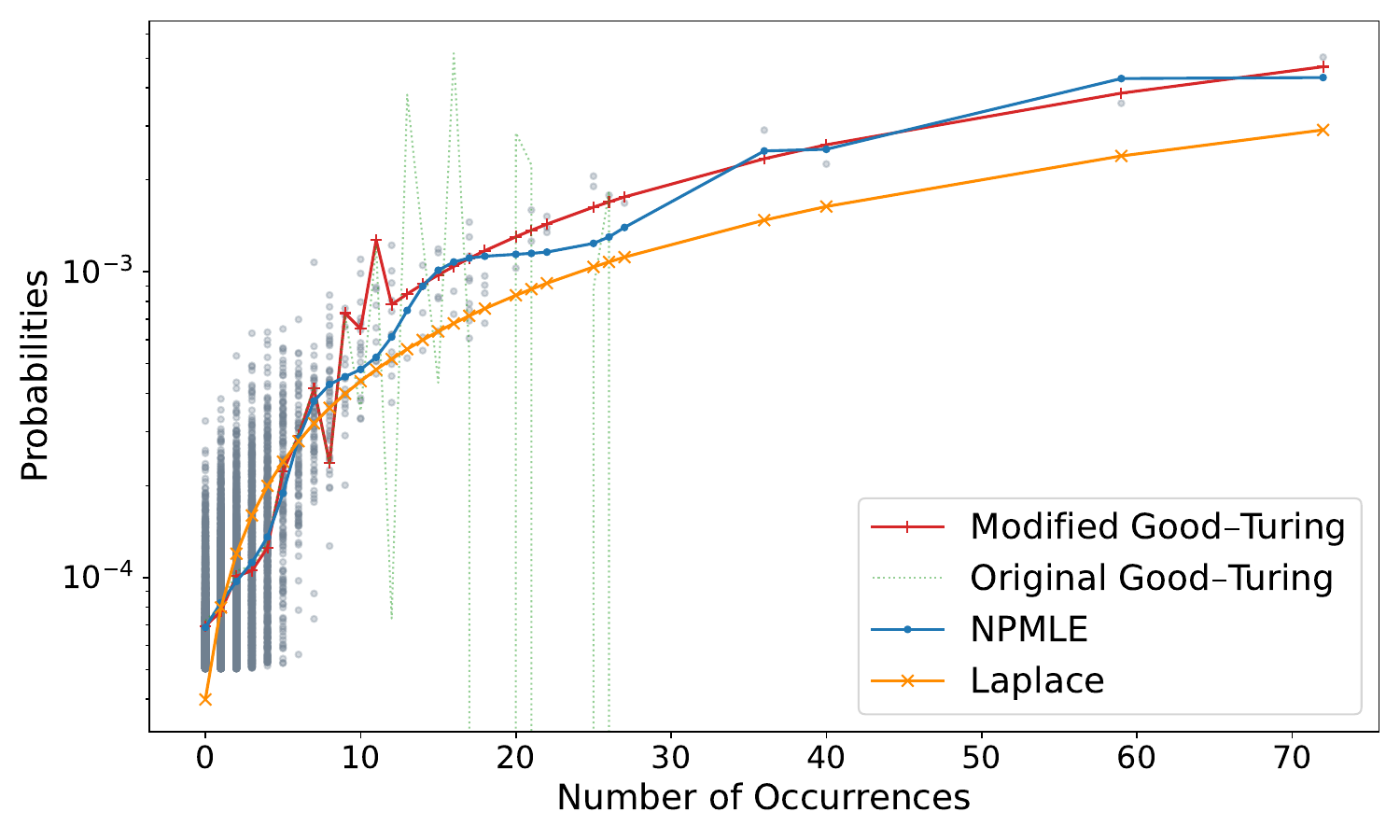}}
{\includegraphics[width=\linewidth]{figs_final/smoothing_zipf.pdf}}
\caption[\texorpdfstring{Smoothing effect of distribution estimators}{Smoothing effect of distribution estimators}]{Smoothing effect of distribution estimators. Data are drawn from the Zipf distribution with  $\alpha=1/2, k=10000$, and $n=15000$ in a single trial. (See \ifthenelse{\boolean{arxiv}}{Figure \ref{fig:smooth-full} in the Appendix}{SI Appendix \prettyref{fig:smooth-full}} for \textit{Hamlet} data.)
Estimated probabilities $\widehat{p}_i$ are plotted in solid against the empirical count $N_i$, while gray dots represent scatter plot of true probabilities $p^\star_i$ versus $N_i$. 
Unlike the Good--Turing estimators, the estimated probabilities by the NPMLE and Laplace estimators are  always positive and monotonically increasing in the empirical count, thanks to their Bayes form.}
\label{fig:smooth}
\end{figure}



We conclude this section by discussing several connections and open problems.
\begin{itemize}
\item In addition to sequence MLE (empirical frequency) and the nonparametric MLE (this paper), another MLE for distribution estimation is the the \emph{profile maximum likelihood} (PML) \cite{PML}, which, despite being computationally challenging, is provably  competitively optimal for estimating distribution properties such as entropy or support size. 
Interestingly, the NPMLE can be related to the PML via convex relaxation, and the latter can be shown to also achieve near-optimal KL regret as a corollary to our \prettyref{thm:main} (see \ifthenelse{\boolean{arxiv}}{Appendix \ref{sec:PML}}{SI Appendix \prettyref{sec:PML}}.)


\item In this paper we applied  the proposed NPMLE only to estimate word frequency (unigrams). Stronger language models requires estimating N-grams, such as Katz smoothing, which applies the 
Good--Turing estimator with ``back-off'' \cite{JurafskyMartin}. It is an interesting direction to develop better N-gram estimators based on NPMLE in a similar competitive framework  \cite{falahatgar2020towards}.

    \item Strictly speaking, the regret bound $\widetilde O(n^{-2/3})$ in the main \prettyref{thm:main} is not entirely dimension-free as the logarithmic factors depend on the alphabet size $k$. In principle, it is possible to bound the regret uniformly in $k$, even for infinite alphabet. Indeed, \cite{orlitsky2015competitive} showed the suboptimal regret bound $\widetilde O(n^{-1/3})$ for modified Good--Turing that does not depend on $k$. Obtaining a truly dimension-free regret bound at the optimal rate remains an open problem.

\item The present paper, along with prior works, assumes Poisson sampling for the technical reason that under this model, 
the symbol counts are independent.
Empirically we observed little to no difference between the Poisson and the 
multinomial (fixed sample size) model. Nevertheless,
extending our theory and methodology to the multinomial model, which belongs to the broader class of ``permutation‑invariant decision problems'' studied by \cite{weinstein2021minimum}, is a challenging future direction.
Note that for estimating \textit{functionals} of probability distributions (such as entropy or support size \cite{jiao2015minimax,WY14,WY15}), it is known that the order of the optimal sample complexity is unchanged with or without Poissonization.



\end{itemize}

\ifthenelse{\boolean{arxiv}}{\subsection{Proof techniques}}{\section*{Proof techniques}}
In this section we explain the main idea of \prettyref{thm:main}, with details deferred to \ifthenelse{\boolean{arxiv}}{Section \ref{sec:main_proof}}{SI Appendix \prettyref{sec:main_proof}}--
\ref{sec:main_proof_PI}.
The key challenge in proving a competitive regret bound for the NPMLE is that the permutation invariant oracle $\widehat p^{\PI}$ lacks an explicit form.
Orlitsky and Suresh~\cite{orlitsky2015competitive} sidestepped this difficulty by defining a \textit{stronger} oracle---in their terminology, the optimal ``natural'' estimator---and proving a regret bound for the Good--Turing estimator with respect to this natural oracle\footnote{Specifically, an estimator $\hat{p}$ is called natural by \cite{orlitsky2015competitive} 
if $\hat p_i = \hat p_j$ whenever $N_i=N_j$; that is, for two symbols occurring the same number of times, their estimated probabilities are identical. 
The natural oracle attains  the smallest KL risk among all natural estimators, which redistributes the true total probability equally to all symbols occurring the same number of times, namely, $\pnatural_i = \Big(\sum_{j=1}^k p^\star_j \Indc\{N_j=N_i\}\Big)/\Big(\sum_{j=1}^k \Indc\{N_j=N_i\}\Big)$.} instead.
This approach avoids the need to analyze $\widehat p^{\PI}$, but it is necessarily loose: any estimator has regret at least $\widetilde{\Omega}(n^{-1/2})$ with respect to the natural oracle~\cite{AchJafOrl13}, so this strategy is too crude to establish the optimal rate $\widetilde O(n^{-2/3})$ in Theorem~\ref{thm:main}.

We therefore adopt a different strategy. The key idea is to introduce a \textit{weaker} (rather than stronger) oracle, called the separable oracle. 
We first show that the NPMLE can compete optimally with the separable oracle with a regret of $\widetilde O(n^{-2/3})$. Then we show that the separable oracle itself incurs the same $\widetilde O(n^{-2/3})$ regret 
over the PI oracle.



To this end, we first re-interpret $\widehat p^{\PI}$ as the optimal estimator in a Bayesian version of the studied model.
This alternative interpretation places the PI oracle into an established framework for the analysis of empirical Bayes procedures~\cite{greenshtein2009asymptotic}, and opens the door to a tighter analysis.
More precisely, for this Bayesian analogue we first define $p = \pi( p^\star)$ to be a random permutation of $p^\star$, with 
$\pi$ uniformly distributed on $S_k$, the set of all permutations on $[k]$, and draw conditionally independent observations $N_i \sim \poi(\theta_i)$ conditioned on $(p_1,\dots,p_k)$, where $\theta_i := np_i$. 
One can show that (see \ifthenelse{\boolean{arxiv}}{Appendix \ref{sec:PI_equivalence}}{SI Appendix, \prettyref{sec:PI_equivalence}})
the PI oracle is given by the posterior mean under this model:
\begin{align}\label{eq:PI_oracle_intro}
	\widehat p^{\PI}_i = \E[p_i | N]\,,
\end{align}
which is allowed to depend on the full counts $N = (N_1,\ldots,N_k)$. 
Here and below, the notation $\E$ without subscript denotes expectation with respect to the joint distribution of $(p_i,N_i)_{i=1}^k$. We compare $\widehat p^{\PI}$ to a weaker oracle, defined by
\begin{align}\label{eq:separable_oracle_intro}
	\bar{p}_i^{\Sym} := \E[p_i | N_i]\,,
\end{align}
which is only allowed to depend on the individual count $N_i$. 
Following the terminology in compound estimation \cite{greenshtein2009asymptotic}, we call $\bar p^{\Sym} = (\bar p^{\Sym}_1,\ldots,\bar p_k^{\Sym})$ the \textit{separable oracle}, since each $\bar{p}_i^{\Sym}$ only depends on $N_i$. 
Note that the separable oracle is simply given by Bayesian estimator \prettyref{eq:bayes} (scaled by $1/n$), with the empirical distribution of $\theta^\star$ playing the role of the prior. In comparison, 
the PI oracle \prettyref{eq:PI_oracle_intro} is 
\begin{equation}
\widehat p^{\PI}_i = 
\frac{
\sum_{\pi \in S_k} p^\star_{\pi(i)} \prod_{j=1}^k 
(p^{\star}_{\pi(j)})^{N_j}
}{\sum_{\pi \in S_k}  \prod_{j=1}^k 
(p^{\star}_{\pi(j)})^{N_j} }.
    \label{eq:PI-oracle-full}
\end{equation}
This complicated form presents difficulties for both theoretical analysis and computation.

Using the separable oracle as an intermediary, our proof starts by decomposing the regret into two parts (see \eqbr{eq:reg_decomp} in \ifthenelse{\boolean{arxiv}}{\prettyref{sec:main_proof}}{SI Appendix} for details):
\begin{align}\label{eq:reg_decomp_main}
\notag &\reg(\pNPMLE) =\E \qth{\DKL(p^\star \| \pNPMLE)} - \E\qth{\DKL(p^\star \| \widehat{p}^{\PI})}\\
&\leq \E \qth{\sum_{i=1}^k p_i^\star \log \frac{\bar{p}^{\Sym}_i}{\bar{p}_i} - \bar{p}^{\Sym}_i + \bar{p}_i} + \E \big[\DKL(\widehat{p}^{\PI} \| \bar{p}^{\Sym})\big],
\end{align}
where 
$\bar p$ is the unnormalized version of the NPMLE \eqbr{def:npmle_est}:
\begin{equation}
	\bar p_i := \frac{ \theta_{\widehat G}(N_i) + \hp \indc{N_i = 0}}{n}.
    \label{eq:NPMLE-unnormalized}
\end{equation}
Note that unlike the PI oracle \prettyref{eq:PI_oracle_intro}, the separable oracle
$\bar{p}^{\Sym}$ in \prettyref{eq:separable_oracle_intro}
is not necessarily normalized so the KL divergence in the second term is understood more generally as 
$\DKL(\widehat{p}^{\PI} \| \bar{p}^{\Sym}) :=\sum_{i=1}^k (\widehat{p}^{\PI}_i \log({\widehat{p}^{\PI}_i}/{\bar{p}^{\Sym}_i}) - \widehat{p}^{\PI}_i + \bar{p}^{\Sym}_i)$.
The two terms in \eqbr{eq:reg_decomp_main} have the following interpretation:
The first term is analogous to a KL divergence between $\bar p$ and $\bar{p}^{\Sym}$, and can be viewed as the regret of $\pNPMLE$ over the separable oracle $\bar{p}^{\Sym}$. 
The second term 
is  the regret of the separable oracle $\bar{p}^{\Sym}$ over the PI oracle $\widehat{p}^{\PI}$.
We show that both terms are $\widetilde{O}(n^{-2/3}\wedge \frac{k}{n})$.

\ifthenelse{\boolean{arxiv}}{\subsubsection{Competing against the separable oracle}}{\subsection*{Competing against the separable oracle}} In analogy to the ``Bayesian" interpretation for $\widehat p^{\PI}$ above, we consider a slightly different Bayesian experiment for the analysis of $\bar p^{\Sym}$. For any distribution $G$ on $\R_+$, consider the pair $(\theta, Y)\in \R_+\times \mathbb{N}$ generated by $\theta \sim G$, and $Y|\theta \sim \Poi(\theta)$. Then the current data generating scheme for $(N_1,\ldots,N_k)$ under the true means $\theta^\star = (np^\star_1,\ldots,np^\star_k)$ (with Poissonized sample size $\Poi(n)$) is related to the above experiment via the so-called \emph{fundamental theorem of compound estimation} \cite{zhang2003compound}: for any measurable function $f: \R_+\times \mathbb{N} \rightarrow \R$,
\begin{align}\label{eq:compound_fundamental}
\E\Big[\frac{1}{k}\sum_{i=1}^k f(\theta_i^\star, N_i)\Big] = \E_{G_k}[f(\theta,Y)],
\end{align}
where $G_k := k^{-1}\sum_{i=1}^k \delta_{\theta^\star_i}$ denotes the empirical distribution of the true means. Suppose for now that the NPMLE $\widehat{G}$ in \eqbr{def:npmle_est} is a deterministic and data-independent distribution and is close to $G_k$ in some appropriate sense to be specified below. Then both $\bar p^{\Sym}_i = n^{-1}\E_{G_k}[\theta_i|N_i]$ and $\bar p_i = n^{-1}(\E_{\widehat{G}}[\theta_i|N_i]+ \hp \indc{N_i = 0})$ only depend on $N_i$, so that  \eqbr{eq:compound_fundamental} applies. Drawing on the recent result on Poisson empirical Bayes by a subset of the authors \cite{shen2022empirical}, we obtain the following comparison with $\bar{p}^{\Sym}$ (in a related metric to the KL divergence):
\ifthenelse{\boolean{arxiv}}{\begin{equation}\label{eq:sep_quantity}
\begin{split}
&\E\sum_{i=1}^k\Big[p^\star_i \log \frac{\bar p^{\Sym}_i}{\bar p_i} - \bar p^{\Sym}_i + \bar p_i\Big] \\
&= \frac{k}{n} \E_{G_k}\Big[ \theta \log \frac{\E_{G_k}[\theta | Y]}{\E_{\widehat{G}}[\theta | Y] + \tau\bm{1}\{Y=0\}} - \E_{G_k}[\theta | Y] + \E_{\widehat{G}}[\theta | Y] + \tau\bm{1}\{Y = 0\}\Big]\\
&\leq C\frac{k}{n}\eps^2 \log^5(nk) + \text{additional negligible terms}, 
\end{split}
\end{equation}}{
\begin{equation}\label{eq:sep_quantity}
\begin{split}
&\E\sum_{i=1}^k\Big[p^\star_i \log \frac{\bar p^{\Sym}_i}{\bar p_i} - \bar p^{\Sym}_i + \bar p_i\Big] \\
&= \frac{k}{n} \E_{G_k}\Big[ \theta \log \frac{\E_{G_k}[\theta | Y]}{\E_{\widehat{G}}[\theta | Y] + \tau\bm{1}\{Y=0\}} - \E_{G_k}[\theta | Y]\\
&\quad + \E_{\widehat{G}}[\theta | Y] + \tau\bm{1}\{Y = 0\}\Big]\\
&\leq C\frac{k}{n}\eps^2 \log^5(nk) + \text{additional negligible terms}, 
\end{split}
\end{equation}}
where $\eps^2$ is the squared \textit{Hellinger distance} between  $f_{G_k}$ and $f_{\widehat{G}}$. In words, with $k/n$ being the natural rescaling factor in applying \eqbr{eq:compound_fundamental}, the above bound relates the regret over the separable oracle $\bar p^{\Sym}$ to the Poisson mixture density estimation error.
By showing that the NPMLE in \eqbr{def:npmle} achieves the near-optimal rate of $\eps^2 = \widetilde{O}(\frac{n^{1/3}}{k} \wedge 1)$, the above bound leads to the $\widetilde{O}(n^{-2/3}\wedge \frac{k}{n})$ regret against the separable oracle. 

A technical caveat in the above program is that $\widehat{G}$ depends on the entire sample $(N_1,\dots,N_k)$ so strictly speaking $\bar p_i$ is not separable and \prettyref{eq:compound_fundamental} cannot be applied. To overcome this challenge, we apply additional arguments to first round $\widehat G$ to a discretization of the probability simplex ($\varepsilon$-net), and then apply the above argument to each fixed distributions in this net with a careful union bound. 

\ifthenelse{\boolean{arxiv}}{\subsubsection{Comparing the PI oracle to the separable oracle}}{\subsection*{Comparing the PI oracle to the separable oracle}}
A key ingredient of our analysis is a \textit{mean-field} view that approximates the PI oracle $\widehat{p}^{\PI}$, a complicated object both mathematically and computationally, by the much simpler separable oracle $\bar{p}^{\Sym}$. Concretely, the joint distribution of $(p_1,N_1,\dots,p_k,N_k)$ is given by
\begin{align*}
\frac{1}{k!}\sum_{\pi\in S_k} \prod_{i=1}^k \indc{p_i = p_{\pi(i)}^\star} \Poi(N_i; np_i), 
\end{align*}
called a ``permutation mixture'' in the language of \cite{han2024approximate}. The mean-field approximation proceeds \emph{as if} $(p_1,N_1), \dots, (p_k,N_k)$ were mutually independent, so that by \eqbr{eq:PI_oracle_intro} and \eqbr{eq:separable_oracle_intro},
\begin{align*}
\widehat{p}^{\PI}_i = \E[p_i | N] \approx \E[p_i | N_i] = \bar{p}^{\Sym}_i. 
\end{align*}
The key technical challenge is to prove a quantitative bound on the difference between these two oracles. To this end, the empirical Bayes literature contains several ad hoc techniques \cite{hannan1955asymptotic, greenshtein2009asymptotic}, but these results are too loose in our problem (see \ifthenelse{\boolean{arxiv}}{\prettyref{sec:main_proof_PI}}{SI Appendix \prettyref{sec:main_proof_PI}} for details).

The following theorem quantifies the mean-field approximation of the PI oracle $\widehat{p}^{\PI}$, a critical step in establishing the regret upper bound in Theorem \ref{thm:main}. 
\begin{theorem}\label{thm:PI-Poisson}
There is an universal constant $C>0$ such that the PI oracle \eqbr{eq:PI_oracle_intro} and the separable oracle \eqbr{eq:separable_oracle_intro} satisfy
\begin{equation*}
\begin{split}
    \E \DKL(\widehat{p}^{\PI} \| \bar{p}^{\Sym}) \le C\pth{\frac{k\wedge n^{1/3}}{n}\log^2 n + \frac{\log^3 n}{n}}. 
\end{split}
\end{equation*}
\end{theorem}

The proof of Theorem \ref{thm:PI-Poisson} relies on several technical ingredients. 
First, we introduce a stochastic interpolation $Z_i$ between each $\theta_i$ and $N_i$ that acts as a ``noisy'' version of $\theta_i$ but a ``cleaner'' version of $N_i$. 
Second,  we apply an information-theoretic argument to show that
\begin{align}
\E \DKL(\widehat{p}^{\PI} \| \bar{p}^{\Sym}) \le \frac{C}{n}\DKL\pth{P_{Z} \Big\| \prod_{i=1}^k P_{Z_i}},
\label{eq:KLPI}
\end{align}
where $P_{Z_i}$ denotes the marginal law of $Z_i$'s and $P_Z$ their joint law (a permutation mixture).
The raison d'\^{e}tre for introducing the interpolation is the following: if $Z$ were replaced by $\theta$, the KL divergence $\DKL(P_{\theta} \| \prod_{i=1}^k P_{\theta_i})$ would be prohibitively large in the dimension $k$. In contrast, a recent work \cite{han2024approximate} by a subset of the authors discovers a surprising result that, for noisy versions of $\theta$, the KL divergence $\DKL(P_{Z} \| \prod_{i=1}^k P_{Z_i})$ becomes significantly smaller and \emph{dimension-free} even for $k=\infty$. Technically, this result in \cite{han2024approximate} requires a bounded $\chi^2$-diameter: $\mathsf{D}_{\chi^2}(\calP) := \max_{P_1,P_2\in \calP} \chi^2(P_1 \| P_2) = O(1)$, where $\calP:= \{P_{Z_1|\theta_1 = \theta_i^\star}: i\in [k]\}$ is a family of ``noisy'' distributions and 
$\chi^2(P_1 \| P_2):=
\sum_i \frac{P_1(i)^2}{P_2(i)}-1$ is the $\chi^2$-divergence. We generalize this result to unbounded $\chi^2$-diameters by partitioning $\calP$ into $m$ sub-families each with bounded $\chi^2$-diameters, which allows us to show the KL divergence in \prettyref{eq:KLPI} is nearly linear in the size $m$ of the partition. It turns out for Poisson models we may choose $m = O(k\wedge n^{1/3})$, leading to the result in  \prettyref{thm:PI-Poisson}.

Our technique also applies to mean-field approximations beyond Poisson models. For instance, we provide the first major improvement of the state-of-the-art result in \cite{greenshtein2009asymptotic} for compound normal mean models (see \ifthenelse{\boolean{arxiv}}{\prettyref{sec:main_proof_PI}}{SI Appendix \prettyref{sec:main_proof_PI}}).

\ifthenelse{\boolean{arxiv}}{}{
\acknow{Y. Han is supported in part by a fund from Renaissance Philanthropy. J. Niles-Weed is supported in part by the National Science Foundation under Grant No.~DMS-2339829.
Part of the work of Y.~Wu was supported by the National Science Foundation under Grant No.~DMS-1928930, while Y.~Wu was in residence at the Simons Laufer Mathematical Sciences Institute in Berkeley, California, during the Spring 2025 semester.}

\showacknow{} 
}

%% file: proofs.tex
\section{Regret analysis of the NPMLE: Proof of \prettyref{thm:main}}
\label{sec:main_proof}

In this section and the next two, we prove our main result stated in \prettyref{thm:main} \ifthenelse{\boolean{arxiv}}{}{in the main text }that establishes the regret optimality of the NPMLE estimator.
The major steps of the proof are presented here, including a more general regret bound that implies the desired result, along with key supporting results whose proofs are deferred to \prettyref{sec:main_proof_separable} and 
\ref{sec:main_proof_PI}.

We first recall some notation\ifthenelse{\boolean{arxiv}}{}{ from the main text}. For any distribution $G$ on $\R_+$, the posterior mean in the Poisson model is given by
\begin{align}\label{eq:posterior_mean}
\theta_G(y) := \E_G[\theta|Y = y] = (y+1)\frac{f_G(y+1)}{f_G(y)} = (y+1)\Big(\frac{\Delta f_G(y)}{f_G(y)} + 1\Big),
\end{align}
where the expectation is taken over the pair $(\theta,Y)$ distributed as $\theta \sim G$, $Y|\theta \sim \Poi(\theta)$, $f_G(\cdot)$ denotes the Poisson mixture probability mass function $f_G(y) = \int \frac{\theta^y e^{-\theta}}{y!} G(\d\theta)$, and $\Delta f_G(y) = f_G(y+1) - f_G(y)$ is the forward difference. Given observations $N = (N_1,\ldots,N_k)$, the nonparametric maximum likelihood estimator (NPMLE) is defined as
\begin{align*}
\widehat G^{\NPMLE}:= \argmax_{G\in \calP(\R_+)} \sum_{i=1}^k \log f_G(N_i), 
\end{align*}
maximizing over all probability distributions on $\R_+$. The NPMLE-based estimator for $p^\star$ is: for $1\leq i\leq k$, 
\begin{align}\label{def:npmle_estimator}
\widehat p_i^{\NPMLE} := \frac{\bar p_i^{\NPMLE}}{Z^{\NPMLE}}, \quad \text{ where } \bar{p}^{\NPMLE}_i := \frac{\theta_{\widehat G^{\NPMLE}}(N_i) + \tau \bm{1}\{N_i = 0\}}{n}, \quad Z^{\NPMLE} := \sum_{i=1}^k \bar p_i^{\NPMLE},
\end{align}
where $\tau > 0$ is some tuning parameter to account for unseen symbols. We restate \prettyref{thm:main}\ifthenelse{\boolean{arxiv}}{}{ of the main text} as below. 

\begin{theorem}\label{thm:main-SI}
With $\hp = k^{-C_0}$ for $C_0\ge 1$, there exists some universal $C > 0$ such that 
\begin{align*}
\reg(\widehat{p}^{\NPMLE}) \le C\pth{n^{-2/3}\wedge \frac{k}{n}}\log^{14}(nk) . 
\end{align*}
\end{theorem}

\subsection{A general regret bound}
To prove \prettyref{thm:main-SI}, we first show a general result that bounds the KL regret of any estimator $\widehat{G} = \widehat G(N)$ in terms of its density estimation guarantee. 
As a price of this generality, we introduce an additional tuning parameter $\rho>0$ in the final estimator:
\begin{align}\label{eq:est_regularized}
\widehat{p}_i := \frac{\bar{p}_i}{Z}, \quad \text{ where }\bar{p}_i:= \frac{\theta_{\widehat{G}}(N_i; \rho) + \hp \indc{N_i = 0}}{n}, \quad Z:= \sum_{i=1}^k \bar{p}_i,
\end{align}
where $\theta_G(y; \rho)$ is the regularized version of the Bayes estimator \eqbr{eq:posterior_mean} with $G$ as the prior, defined as\footnote{For $x,y\in \reals$, define $x\wedge y:= \min\{x,y\}$ and $x\vee y:= \max\{x,y\}$.}
\begin{align}\label{eq:regularized_bayes}
\theta_G(y;\rho) := (y+1)\Big(\frac{\Delta f_G(y)}{f_G(y) \vee \rho} + 1\Big), \quad y \in \naturals.
\end{align}
The following result controls the regret of the estimator $\widehat{p}$ in \eqbr{eq:est_regularized} in terms of the Hellinger distance\footnote{The squared Hellinger distance between pmfs $f$ and $g$ on $\integers_+$ are defined as $H^2(f,g) = \sum_{i\geq 0} \Big(\sqrt{f(i)}-\sqrt{g(i)}\Big)^2$.
} between the estimated Poisson mixture $\widehat f_{\widehat G}$ and the population. 
\begin{theorem}\label{thm:upper_bound_general}
Let $G_k := \frac{1}{k}\sum_{i=1}^k \delta_{np_i^\star}$ be the empirical distribution of $(np_1^\star, \dots, np_k^\star)$. Assume that the estimator $\widehat{G} = \widehat{G}(N)$ satisfies that
\begin{align}\label{eq:hellinger_bound}
\sup_{p^\star \in \Delta_k} \P\pth{ H^2(f_{G_k}, f_{\widehat{G}}) \ge \varepsilon^2 } \le \delta, 
\end{align}
for some $\varepsilon \ge k^{-1/2}$ and $\delta > 0$, then for $\hp = k^{-C_0}$ with any $C_0\ge 1$ and $\rho = c(nk)^{-5}$ with a small constant $c>0$, the distribution estimator $\widehat{p}$ in \eqbr{eq:est_regularized} satisfies the regret upper bound
\begin{align*}
\reg(\widehat{p}) \le \frac{C}{n}\pth{k\varepsilon^2\log^6(nk) + (k\wedge n^{1/3})\log^{4.5}(nk) + \delta(n+k)\log^6(nk)}, 
\end{align*}
with an absolute constant $C>0$. 
\end{theorem}

Next, we apply \Cref{thm:upper_bound_general} 
to $\widehat G = \widehat{G}^{\NPMLE}$ in \eqbr{def:npmle}
to deduce \Cref{thm:main-SI}. To this end, we need two main properties of the NPMLE. First, the first-order optimality condition \prettyref{eq:kkt} of the NPMLE $\widehat{G}^{\NPMLE}$ implies 
the following deterministic lower bound on the estimated Poisson mixture density on each data point (see \cite[Lemma 16]{shen2022empirical}):
\begin{align*}
f_{\widehat{G}^{\NPMLE}}(N_i) \ge \frac{c}{k\sqrt{N_i+1}},\quad  i\in [k].
\end{align*}
Therefore, the regularization in \eqbr{eq:regularized_bayes} with 
$\rho = c(nk)^{-5}$ has no effect for the NPMLE $\widehat{G}^{\NPMLE}$, i.e. $\theta_{\widehat{G}^{\NPMLE}}(N_i; \rho) = \theta_{\widehat{G}^{\NPMLE}}(N_i)$ as long as $N_i\le 2n$ for all $i\in [k]$, an event with probability at least $1-\exp(-\Omega(n))$, because each $N_i \leq \sum N_i \sim \Poi(n)$. Second, Lemma \ref{lem:compound_density} applied to $p=1$ shows that the NPMLE satisfies the density estimator error bound
\eqbr{eq:hellinger_bound} with
\begin{align*}
\varepsilon^2 \asymp 1 \wedge \frac{n^{1/3}\log^8 k}{k}, \qquad \delta = \exp\pth{ -\Omega\pth{n^{1/3}\log k} }, 
\end{align*}
and thus \Cref{thm:upper_bound_general} implies \Cref{thm:main-SI}.

\subsection{Proof outline of \Cref{thm:upper_bound_general}}

We now describe the main steps of proving \Cref{thm:upper_bound_general}, deferring the detailed proofs to Sections \ref{sec:main_proof_separable} and \ref{sec:main_proof_PI} below. Consider the following intermediate quantity known as the (unnormalized) \emph{separable oracle}, defined in \eqbr{eq:separable_oracle_intro}\ifthenelse{\boolean{arxiv}}{}{ of the main text}:
\begin{align}\label{eq:separable_oracle}
\bar{p}_i^{\Sym} = \frac{\theta_{G_k}(N_i)}{n} = \frac{N_i+1}{n}\cdot \frac{f_{G_k}(N_i+1)}{f_{G_k}(N_i)}, 
\end{align}
which corresponds to the optimal decision rule applied to each data point $N_i$. Then we may decompose the excess risk of $\widehat{p}$ over the permutation-invariant oracle $\widehat{p}^{\PI}$ as  
\begin{align}\label{eq:reg_decomp}
\notag &\E \qth{\DKL(p^\star \| \widehat{p})} - \E\qth{\DKL(p^\star \| \widehat{p}^{\PI})}\\
\notag &= \E \qth{\sum_{i=1}^k p^\star_i \log \frac{\widehat{p}^{\PI}_i}{\bar{p}_i/Z}} = \E \qth{\sum_{i=1}^k \pth{p^\star_i \log \frac{\bar{p}^{\Sym}_i}{\bar{p}_i} + p_i^\star\log \frac{\widehat{p}^{\PI}_i}{\bar{p}^{\Sym}_i}}} + \E\qth{\log Z}\\
\notag&= \underbrace{\E \qth{\sum_{i=1}^k p_i^\star \log \frac{\bar{p}^{\Sym}_i}{\bar{p}_i} - \bar{p}^{\Sym}_i + \bar{p}_i}}_{(\mathrm{I})} + \underbrace{\E\qth{\sum_{i=1}^k p_i^\star \log \frac{\widehat{p}^{\PI}_i}{\bar{p}^{\Sym}_i} - \widehat{p}^{\PI}_i + \bar{p}^{\Sym}_i}}_{(\mathrm{II})} + \E[\log Z - Z + 1]\\
&\leq (\mathrm{I}) + (\mathrm{II}),
\end{align}
using the fact that $\log x \leq x - 1$ for $x > 0$. In words, \eqbr{eq:reg_decomp} decomposes the regret into two parts, where $(\mathrm{I})$ concerns the closeness between our (unnormalized) estimator $\bar{p}_i$ and the separable oracle $\bar{p}_i^{\Sym}$, and $(\mathrm{II})$ concerns the closeness between the separable oracle $\bar{p}_i^{\Sym}$ and the permutation-invariant oracle $\widehat{p}_i^{\PI}$. 


By considering the cases of large and small counts separately, we further decompose $(\mathrm{I})$ as
\begin{align*}
(\mathrm{I}) = \underbrace{\E \qth{\sum_{i=1}^k \Big(p^\star_i \log \frac{\bar{p}^{\Sym}_i}{\bar{p}_i} - \bar{p}^{\Sym}_i + \bar{p}_i\Big) \indc{N_i \leq y_0}}}_{(\mathrm{I}_1)} + \underbrace{\E \qth{\sum_{i=1}^k \Big(p^\star_i \log \frac{\bar{p}^{\Sym}_i}{\bar{p}_i} - \bar{p}^{\Sym}_i + \bar{p}_i\Big) \indc{N_i > y_0}}}_{(\mathrm{I}_2)},
\end{align*}
where $y_0 = C_1\log^2(nk)$ for some large $C_1 > 0$ (assumed for simplicity to be an integer). \Cref{thm:upper_bound_general} then follows from the following upper bounds on $(\mathrm{I}_1), (\mathrm{I}_2)$, and $(\mathrm{II})$. 

\begin{proposition}\label{prop:reg_sep_smally}
For $(\mathrm{I}_1)$, there is an absolute constant $C>0$ depending only on $(C_0, C_1)$ such that
\begin{align*}
(\mathrm{I}_1)\le \frac{C\log^6(nk)}{n}\pth{k\varepsilon^2 + \delta(n+k)}.
\end{align*}
\end{proposition}

\begin{proposition}\label{prop:reg_sep_largey}
For $(\mathrm{I}_2)$, there is an absolute constant $C>0$ depending only on $(C_0, C_1)$ such that
\begin{align*}
(\mathrm{I}_2)\le \frac{C}{n}\pth{k\varepsilon^2\log^4(nk) + (n^{1/3}\wedge k)\log^{4.5}(nk) + \delta\log^2(nk)}. 
\end{align*}
\end{proposition}

\begin{proposition}\label{prop:reg_PI}
For $(\mathrm{II})$, there is a universal constant $C>0$ such that
\begin{align*}
(\mathrm{II})\le \frac{C}{n}\pth{ (k\wedge n^{1/3})\log^2 n + \log^3 n }. 
\end{align*}
\end{proposition}

We present the proofs for Propositions \ref{prop:reg_sep_smally} and \ref{prop:reg_sep_largey} in Section \ref{sec:main_proof_separable}, and the proof of Proposition \ref{prop:reg_PI} in Section \ref{sec:main_proof_PI}. 

\subsection{Auxiliary results}

We collect in this section a few auxiliary results that will be repeatedly used in our analysis.
\begin{lemma}[Lemma 13 in \cite{shen2022empirical}]\label{lem:bayes_deviation}
There exists a universal constant $C > 0$ such that the following holds.
For any prior $G$ on $\R_+$ and any $y\in\Z_+$, 
\begin{align}\label{ineq:bayes_upper_1}
|\theta_G(y)-y| \leq \E(|\theta - Y| |Y = y) &\leq C\sqrt{y+ 1}\log\frac{1}{f_G(y)},
\end{align}
where the expectation is taken over $\theta\sim G$ and $Y|\theta\sim \poi(\theta)$, and $\theta_G(y) = \E[\theta|Y=y]$.
Moreover, for any $\rho\leq 1/e$ and $y\geq 0$,
\begin{align}
|\theta_G(y;\rho) - (y+1)| \leq C\sqrt{y+1}\log\frac{1}{\rho}.
\label{eq:R1R3mom}
\end{align}
\end{lemma}

Let $\theta = (\theta_1,\ldots,\theta_k) \in \R_+^k$ be a deterministic vector and $N = (N_1,\ldots,N_k)$ be independent variables with $N_i \sim \poi(\theta_i)$ 
for $1\leq i \leq k$. Let $f_i(\cdot) \equiv \poi(\cdot;\theta_i)$ be the $i$th marginal pmf, $G_k \equiv k^{-1}\sum_{i=1}^k\delta_{\theta_i}$, and the average density be
\begin{align}
f_{G_k}(y) \equiv \frac{1}{k}\sum_{i=1}^k f_i(y), \quad y \in \mathbb{Z}_+.
\end{align}
For any distribution $G$ on $\R_{+}$ and $p>0$, let
\begin{align*}
m_p(G) \equiv \int_{\R_+} u^p G(\d u).
\end{align*}
Let $\widehat{G}$ be the NPMLE given by \eqbr{def:npmle}. The $\Prob$ and $\E$ in the lemma below are under the randomness of $N$ described above.

\begin{lemma}[Proposition 27 in \cite{shen2022empirical}]\label{lem:compound_density}
Suppose $p > 0$ and $m_p(G_k)^{1/p} \leq k^{10}$. Let
\begin{align}\label{eq:eps_compound}
\varepsilon_k := \big(k^{-p/(2p+1)}m_p(G_k)^{1/(4p+2)} \vee k^{-1/2}\big) \log^4 k. 
\end{align}
Then there exists some $t_\star = t_\star(p)$ such that for all $t \geq t_\star$,
\begin{align*}
\Prob\Big(H(f_{\widehat{G}}, f_{G_k}) \geq t\varepsilon_k\Big) \leq 2\exp\Big(-\frac{t^2k\varepsilon_k^2}{8\log k}\Big) \leq 2\exp\Big(-\frac{t^2\log^2 k}{8}\Big).
\end{align*}
Consequently, there exists some $C = C(p) > 0$ such that $\E[H^2(f_{\widehat{G}}, f_{G_k})] \leq C\varepsilon_k^2$.
\end{lemma}

The next lemma is the well-known Chernoff bound. 
\begin{lemma}[Theorems 4.4 and 4.5 in \cite{mitzenmacher2005probability}]\label{lemma:chernoff} Fix any $\theta > 0$ and positive integer $n$. For $X\sim \Poi(\theta)$, $X\sim \mathrm{B}\pth{n, \frac{\theta}{n}}$, or $X\overset{\mathrm{d}}{=} \sum_{i=1}^n X_i$ with independent $X_i\sim \Bern(p_i)$ and $\sum_{i=1}^n p_i = \theta$, the following inequalities hold for all $\delta >0$: 
\begin{align*}
\P\pth{X \ge (1+\delta)\theta} &\le \pth{\frac{e^{\delta}}{(1+\delta)^{1+\delta}}}^{\theta} \le \exp\pth{-\frac{(\delta \wedge \delta^2)\theta}{3}}, \\
\P\pth{X \le (1-\delta)\theta} &\le \pth{\frac{e^{-\delta}}{(1-\delta)^{1-\delta}}}^{\theta} \le \exp\pth{-\frac{\delta^2\theta}{2}}.
\end{align*}
In particular, for every $t\ge 0$, 
\begin{align*}
\P\pth{|\sqrt{X} - \sqrt{\theta}| \ge t} \le 2\exp\pth{-\frac{t^2}{3}}.
\end{align*}
\end{lemma}

\section{Proof of Theorem \ref{thm:upper_bound_general} (Part 1)}\label{sec:main_proof_separable}

\subsection{Proof of Proposition \ref{prop:reg_sep_smally}}
Recall that $\theta_i^\star = np_i^\star$, then
\begin{align*}
(\mathrm{I}_1) &= \frac{1}{n}\sum_{y=0}^{y_0}\E\Big[\sum_{i=1}^k \Big(\theta^\star_i \log \frac{\theta_{G_k}(N_i)}{\theta_{\widehat G}(N_i;\rho) + \hp\indc{N_i = 0}} - \theta_{G_k}(N_i) + (\theta_{\widehat G}(N_i;\rho) + \hp\indc{N_i = 0})\Big)\indc{N_i = y}\Big]\\
&= \frac{1}{n}\sum_{y=0}^{y_0}\E\underbrace{\Big[S_y\log \frac{\theta_{G_k}(y)}{\theta_{\widehat{G}}(y;\rho) + \hp\indc{y=0}} - \Phi_y\theta_{G_k}(y) + \Phi_y(\theta_{\widehat{G}}(y;\rho) + \hp\indc{y=0})\Big]}_{B_y},
\end{align*}
where $\Phi_y := \sum_{i=1}^k \indc{N_i = y}$ and $S_y := \sum_{i=1}^k \theta_i^\star \indc{N_i = y}$. We note the following fact:
\begin{align}\label{eq:Phi_property}
\E[\Phi_y] = k f_{G_k}(y), \qquad \E[S_y] = k\theta_{G_k}(y) f_{G_k}(y) = k (y+1)f_{G_k}(y+1).
\end{align}

We will show that for all $y\in [0, y_0]$,
\begin{align}\label{eq:target_small_y}
\E[B_y] = O\pth{\pth{k\varepsilon^2 + \delta(n+k)}\log^4(nk)  }. 
\end{align}
In view of the choice $y_0 = C_1 \log^2(nk)$, the above bound will imply the desired conclusion. In the sequel, we consider three cases to establish \eqbr{eq:target_small_y}. 

\subsubsection*{Case I: $f_{G_k}(y)\le 4\eps^2$}
In this case we apply several deterministic bounds. First, by Lemma \ref{lem:bayes_deviation}, we have
\begin{align*}
\theta_{G_k}(y) \leq C(y+1)\log\pth{\frac{1}{f_{G_k}(y)}}, \quad 
\theta_{\widehat{G}}(y;\rho) + \hp\indc{y=0} \leq C(y+1)\log\pth{\frac{1}{\rho}}. 
\end{align*}
Second, by Lemma \ref{lem:reg_bayes_lower} stated below and the choices of $\tau,\rho$, we also have
\begin{align*}
\theta_{\widehat G}(y;\rho) + \hp\indc{y=0} \geq \hp\wedge \frac{\rho^2}{16}.  
\end{align*}
By the above bounds, 
\begin{align*}
\E\qth{B_y} &\le \E\qth{S_y}\log \frac{C(y+1)\log(1/f_{G_k}(y))}{(\rho^2/16)\wedge \hp} + \E\qth{\Phi_y}\cdot C(y+1)\log\pth{\frac{1}{\rho}} \\
&\overset{\eqbr{eq:Phi_property}}{=} O\pth{ky_0f_{G_k}(y)\log \frac{1}{f_{G_k}(y)}\log \frac{y_0\log(1/f_{G_k}(y))}{(\rho^2/16)\wedge \hp} + ky_0f_{G_k}(y)\log\pth{\frac{1}{\rho}} } \\
&= O\pth{ k\varepsilon^2 \log^4(nk)},
\end{align*}
where the last step uses the assumptions $f_{G_k}(y) = O(\varepsilon^2)$ and $\varepsilon\ge 1/\sqrt{k}$. This proves \eqbr{eq:target_small_y}.

\subsubsection*{Case II: $f_{G_k}(y)>4\varepsilon^2$ and $f_{G_k}(y+1) \le 4\varepsilon^2$}
First, similar to Case I, we have the deterministic bound
\begin{align*}
\log \frac{\theta_{G_k}(y)}{\theta_{\widehat{G}}(y;\rho) + \hp\indc{y=0}} &\le \log\frac{C(y+1)\log(1/f_{G_k}(y))}{(\rho^2/16)\wedge \tau} = O\pth{\log(nk)}, \\
\Phi_y (\theta_{\widehat{G}}(y;\rho) + \hp) &\le k\pth{C(y+1)\log\pth{\frac{1}{\rho}} + \hp} = O\pth{k\log^3(nk)}.
\end{align*}
For a better upper bound of $\theta_{\widehat{G}}(y;\rho)$, note that when $H^2(f_{G_k}, f_{\widehat{G}})\le \varepsilon^2$, we have for any $y\geq 0$ that
\begin{align*}
\pth{\sqrt{f_{G_k}(y)} - \sqrt{f_{\widehat{G}}(y)}}^2 + \pth{\sqrt{f_{G_k}(y+1)} - \sqrt{f_{\widehat{G}}(y+1)}}^2\le \varepsilon^2. 
\end{align*}
In particular, $f_{\widehat{G}}(y)\ge f_{G_k}(y)/2 > \eps > \rho$ thanks to the assumption $f_{G_k}(y)>4\varepsilon^2$, and $f_{\widehat{G}}(y+1)\le 9\varepsilon^2$ since $f_{G_k}(y+1)\le 4\varepsilon^2$. Therefore, when $H^2(f_{G_k}, f_{\widehat{G}})\le \varepsilon^2$,
\begin{align*}
\theta_{\widehat{G}}(y;\rho) = (y+1)\frac{f_{\widehat{G}}(y+1)}{f_{\widehat{G}}(y)} = O\pth{\frac{y_0\varepsilon^2}{f_{G_k}(y)}}.
\end{align*}
Using the above upper bounds, in this case we get
\begin{align*}
\E\qth{B_y} &= O\pth{ \E\qth{S_y}\log(nk)+ \pth{\frac{y_0\varepsilon^2}{f_{G_k}(y)} + \hp}\E\qth{\Phi_y} + k\log^3(nk)\P\pth{H^2(f_{G_k}, f_{\widehat{G}}) > \varepsilon^2 } } \\
&\overset{\eqbr{eq:Phi_property}}{=} O\pth{ ky_0 f_{G_k}(y+1)\log(nk) + kf_{G_k}(y)\pth{\frac{y_0\varepsilon^2}{f_{G_k}(y)} + \hp} + k\delta \log^3(nk) } \\
&= O\pth{ k\varepsilon^2\log^3(nk) + k\delta \log^3(nk) }, 
\end{align*}
where the last step uses the assumption $\hp \le k^{-1}\le \varepsilon^2$. 

\subsubsection*{Case III: $f_{G_k}(y)>4\varepsilon^2$ and $f_{G_k}(y+1)>4\varepsilon^2$}
Using the deterministic upper bounds in the previous cases, we have
\begin{align*}
\E\qth{B_y\indc{ H^2(f_{G_k}, f_{\widehat{G}}) > \varepsilon^2 }} &= O\pth{n\log(nk) + k\log^3(nk)}\cdot \P\pth{H^2(f_{G_k}, f_{\widehat{G}}) > \varepsilon^2} \\
&= O\pth{ \delta \pth{n\log(nk) + k\log^3(nk)} }. 
\end{align*}
When $H^2(f_{G_k}, f_{\widehat{G}})\le \varepsilon^2$, we have
\begin{align*}
\pth{\sqrt{f_{G_k}(y)} - \sqrt{f_{\widehat{G}}(y)}}^2 + \pth{\sqrt{f_{G_k}(y+1)} - \sqrt{f_{\widehat{G}}(y+1)}}^2\le \varepsilon^2. 
\end{align*}
In particular, by the assumptions $f_{G_k}(y)>4\varepsilon^2$ and $f_{G_k}(y+1)>4\varepsilon^2$, the above implies that
\begin{align*}
\frac{f_{G_k}(y)}{2}\le f_{\widehat{G}}(y)\le 2f_{G_k}(y), \quad \frac{f_{G_k}(y+1)}{2}\le f_{\widehat{G}}(y+1)\le 2f_{G_k}(y+1). 
\end{align*}
In addition, since $f_{\widehat{G}}(y)\ge f_{G_k}(y)/2 > 2\varepsilon^2 \ge 2/k \ge \rho$, it holds that $\theta_{\widehat{G}}(y;\rho) = \theta_{\widehat{G}}(y)$. Consequently,
\begin{align*}
&\E[B_y\indc{H^2(f_{G_k}, f_{\widehat{G}}) \le \varepsilon^2}] \\
&\leq k\hp + \E\qth{ \pth{ S_y \log \frac{\theta_{G_k}(y)}{\theta_{\widehat{G}}(y)} - \Phi_y(\theta_{G_k}(y) - \theta_{\widehat{G}}(y))}\indc{H^2(f_{G_k}, f_{\widehat{G}})\le \varepsilon^2}}\\
&= k\hp + \underbrace{\E\qth{ \Phi_y \pth{ \theta_{G_k}(y)\log \frac{\theta_{G_k}(y)}{\theta_{\widehat{G}}(y)} - \theta_{G_k}(y) + \theta_{\widehat{G}}(y)}\indc{H^2(f_{G_k}, f_{\widehat{G}})\le \varepsilon^2} }}_{(\mathrm{A})}\\
&\quad + \underbrace{\E\qth{ (S_y - \Phi_y \theta_{G_k}(y))\log \frac{\theta_{G_k}(y)}{\theta_{\widehat{G}}(y)}\indc{H^2(f_{G_k}, f_{\widehat{G}})\le \varepsilon^2}} }_{(\mathrm{B})}.
\end{align*}
To bound $(\mathrm{A})$, we use $a\log (a/b) - a + b \leq (a-b)^2/b$ for $a,b>0$ to obtain
\begin{align*}
(\mathrm{A}) \leq \E\qth{\Phi_y \frac{\big(\theta_{G_k}(y) - \theta_{\widehat{G}}(y)\big)^2}{\theta_{\widehat{G}}(y)}\indc{H^2(f_{G_k}, f_{\widehat{G}})\le \varepsilon^2}}. 
\end{align*}
Define the following notations $a = f_{G_k}(y), b = f_{G_k}(y+1), \widehat{a} = f_{\widehat{G}}(y)$, and $\widehat{b} = f_{\widehat{G}}(y+1)$. We have shown that when $H^2(f_{G_k}, f_{\widehat{G}})\le \varepsilon^2$, it holds that $\widehat{a}\in [a/2,2a], \widehat{b}\in [b/2,2b]$, and 
\begin{align*}
\frac{(a-\widehat{a})^2}{a} + \frac{(b-\widehat{b})^2}{b} \le 3\frac{(a-\widehat{a})^2}{a + \widehat{a}} + 3\frac{(b-\widehat{b})^2}{b + \widehat{b}} &\le 6\pth{\sqrt{a} - \sqrt{\widehat{a}}}^2 + 6\pth{\sqrt{b} - \sqrt{\widehat{b}}}^2 \\
&\le 6H^2(f_{G_k}, f_{\widehat{G}})\le 6\varepsilon^2. 
\end{align*}
Therefore, when $H^2(f_{G_k}, f_{\widehat{G}})\le \varepsilon^2$, 
\begin{align*}
\frac{\big(\theta_{G_k}(y) - \theta_{\widehat{G}}(y)\big)^2}{\theta_{\widehat{G}}(y)} &= (y+1)\frac{\pth{\frac{b}{a}-\frac{\widehat{b}}{\widehat{a}}}^2}{\frac{b}{a}} \le 2(y+1)\frac{a}{b}\pth{ \pth{\frac{b-\widehat{b}}{a}}^2 + \pth{\frac{\widehat{b}}{a} - \frac{\widehat{b}}{\widehat{a}}}^2} \\
&= \frac{2(y+1)}{a}\pth{\frac{(b-\widehat{b})^2}{b} + \frac{\widehat{b}^2(a-\widehat{a})^2}{\widehat{a}^2 b}} = O\pth{\frac{\varepsilon^2\log^{3}(nk)}{a}}, 
\end{align*}
where the last step uses
\begin{align*}
    \frac{\widehat{b}}{\widehat{a}} = \frac{f_{\widehat{G}}(y+1)}{f_{\widehat{G}}(y)} \le \frac{2f_{G_k}(y+1)}{f_{G_k}(y)/2} = \frac{4\theta_{G_k}(y)}{y+1}  \le 4\cdot \frac{y+C\sqrt{y+1}\log(1/f_{G_k}(y))}{y+1} = O\pth{\log k}
\end{align*}
by Lemma \ref{lem:bayes_deviation} and $f_{G_k}(y) > 4\varepsilon^2\ge 4/k$. Therefore,
\begin{align*}
(\mathrm{A}) = O\pth{\frac{\varepsilon^2\log^{3}(nk)}{f_{G_k}(y)} \E\qth{\Phi_y}} \overset{\eqbr{eq:Phi_property}}{=} O\pth{k\varepsilon^2\log^{3}(nk)}. 
\end{align*} 

To upper bound $(\mathrm{B})$, by Cauchy--Schwarz 
\begin{align*}
(\mathrm{B})&\leq \E^{1/2}\Big[ \frac{(S_y  - \Phi_y\theta_{G_k}(y))^2}{kf_{G_k}(y)\theta_{G_k}(y)}\Big] \cdot \E^{1/2}\Big[ kf_{G_k}(y)\theta_{G_k}(y)\log^2 \frac{\theta_{G_k}(y)}{\theta_{\widehat{G}}(y)}\indc{H^2(f_{G_k}, f_{\widehat{G}})\le \varepsilon^2} \Big] \\
&=: \sqrt{(\mathrm{B}_1)(\mathrm{B}_2)}.
\end{align*}
Since $\E S_y = \theta_{G_k}(y) \cdot \E \Phi_y$ by \eqbr{eq:Phi_property}, we have
\begin{align*}
\E(S_y  - \Phi_y\theta_{G_k}(y))^2 &= \Var(S_y  - \Phi_y\theta_{G_k}(y)) \le 2 \Var(S_y) + 2\theta^2_{G_k}(y) \Var(\Phi_y) \\
&\leq 2\sum_{i=1}^k (\theta_i^\star)^2\poi(y;\theta_i^\star) + 2\theta_{G_k}^2(y) \cdot kf_{G_k}(y).
\end{align*}
Since $y \leq y_0 = C_1\log^2(nk)$, by considering the cases $\theta_i^\star \leq 2C_1\log^2(nk)$ and $\theta_i^\star > 2C_1\log^2(nk)$ separately and applying the Chernoff bound (cf. Lemma \ref{lemma:chernoff}) we have 
\begin{align*}
\sum_{i=1}^k (\theta_i^\star)^2\poi(y;\theta_i^\star) &= O\pth{\log^2(nk) \sum_{i=1}^k \theta_i^\star\poi(y;\theta_i^\star) + \frac{n^2}{(nk)^2} } = O\pth{\log^2(nk) kf_{G_k}(y)\theta_{G_k}(y) + \frac{1}{k^2} }, 
\end{align*}
as long as the constant $C_1>0$ is large enough. Therefore,
\begin{align*}
(\mathrm{B}_1) \stepa{=} O\pth{ \log^2(nk) + \frac{1}{k^2} + \theta_{G_k}(y) } \stepb{=} O\pth{\log^2(nk) + y + \sqrt{y+1}\log\frac{1}{f_{G_k}(y)}} \stepc{=} O\pth{\log^2(nk)}, 
\end{align*}
where (a) uses $kf_{G_k}(y)\theta_{G_k}(y) = k(y+1)f_{G_k}(y+1) \ge 4k\varepsilon^2\ge 4$, (b) is thanks to Lemma \ref{lem:bayes_deviation}, and (c) uses $y\le y_0=C_1\log^2(nk)$ and $f_{G_k}(y)>4\varepsilon^2$. 

As for $(\mathrm{B}_2)$, recall that $H^2(f_{G_k}, f_{\widehat{G}})\le \varepsilon^2$ implies
\begin{align*}
\frac{\theta_{G_k}(y)}{\theta_{\widehat{G}}(y)} = \frac{f_{G_k}(y+1)}{f_{\widehat{G}}(y+1)}\frac{f_{\widehat{G}}(y)}{f_{G_k}(y)} \in \qth{\frac{1}{4},4}. 
\end{align*}
Since $|\log x| \le 2|x-1|$ for all $x\in [1/4,4]$, we get
\begin{align*}
(\mathrm{B}_2) &\le 4kf_{G_k}(y)\theta_{G_k}(y)\E\qth{\pth{\frac{\theta_{G_k}(y)}{\theta_{\widehat{G}}(y)}-1}^2 \indc{H^2(f_{G_k}, f_{\widehat{G}})\le \varepsilon^2}}\\
&\le 16kf_{G_k}(y)\E\qth{ \frac{\big(\theta_{G_k}(y) - \theta_{\widehat{G}}(y)\big)^2}{\theta_{\widehat{G}}(y)}\indc{H^2(f_{G_k}, f_{\widehat{G}})\le \varepsilon^2}} = O\pth{k\varepsilon^2\log^3(nk)}, 
\end{align*}
where the last step follows from the analysis of $(\mathrm{A})$. Combining the upper bounds of $(\mathrm{B}_1)$ and $(\mathrm{B}_2)$ gives that
\begin{align*}
(\mathrm{B}) = O\pth{ \sqrt{k\varepsilon^2 \log^5(nk)} } = O\pth{k\varepsilon^2\log^3(nk)},
\end{align*}
as $k\varepsilon^2 \ge 1$. Collecting the upper bounds of $(\mathrm{A})$ and $(\mathrm{B})$, as well as $\hp\le k^{-1}\le \varepsilon^2$, completes the proof of \eqbr{eq:target_small_y}.  \qed

\begin{lemma}\label{lem:reg_bayes_lower}
Let $y\ge 1$ be an integer, and $\rho<1$. For any probability distribution $G$ supported on $\R_+$, it holds that $\theta_G(y;\rho) \geq \rho^2/16$.
\end{lemma}
\begin{proof}
If $f_G(y) \leq \rho/2$, then
\begin{align*}
\theta_G(y;\rho) = (y+1)\Big(\frac{f_G(y+1) - f_G(y)}{\rho} + 1\Big) \geq (y+1)\pth{1-\frac{f_G(y)}{\rho}} \geq \frac{1}{2}. 
\end{align*}
If $f_G(y) > \rho/2$, then
\begin{align*}
f_G(y+1) &= \int_{\R_+} \frac{\theta^{y+1}e^{-\theta}}{(y+1)!}G(\d \theta) \geq \int_{\theta \geq \rho/4}  \frac{\theta^{y}e^{-\theta}}{y!}\frac{\theta}{y+1}G(\d \theta) \geq \frac{\rho}{4(y+1)}\int_{\theta \ge \rho/4} \frac{\theta^{y}e^{-\theta}}{y!} G(\d\theta)\\
&= \frac{\rho}{4(y+1)}\Big(f_G(y) - \int_{\theta < \rho/4} \frac{\theta^{y}e^{-\theta}}{y!} G(\d\theta)\Big) \stepa{\ge} \frac{\rho}{4(y+1)}\pth{\frac{\rho}{2}-\frac{\rho}{4}}= \frac{\rho^2}{16(y+1)},
\end{align*}
where (a) follows from $\poi(y;\theta) \leq \poi(1;\theta) \leq \theta$ for $\theta \le 1$ and $y\geq 1$. Therefore, 
\begin{align*}
\theta_G(y;\rho) \ge (y+1)\frac{f_G(y+1)}{f_{G}(y) \vee \rho} \stepb{\geq} (y+1)f_G(y+1) \ge \frac{\rho^2}{16}, 
\end{align*}
where in (b) we use the simple inequality $f_G(y)\le 1$ as a pmf. 
\end{proof}

\subsection{Proof of Proposition \ref{prop:reg_sep_largey}}
We first prove that 
\begin{align}\label{eq:equivalence}
\E \Big[\sum_{i=1}^k \Big(p^\star_i \log \frac{\bar{p}^{\Sym}_i}{\bar{p}_i} - \bar{p}^{\Sym}_i + \bar{p}_i\Big)\indc{N_i > y_0}\Big] = \E \Big[\sum_{i=1}^k \Big(\widehat{p}_i^{\PI} \log \frac{\bar{p}^{\Sym}_i}{\bar{p}_i} - \bar{p}^{\Sym}_i + \bar{p}_i\Big)\indc{N_i > y_0}\Big]. 
\end{align}
Note that both sides of \eqbr{eq:equivalence} are stated under the original data generation scheme $N_i\sim \Poi(np_i^\star)$. In the proof of \eqbr{eq:equivalence}, we will resort to a different generating scheme discussed in \Cref{sec:PI_equivalence} where $\widehat{p}_i^{\PI}$ and $\bar{p}^{\Sym}_i$ admit simple expressions. In the new scheme, we first permute $(p^\star_1,\dots,p^\star_k)$ uniformly at random into $(p_1,\dots,p_k)$, with $p_i = p^\star_{\pi(i)}$ and $\pi\sim \mathrm{Unif}(S_k)$, and draw conditionally independently observations $N_i \sim \poi(\theta_i)$ conditioned on $(p_1,\dots,p_k)$, where $\theta_i := np_i$. Note that for any function $T = T(N,p^\star)$ with a permutation-invariant distribution (i.e., $T(\pi(N), \pi(p^\star)) \equald T(N, p^\star)$ for any permutation $\pi$), we have 
\begin{align*}
\E\qth{T(N,p^\star)} = \E\qth{T(N,p)}, 
\end{align*}
where the LHS is under the original data generation scheme, and the RHS is under the new data generation scheme. Therefore, using the fact that the distribution of $\sum_{i=1}^k \Big(p^\star_i \log \frac{\bar{p}^{\Sym}_i}{\bar{p}_i} - \bar{p}^{\Sym}_i + \bar{p}_i\Big)\indc{N_i > y_0}$ is permutation-invariant, we have
\begin{align*}
\E \Big[\sum_{i=1}^k \Big(p^\star_i \log \frac{\bar{p}^{\Sym}_i}{\bar{p}_i} - \bar{p}^{\Sym}_i + \bar{p}_i\Big)\indc{N_i > y_0}\Big] &= \E \Big[\sum_{i=1}^k \Big(p_i \log \frac{\bar{p}^{\Sym}_i}{\bar{p}_i} - \bar{p}^{\Sym}_i + \bar{p}_i\Big)\indc{N_i > y_0}\Big] \\
&\stepa{=} \E \Big[\sum_{i=1}^k \Big(\E\qth{p_i | N} \log \frac{\bar{p}^{\Sym}_i}{\bar{p}_i} - \bar{p}^{\Sym}_i + \bar{p}_i\Big)\indc{N_i > y_0}\Big] \\
&= \E \Big[\sum_{i=1}^k \Big( \widehat{p}_i^{\PI} \log \frac{\bar{p}^{\Sym}_i}{\bar{p}_i} - \bar{p}^{\Sym}_i + \bar{p}_i\Big)\indc{N_i > y_0}\Big], 
\end{align*}
where (a) is simply the tower property of conditional expectation, and the final expectation is the same under both schemes as the distribution of $\sum_{i=1}^k \Big(\widehat p^{\PI}_i \log \frac{\bar{p}^{\Sym}_i}{\bar{p}_i} - \bar{p}^{\Sym}_i + \bar{p}_i\Big)\indc{N_i > y_0}$ is permutation-invariant. This proves \eqbr{eq:equivalence}. 

Based on \eqbr{eq:equivalence}, we can write
\begin{align*}
(\mathrm{I}_2) &= \E \Big[\sum_{i=1}^k \Big(\widehat{p}_i^{\PI} \log \frac{\bar{p}^{\Sym}_i}{\bar{p}_i} - \bar{p}^{\Sym}_i + \bar{p}_i\Big)\indc{N_i > y_0}\Big] \\
&= \E \Big[\sum_{i=1}^k \Big(\bar{p}_i^{\Sym} \log \frac{\bar{p}^{\Sym}_i}{\bar{p}_i} - \bar{p}^{\Sym}_i + \bar{p}_i\Big)\indc{N_i > y_0}\Big] + \E \Big[\sum_{i=1}^k  (\widehat{p}_i^{\PI} - \bar{p}_i^{\Sym}) \log \frac{\bar{p}^{\Sym}_i}{\bar{p}_i}\indc{N_i > y_0}\Big] \\
&=: (\mathrm{A}) + (\mathrm{B}). 
\end{align*}
In the sequel, we upper bound $(\mathrm{A})$ and $(\mathrm{B})$ separately to prove Proposition \ref{prop:reg_sep_largey}.

\subsubsection*{Upper bound of $(\mathrm{A})$}
Recall that $\bar{p}_i^{\Sym} = \theta_{G_k}(N_i)/n$ and $\bar{p}_i = \theta_{\widehat{G}}(N_i; \rho)/n$. Using $a\log(a/b) - a + b\leq (a-b)^2/b$ for $a,b > 0$, we have
\begin{align*}
(\mathrm{A}) &\leq \frac{1}{n}\E\Big[\sum_{i=1}^k \frac{\big(\theta_{G_k}(N_i) - \theta_{\widehat{G}}(N_i; \rho)\big)^2}{\theta_{\widehat{G}}(N_i; \rho)}\indc{N_i > y_0}\Big] \\
&\le \frac{2}{n}\E\Big[\sum_{i=1}^k \frac{\big(\theta_{G_k}(N_i; \rho) - \theta_{\widehat{G}}(N_i; \rho)\big)^2}{\theta_{\widehat{G}}(N_i; \rho)}\indc{N_i > y_0}\Big] + \frac{2}{n}\underbrace{\E\Big[\sum_{i=1}^k \frac{\big(\theta_{G_k}(N_i; \rho) - \theta_{G_k}(N_i)\big)^2}{\theta_{\widehat{G}}(N_i; \rho)}\indc{N_i > y_0}\Big]}_{\zeta_1}. 
\end{align*}
Deferring the proof of $\zeta_1 = O(1)$ to the end of this section, we will show that
\begin{align}\label{eq:target_large_y}
    \notag &\E\Big[\sum_{i=1}^k \frac{\big(\theta_{G_k}(N_i; \rho) - \theta_{\widehat{G}}(N_i; \rho)\big)^2}{\theta_{\widehat{G}}(N_i; \rho)}\indc{N_i > y_0}\Big] \\
    &= O\pth{ k\varepsilon^2\log^4(nk) + (n^{1/3}\wedge k)\log^{4.5}(nk) + \delta\log^2(nk) }. 
\end{align}

Let $M = \lceil n^{2/3} \rceil$ be an auxiliary parameter, then 
\begin{align*}
&\E\Big[\sum_{i=1}^k \frac{\big(\theta_{G_k}(N_i; \rho) - \theta_{\widehat{G}}(N_i; \rho)\big)^2}{\theta_{\widehat{G}}(N_i; \rho)}\indc{N_i > y_0}\Big] \\
&\stepb{\le} 2\E\Big[\sum_{i=1}^k \frac{\big(\theta_{G_k}(N_i; \rho) - \theta_{\widehat{G}}(N_i; \rho)\big)^2}{N_i}\indc{N_i > y_0}\Big] \\
&\le 2\E\Big[\sum_{i=1}^k \frac{\big(\theta_{G_k}(N_i; \rho) - \theta_{\widehat{G}}(N_i; \rho)\big)^2}{N_i}\indc{y_0 < N_i \le M}\Big] + 2\underbrace{\E\Big[\sum_{i=1}^k \frac{\big(\theta_{G_k}(N_i; \rho) - \theta_{\widehat{G}}(N_i; \rho)\big)^2}{N_i}\indc{N_i > M}\Big]}_{\zeta_2} \\
&\le 2\E\Big[\sum_{i=1}^k \frac{\big(\theta_{G_k}(N_i; \rho) - \theta_{\widehat{G}}(N_i; \rho)\big)^2}{N_i}\indc{y_0 < N_i \le M, H^2(f_{G_k},f_{\widehat{G}})\le \varepsilon^2}\Big] + 2\zeta_2 \\
&\qquad + 2\underbrace{\E\Big[\sum_{i=1}^k \frac{\big(\theta_{G_k}(N_i; \rho) - \theta_{\widehat{G}}(N_i; \rho)\big)^2}{N_i}\indc{y_0 < N_i \le M, H^2(f_{G_k},f_{\widehat{G}})>\varepsilon^2}\Big]}_{\zeta_3}, 
\end{align*}
where (b) follows from Lemma \ref{lem:bayes_deviation} and $N_i>y_0=C_1\log^2(nk)$ that the deterministic inequality
\begin{align}\label{eq:theta_Ghat_lb}
\theta_{\widehat{G}}(N_i; \rho) \ge N_i - C\sqrt{N_i+1}\log \pth{\frac{1}{\rho}} \ge \frac{N_i}{2}
\end{align}
holds for a large enough constant $C_1>0$. To deal with the main term above, we need to decouple the dependence in the randomness of $N_i$ and $\widehat{G}$. To this end, we apply an $\varepsilon$-net argument to replace the data-dependent distribution $\widehat{G}$ by some fixed distribution $H$. Specifically, let $\eta = (nk)^{-1}$ and $\{H_1,\ldots, H_L\}$ be a proper $(d_M, \eta)$-covering of the set $\{G: H^2(f_G, f_{G_k}) \leq \eps^2\}$, where
\begin{align*}
d_M(G, H) := \pnorm{\theta_{G}(\cdot;\rho) - \theta_{H}(\cdot,\rho)}{\infty,M} := \max_{y=0,\dots,M} \big|\theta_{G}(y;\rho) - \theta_{H}(y;\rho)\big|.
\end{align*}
In other words, for every $\widehat{G}$ satisfying $H^2(f_{G_k}, f_{\widehat{G}})\le \varepsilon^2$, there exists $\ell\in [L]$ such that $\|\theta_{\widehat{G}}(\cdot; \rho) - \theta_{H_\ell}(\cdot; \rho)\|_{\infty,M}\le \eta$. By Lemma \ref{lem:bayes_rule_cover}, which is a restatement of \cite[Lemma 29]{shen2022empirical}, we also have the cardinality bound $L = \exp\pth{O(n^{1/3}\log^{2.5}(nk))}$. Using this covering, the main term above can be bounded by
\begin{align*}
&\E\Big[\sum_{i=1}^k \frac{\big(\theta_{G_k}(N_i; \rho) - \theta_{\widehat{G}}(N_i; \rho)\big)^2}{N_i}\indc{y_0 < N_i \le M, H^2(f_{G_k},f_{\widehat{G}})\le \varepsilon^2}\Big] \\
&\le 2\E\Big[ \max_{\ell\in [L]}\sum_{i=1}^k \frac{\big(\theta_{G_k}(N_i; \rho) - \theta_{H_\ell}(N_i; \rho)\big)^2}{N_i}\indc{y_0 < N_i \le M} \Big] \\
&\qquad + 2\underbrace{\E\Big[ \min_{\ell\in [L]}\sum_{i=1}^k \frac{\big(\theta_{\widehat{G}}(N_i; \rho) - \theta_{H_\ell}(N_i; \rho)\big)^2}{N_i}\indc{y_0 < N_i \le M, H^2(f_{G_k},f_{\widehat{G}})\le \varepsilon^2} \Big]}_{\zeta_4}. 
\end{align*}

To proceed, we analyze the independent random variables $(Z_1^{(\ell)},\dots,Z_k^{(\ell)})$ with
\begin{align*}
    Z_i^{(\ell)} := \frac{\big(\theta_{G_k}(N_i; \rho) - \theta_{H_\ell}(N_i; \rho)\big)^2}{N_i}\indc{y_0 < N_i \le M}. 
\end{align*}
We will show that $Z_i^{(\ell)}$ is $(\sigma_{i,\ell}^2, b)$-subExponential, for some parameters $\sigma_{i,\ell}$ and $b$, so that standard tail bounds of Bernstein's inequality give 
\begin{align}\label{eq:subExp}
\E\qth{\max_{\ell\in [L]} \sum_{i=1}^k Z_i^{(\ell)} } = \max_{\ell\in [L]} \E\qth{\sum_{i=1}^k Z_i^{(\ell)} } + O\pth{ \max_{\ell\in [L]}\sqrt{\sum_{i=1}^k \sigma_{i,\ell}^2\log L} + b\log L }.
\end{align}
We deal with the terms in \eqbr{eq:subExp} separately. For the expectation, we have
\begin{align*}
\E\qth{\sum_{i=1}^k Z_i^{(\ell)} } &= \sum_{i=1}^k \sum_{y_0<y\le M}\Poi(y; \theta_i^\star)\frac{(\theta_{G_k}(y; \rho) - \theta_{H_\ell}(y;\rho) )^2}{y} \\
&\le 2\sum_{i=1}^k \sum_{y_0<y\le M}\Poi(y; \theta_i^\star) (y+1)\pth{ \frac{\Delta f_{G_k}(y)}{f_{G_k}(y)\vee\rho} - \frac{\Delta f_{H_\ell}(y)}{f_{H_\ell}(y)\vee\rho} }^2 \\
&\le 2k\sum_{y=0}^\infty (y+1)f_{G_k}(y)\pth{ \frac{\Delta f_{G_k}(y)}{f_{G_k}(y)\vee\rho} - \frac{\Delta f_{H_\ell}(y)}{f_{H_\ell}(y)\vee\rho} }^2 \\
&\stepc{=}O\pth{ k\log^4(1/\rho)H^2(f_{G_k}, f_{H_\ell}) + \rho^{10} } = O\pth{k\varepsilon^2 \log^4(nk)}, 
\end{align*}
where (c) is a functional inequality in Lemma \ref{lem:hellinger-to-regret}, and the last step uses the definition of a proper covering that $H^2(f_{G_k}, f_{H_\ell})\le \varepsilon^2$. For the norm parameter $b$, Lemma \ref{lem:bayes_deviation} clearly leads to a deterministic inequality
\begin{align*}
|Z_i^\ell| = O\pth{ \log^2\pth{\frac{1}{\rho}} }, 
\end{align*}
so that $b = O(\log^2(1/\rho))$. For the variance parameter $\sigma_{i,\ell}^2$, we combine the above observations to get
\begin{align*}
\sum_{i=1}^k \sigma_{i,\ell}^2 = \sum_{i=1}^k \var(Z_i^{(\ell)}) \le \sum_{i=1}^k \E\qth{\pth{Z_i^{(\ell)}}^2} = O\pth{ \log^2\pth{\frac{1}{\rho}} } \sum_{i=1}^k \E\qth{Z_i^{(\ell)}} = O(k\varepsilon^2 \log^6(nk)).
\end{align*}
Therefore, recalling that $L = \exp\pth{O(n^{1/3}\log^{2.5}(nk))}$, \eqbr{eq:subExp} leads to
\begin{align*}
&\E\Big[ \max_{\ell\in [L]}\sum_{i=1}^k \frac{\big(\theta_{G_k}(N_i; \rho) - \theta_{H_\ell}(N_i; \rho)\big)^2}{N_i}\indc{y_0 < N_i \le M} \Big] \\
&= O\pth{k\varepsilon^2 \log^4(nk) + \sqrt{k\varepsilon^2 \log^6(nk)\cdot n^{1/3}\log^{2.5}(nk)} + \log^2(nk)\cdot n^{1/3}\log^{2.5}(nk)}\\
&= O\pth{k\varepsilon^2 \log^4(nk) + n^{1/3}\log^{4.5}(nk) }. 
\end{align*}
In the case where $k$ is small, there is a better deterministic upper bound by Lemma \ref{lem:bayes_deviation}: 
\begin{align*}
\sum_{i=1}^k \frac{\big(\theta_{G_k}(N_i; \rho) - \theta_{H_\ell}(N_i; \rho)\big)^2}{N_i}\indc{y_0 < N_i \le M} = O(k\log^2(nk)). 
\end{align*}
In summary, the main term is upper bounded as
\begin{align*}
\E\Big[ \max_{\ell\in [L]}\sum_{i=1}^k \frac{\big(\theta_{G_k}(N_i; \rho) - \theta_{H_\ell}(N_i; \rho)\big)^2}{N_i}\indc{y_0 < N_i \le M} \Big] = O\pth{ k\varepsilon^2\log^4(nk) + (n^{1/3}\wedge k)\log^{4.5}(nk) }. 
\end{align*}

In the following we upper bound the remainder terms $\zeta_1$ to $\zeta_4$. 

\paragraph{Upper bounding $\zeta_1$.}  In addition, by Lemma \ref{lem:bayes_deviation} again, 
\begin{align*}
|\theta_{G_k}(y) - \theta_{G_k}(y; \rho)| &= (y+1)\frac{|\Delta f_{G_k}(y)|}{f_{G_k}(y)}\pth{1-\frac{f_{G_k}(y)}{\rho}}_+ \\
&= |\theta_{G_k}(y)-(y+1)|\pth{1-\frac{f_{G_k}(y)}{\rho}}_+ \\
&= O\pth{\sqrt{y+1} \log\pth{\frac{1}{f_{G_k}(y)}} \indc{f_{G_k}(y)\le \rho} }. 
\end{align*}
Therefore, by the choice of $\rho = c(nk)^{-5}$, 
\begin{align*}
\zeta_1 &\overset{\eqbr{eq:theta_Ghat_lb}}{=} O\pth{ \E\qth{\sum_{i=1}^k \log^2\pth{\frac{1}{f_{G_k}(N_i)}}\indc{N_i>y_0, f_{G_k}(N_i)\le \rho} } } \\
&= O\pth{\sum_{i=1}^k \sum_{y>y_0} \Poi(y; \theta_i^\star)\log^2\pth{\frac{1}{f_{G_k}(y)}}\indc{f_{G_k}(y)\le \rho} }  \\
&= O\pth{k\sum_{y>y_0} f_{G_k}(y)\log^2\pth{\frac{1}{f_{G_k}(y)}}\indc{f_{G_k}(y)\le \rho} } \\
&= O\pth{ k\sqrt{\rho}\log^2\pth{\frac{1}{\rho}}\sum_{y=0}^\infty \sqrt{f_{G_k}(y)} } \overset{\text{Lemma } \ref{lem:poi_mixture_sqrt}}{=} O\pth{k\sqrt{n\rho}\log^2\pth{\frac{1}{\rho}}} = O(1). 
\end{align*}

\paragraph{Upper bounding $\zeta_2$.} By Lemma \ref{lem:bayes_deviation}, we have a deterministic upper bound
\begin{align*}
\frac{\big(\theta_{G_k}(N_i; \rho) - \theta_{\widehat{G}}(N_i; \rho)\big)^2}{N_i} = O\pth{ \frac{\pth{\sqrt{N_i+1}\log(1/\rho)}^2}{N_i} } = O\pth{\log^2(nk)}
\end{align*}
when $N_i\ge y_0$. Now since $M = \lceil n^{2/3} \rceil$, 
\begin{align*}
\zeta_2 &= O\pth{\log^2(nk) \E\qth{\sum_{i=1}^k \indc{N_i>M} }} = O\pth{\log^2(nk)}\cdot \min\sth{k, \frac{\sum_{i=1}^k\E[N_i]}{M}} \\
&= O\pth{\pth{k\wedge n^{1/3}}\log^2(nk)}. 
\end{align*}

\paragraph{Upper bounding $\zeta_3$.} Similar to the upper bound of $\zeta_2$, we have
\begin{align*}
\zeta_3 = O\pth{\log^2(nk)}\cdot \P\pth{H^2(f_{G_k}, f_{\widehat{G}})\le \varepsilon^2} = O(\delta\log^2(nk)). 
\end{align*}

\paragraph{Upper bounding $\zeta_4$.} Since there exists $\ell\in [L]$ such that $\|\theta_{\widehat{G}}(\cdot; \rho) - \theta_{H_\ell}(\cdot; \rho)\|_{\infty,M}\le \eta$, it is clear that
\begin{align*}
\zeta_4 \le k\eta^2 = O\pth{\frac{1}{k}}. 
\end{align*}

Combining the above displays now proves \eqbr{eq:target_large_y}, and therefore
\begin{align}\label{eq:bound_A}
(\mathrm{A}) = O\pth{ \frac{k\varepsilon^2}{n}\log^4(nk) + \frac{n^{1/3}\wedge k}{n}\log^{4.5}(nk) + \frac{\delta\log^2(nk)}{n} }. 
\end{align}

\subsubsection*{Upper bound of $(\mathrm{B})$}
By Cauchy--Schwarz, 
\begin{align*}
    (\mathrm{B}) &= \E \Big[\sum_{i=1}^k  (\widehat{p}_i^{\PI} - \bar{p}_i^{\Sym}) \log \frac{\bar{p}^{\Sym}_i}{\bar{p}_i}\indc{N_i > y_0}\Big] \\
    &= \E \Big[\sum_{i=1}^k  (\widehat{p}_i^{\PI} - \bar{p}_i^{\Sym}) \log \frac{\theta_{G_k}(N_i; \rho)}{\theta_{\widehat{G}}(N_i; \rho)}\indc{N_i > y_0}\Big] +  \E \Big[\sum_{i=1}^k  (\widehat{p}_i^{\PI} - \bar{p}_i^{\Sym}) \log \frac{\theta_{G_k}(N_i)}{\theta_{G_k}(N_i; \rho)}\indc{N_i > y_0}\Big] \\
    &\le \E^{1/2}\qth{ \sum_{i=1}^k \frac{(\widehat{p}_i^{\PI} - \bar{p}_i^{\Sym} )^2}{ \theta_{G_k}(N_i; \rho) }\indc{N_i > y_0} }\cdot \Bigg( \E^{1/2}\qth{ \sum_{i=1}^k\theta_{G_k}(N_i; \rho)\log^2 \frac{\theta_{G_k}(N_i; \rho)}{\theta_{\widehat{G}}(N_i; \rho)}\indc{N_i > y_0} }\\
    &\qquad  + \E^{1/2}\qth{ \sum_{i=1}^k\theta_{G_k}(N_i; \rho)\log^2 \frac{\theta_{G_k}(N_i)}{\theta_{G_k}(N_i; \rho)} \indc{N_i > y_0}} \Bigg) \\
    &=: \sqrt{(\mathrm{B}_1)}\pth{\sqrt{(\mathrm{B}_2)} + \sqrt{(\mathrm{B}_3)}}. 
\end{align*}
We proceed to upper bound $(\mathrm{B}_1)$ to $(\mathrm{B}_3)$, respectively. 

\paragraph{Upper bounding $(\mathrm{B}_1)$.} Similar to \eqbr{eq:theta_Ghat_lb}, it holds that $\theta_{G_k}(N_i; \rho)\ge N_i/2$ when $N_i > y_0$. Therefore, using the new data generating scheme below \eqbr{eq:equivalence} and the expressions $\widehat{p}_i^{\PI}=\E\qth{\theta_i | N}$ and $\bar{p}_i^{\Sym} = \E\qth{\theta_i | N_i}$ in \eqbr{eq:PI_oracle_intro} and \eqbr{eq:separable_oracle_intro}\ifthenelse{\boolean{arxiv}}{}{ of the main text}, we can upper bound $(\mathrm{B}_1)$ as
\begin{align*}
(\mathrm{B}_1) &\le 2\E\qth{ \sum_{i=1}^k \frac{(\widehat{p}_i^{\PI} - \bar{p}_i^{\Sym} )^2}{ N_i }\indc{N_i > y_0} } \\
&= \frac{2}{n^2}\E \qth{ \sum_{i=1}^k \frac{(\E\qth{\theta_i | N} - \E\qth{\theta_i | N_i} )^2}{ N_i }\indc{N_i > y_0} } \\
&= O\pth{\frac{n^{1/3}\wedge k}{n^2}\log^2 n + \frac{\log^3 n}{n^2}},
\end{align*}
where the last step uses Lemma \ref{lemma:main_inequality} in the analysis of the permutation-invariant oracle. 
\paragraph{Upper bounding $(\mathrm{B}_2)$.} Similar to \eqbr{eq:theta_Ghat_lb}, we have $\theta_{G_k}(N_i; \rho), \theta_{\widehat{G}}(N_i; \rho)\in [N_i/2, 2N_i]$ if $N_i > y_0$. Since $|\log x|\le 2|1-1/x|$ for $x\in [1/4,4]$, we obtain
\begin{align*}
(\mathrm{B}_2) &\le 4\E\Big[\sum_{i=1}^k \frac{\big(\theta_{G_k}(N_i; \rho) - \theta_{\widehat{G}}(N_i; \rho)\big)^2}{\theta_{\widehat{G}}(N_i; \rho)}\indc{N_i > y_0}\Big] \\
&\overset{\eqbr{eq:target_large_y}}{=} O\pth{ k\varepsilon^2\log^4(nk) + (n^{1/3}\wedge k)\log^{4.5}(nk) + \delta\log^2(nk) }.
\end{align*}
\paragraph{Upper bounding $(\mathrm{B}_3)$.} Again, we have $\theta_{G_k}(N_i; \rho)\in [N_i/2, 2N_i]$ for $N_i > y_0$, and
\begin{align*}
n\ge \theta_{G_k}(N_i) \overset{\eqbr{eq:sep_lower}}{\ge} \frac{1}{ne^n}\sum_{j=1}^k (\theta_j^\star)^2 \ge \frac{1}{nke^n} \Big(\sum_{j=1}^k \theta_j^\star\Big)^2 = \frac{n}{ke^n}.
\end{align*}
Since $\theta_{G_k}(N_i; \rho) = \theta_{G_k}(N_i)$ as long as $f_{G_k}(N_i)\ge \rho$, we get
\begin{align*}
(\mathrm{B}_3) &= O\pth{ \E\qth{\sum_{i=1}^k N_i\log^2 (ke^n) \indc{f_{G_k}(N_i)<\rho} } } \\
&= O\pth{ n^2+\log^2 k }\cdot \sum_{i=1}^k\sum_{y=0}^\infty y\indc{f_{G_k}(y)<\rho}\Poi(y; \theta_i^\star)\\
&= O\pth{ n^2+\log^2 k }\Big(\sum_{y=0}^{2n} 2knf_{G_k}(y)\indc{f_{G_k}(y)<\rho} + \sum_{i=1}^k \sum_{y>2n} y\Poi(y; \theta_i^\star)   \Big) \\
&= O\pth{ n^2+\log^2 k }\Big(\sum_{y=0}^{2n} 2knf_{G_k}(y)\indc{f_{G_k}(y)<\rho} + \sum_{i=1}^k \sum_{y>2n} \theta_i^\star \Poi(y-1; \theta_i^\star)   \Big) \\
&= O\pth{(n^2+\log^2 k)  (kn^2\rho + ne^{-n/3}) } = O(\log^2(nk)), 
\end{align*}
where in the last line we have used the Chernoff bound (cf. Lemma \ref{lemma:chernoff}), the identity $\sum_{i=1}^k \theta_i^\star = n$, and the choice of $\rho = c(nk)^{-5}$. 

Finally, a combination of the above bounds yields
\begin{align}\label{eq:bound_B}
    (\mathrm{B}) = O\pth{ \frac{k\varepsilon^2}{n}\log^4(nk) + \frac{n^{1/3}\wedge k}{n}\log^{4.5}(nk) + \frac{\delta\log^2(nk)}{n} }, 
\end{align}
by the AM-GM inequality $\sqrt{ab}\le \frac{na}{2}+\frac{b}{2n}$. Combining the bounds \eqbr{eq:bound_A} and \eqbr{eq:bound_B} completes the proof. \qed

\begin{lemma}[Lemma 29 of \cite{shen2022empirical}]\label{lem:bayes_rule_cover}
For any $\rho > 0$, let $\Theta_0(\rho) \equiv \{\theta_{G}(\cdot; \rho): G\subseteq\mathcal{P}(\R_+)\}$ be the set of all $\rho$-regularized Bayes forms. For any $\theta_{G}(\cdot;\rho), \theta_{H}(\cdot;\rho)\in\Theta_0(\rho)$ and $M > 0$, let
\begin{align*}
\pnorm{\theta_{G}(\cdot;\rho) - \theta_{H}(\cdot,\rho)}{\infty,M} := \max_{y=0,\dots,M} \big|\theta_{G}(y;\rho) - \theta_{H}(y;\rho)\big|.
\end{align*}
Then for any $\eta \in (0,10^{-3})$ and $M \geq \log^\kappa(1/\rho\eta)$ for some sufficiently large $\kappa > 0$, there exists some universal constant $C > 0$ such that
\begin{align*}
\log \mathcal{N}(\eta, \Theta_0(\rho), \pnorm{\cdot}{\infty, M}) \leq C\sqrt{M}\log^{5/2}(M/\rho\eta).
\end{align*}
\end{lemma}

\begin{lemma}\label{lem:hellinger-to-regret}
Let $\rho\in (0,1)$, and $G, H$ be two priors. There is a universal constant $C>0$ such that
\begin{align*}
\sum_{y=0}^\infty (y+1)f_{G}(y)\pth{ \frac{\Delta f_{G}(y)}{f_{G}(y)\vee\rho} - \frac{\Delta f_{H}(y)}{f_{H}(y)\vee\rho} }^2 \le C\pth{\log^4(1/\rho) \cdot H^2(f_G,f_H) + \rho^{10}}. 
\end{align*}
\end{lemma}
\begin{proof}
The argument is essentially similar to Eq. (6.21)--(6.23) in \cite{shen2022empirical}, and we include a proof for completeness. By triangle inequality, we have
\begin{align*}
&\sum_{y=0}^\infty (y+1)f_{G}(y)\pth{ \frac{\Delta f_{G}(y)}{f_{G}(y)\vee\rho} - \frac{\Delta f_{H}(y)}{f_{H}(y)\vee\rho} }^2 \\
&\le 3\sum_{y=0}^\infty (y+1)f_{G}(y)\pth{ \frac{\Delta f_{G}(y)}{f_{G}(y)\vee\rho} - \frac{2\Delta f_{G}(y)}{f_{G}(y)\vee\rho+f_{H}(y)\vee\rho} }^2 \\
&\qquad + 3\sum_{y=0}^\infty (y+1)f_{G}(y)\pth{ \frac{2(\Delta f_G(y) - \Delta f_{H}(y))}{f_{G}(y)\vee\rho+f_{H}(y)\vee\rho} }^2  \\
&\qquad + 3\sum_{y=0}^\infty (y+1)f_{G}(y)\pth{ \frac{\Delta f_{H}(y)}{f_{H}(y)\vee\rho} - \frac{2\Delta f_{H}(y)}{f_{G}(y)\vee\rho+f_{H}(y)\vee\rho} }^2 \\
&=: 3(R_1 + R_2 + R_3).
\end{align*}
We first upper bound $R_1$ as
\begin{align*}
    R_1 &= \sum_{y=0}^\infty \pth{\sqrt{y+1} \frac{\Delta f_{G}(y)}{f_{G}(y)\vee\rho} }^2\cdot f_G(y)\pth{\frac{f_{G}(y)\vee\rho-f_{H}(y)\vee\rho}{f_{G}(y)\vee\rho+f_{H}(y)\vee\rho}}^2 \\
    &= O\Big( \log^2\pth{\frac{1}{\rho}} \sum_{y=0}^\infty \pth{\sqrt{f_G(y)} - \sqrt{f_H(y)}}^2 \Big) = O\pth{\log^2(1/\rho) H^2(f_G,f_H)}, 
\end{align*}
where the middle step uses Lemma \ref{lem:bayes_deviation} and that for $a,b\ge 0$,
\begin{align*}
a\pth{\frac{a\vee \rho - b\vee\rho}{a\vee\rho+b\vee\rho}}^2 \le \frac{(a\vee \rho - b\vee\rho)^2}{a\vee\rho+b\vee\rho} \le 2\pth{\sqrt{a\vee \rho} - \sqrt{b\vee \rho}}^2 \le 2\pth{\sqrt{a}-\sqrt{b}}^2. 
\end{align*}
Similarly, we obtain the same upper bound of $R_3$. As for $R_2$, we have
\begin{align*}
R_2\le 4\sum_{y=0}^\infty (y+1)\frac{(\Delta f_G(y)-\Delta f_H(y))^2}{f_G(y)\vee\rho + f_H(y)\vee\rho} = O\pth{ \log^4(1/\rho)H^2(f_G,f_H) + \rho^{10} }, 
\end{align*}
where the last step uses the crucial inequality in \cite[Proposition 14]{shen2022empirical}.
\end{proof}

\begin{lemma}\label{lem:poi_mixture_sqrt}
For any distribution such that $\supp(G) \subseteq [0,K]$ for some $K \ge 1$, there exists some universal $C > 0$ such that $\sum_{y=0}^\infty \sqrt{f_G(y)} \leq C\sqrt{K}$. 
\end{lemma}
\begin{proof}
By the Chernoff bound (cf. Lemma \ref{lemma:chernoff}), 
\begin{align*}
\sum_{y=0}^\infty \sqrt{f_G(y)} &= \sum_{y=0}^{2K} \sqrt{f_G(y)} + \sum_{y=2K+1}^{\infty} \sqrt{f_G(y)} \le \sqrt{K} + \sum_{y = 2K + 1}^\infty\Big(\Prob\big(\poi(K) \geq y\big)\Big)^{1/2}\\
&\le \sqrt{K} + \sum_{y=2K+1}^\infty \exp\pth{-\frac{y-K}{6}} = O\pth{\sqrt{K}}.
\end{align*}
\end{proof}

\section{Proof of Theorem \ref{thm:upper_bound_general} (Part 2)}\label{sec:main_proof_PI}

In this section we prove Proposition \ref{prop:reg_PI}, which is the last missing piece in the proof of Theorem \ref{thm:upper_bound_general}. In fact, using the same argument of \eqbr{eq:equivalence}, Proposition \ref{prop:reg_PI} follows readily from \Cref{thm:PI-Poisson}\ifthenelse{\boolean{arxiv}}{}{ in the main text}: 
\begin{align*}
    (\mathrm{II}) = \E\qth{\sum_{i=1}^k p^\star_i \log \frac{\widehat{p}^{\PI}_i}{\bar{p}^{\Sym}_i} - \widehat{p}^{\PI}_i + \bar{p}^{\Sym}_i} = \E\qth{\sum_{i=1}^k \widehat{p}^{\PI}_i \log \frac{\widehat{p}^{\PI}_i}{\bar{p}^{\Sym}_i} - \widehat{p}^{\PI}_i + \bar{p}^{\Sym}_i} = O\pth{\frac{k\wedge n^{1/3}}{n}\log^2 n + \frac{\log^3 n}{n}}. 
\end{align*}
In the remainder of this section, we will prove \Cref{thm:PI-Poisson}. Before that, we will use the Gaussian compound decision problem as a warm-up example to illustrate the proof ideas and compare to existing results in the literature. 

\subsection{Gaussian compound decision problem}
Before describing our results in the Poisson model, we first summarize the existing results and our improvements for the popular Gaussian model which are easier to state. A state-of-the-art result by Greenshtein and Ritov \cite{greenshtein2009asymptotic} examined an analogous setting in the Gaussian location model: let $\theta^\star\in \R^n$ be a deterministic vector, and let $\theta = \pi \circ \theta^\star$ be a random permutation of $\theta^\star$, with $\pi\sim \mathrm{Unif}(S_k)$. Conditioned on $\theta$, the observation vector in the Gaussian location model is $X\sim \calN(\theta,I)$. Under the squared error loss, it is known \cite{green1993nonparametric} that the PI and separable oracles are given by
\begin{equation*}
    \widehat{\theta}^{\PI}_i = \E[\theta_i | X] \quad \text{ and } \quad \widehat{\theta}^{\Sym}_i = \E[\theta_i | X_i],
\end{equation*}
respectively. The main result of \cite{greenshtein2009asymptotic} claims that, if $|\theta_i|\le \mu$ for all $i\in [n]$ and $\mu\ge 1$, then
\begin{align}\label{eq:GR-result}
    \E \qth{ \| \widehat{\theta}^{\PI} - \widehat{\theta}^{\Sym} \|_2^2 } \le e^{C\mu^2},
\end{align}
for some universal constant $C>0$. Notably, this bound does \emph{not} depend on the dimension $n$, thereby validating the mean-field approximation for the PI oracle. However, this bound grows very rapidly with $\mu$ and ceases to be informative for moderately large $\mu$, say $\mu \gg \sqrt{\log n}$. Our following result establishes an improved upper bound when $\mu \gg 1$.


\begin{theorem}\label{thm:PIoracle_Gaussian}
Under the Gaussian setting, if $\mu\ge 1$ and $|\theta_i^\star|\le \mu$ for all $i\in [n]$, then
\begin{align*}
\E\qth{\|\widehat{\theta}^{\mathrm{S}} - \widehat{\theta}^{\mathrm{PI}}\|_2^2} \le \min\sth{C\mu \log^2 n, n}, 
\end{align*}
where $C<\infty$ is an absolute constant. 
\end{theorem}

Compared with \eqbr{eq:GR-result}, \Cref{thm:PIoracle_Gaussian} improves the dependence on $\mu$ from exponential to linear. While the logarithmic factors in $n$ might be an artifact of the analysis, the linear dependence on $\mu$ turns out to be tight in Gaussian location models, as shown in the following result. 

\begin{theorem}\label{prop:PIoracle_Gaussian_LB}
There exist absolute constants $n_0, c>0$ such that, for $n\ge n_0$ and $\mu \ge 8\sqrt{\log n}$, there exist $\theta_1^\star, \dots, \theta_n^\star \in [-\mu, \mu]$ such that
\begin{align*}
\E\qth{\|\widehat{\theta}^{\mathrm{S}} - \widehat{\theta}^{\mathrm{PI}}\|_2^2} \ge c\min\sth{\frac{\mu}{\sqrt{\log n}}, n}. 
\end{align*}
\end{theorem}

Before we present the proofs of \Cref{thm:PIoracle_Gaussian} and \Cref{prop:PIoracle_Gaussian_LB}, we explain the intuition behind the linear dependence on $\mu$. Intuitively, the quantity $\mu$ represents the ``effective number of clusters'': the interval $[-\mu,\mu]$ can be partitioned into $m=\widetilde{O}(\mu)$ subintervals $I_1,\dots,I_m$, each of length $\widetilde{O}(1)$. By the concentration of normal random variables, based on the observation $X_i$, with high probability one can locate $\theta_i$ within an interval of range $\widetilde{O}(1)$. Therefore, even if we additionally provide the membership information for each $\theta_i$, i.e. the index $j(i)\in [m]$ where $\theta_i \in I_{j(i)}$, both the symmetric and permutation-invariant oracles only change a little in view of this additional piece of knowledge. However, conditioned on the memberships $j(1), \dots, j(n)$, one effectively arrives at a particular subgroup of the permutation group $S_n$, where the random permutation $\pi$ now follows a product distribution $\pi \sim \otimes_{j=1}^m \mathrm{Unif}(S_{A_j}). $ Here $S_A$ is the symmetric group over a finite set $A$, and $A_j := \sth{i\in [n]: j(i)=j}$. Thanks to the product structure, the problem of upper bounding the $\ell^2$-distance $\|\widehat{\theta}^{\mathrm{S}} - \widehat{\theta}^{\mathrm{PI}}\|_2^2$ is now a \emph{direct sum} of $m$ \emph{independent} subproblems of upper bounding $\|\widehat{\theta}_{A_j}^{\mathrm{S}} - \widehat{\theta}_{A_j}^{\mathrm{PI}}\|_2^2$ for each $j\in [m]$, where $\theta_A$ is the restriction of $\theta$ to the index set $A$. Consequently, it is natural to expect that the target quantity is linear in the number of subproblems, which explains the linear dependence on $m=\widetilde{O}(\mu)$. Technically, \Cref{prop:PIoracle_Gaussian_LB} shows the easy direction that each subproblem can contribute $\widetilde{\Omega}(1)$ to the $\ell^2$-distance, while \Cref{thm:PIoracle_Gaussian} establishes the hard direction that the contribution from each subproblem is at most $\widetilde{O}(1)$.

\subsubsection{Proof of \Cref{thm:PIoracle_Gaussian}}
We first prove the easy upper bound $\E\qth{\|\widehat{\theta}^{\mathrm{S}} - \widehat{\theta}^{\mathrm{PI}}\|_2^2} \le n$. In fact,
\begin{align*}
\E\qth{\|\widehat{\theta}^{\mathrm{S}} - \widehat{\theta}^{\mathrm{PI}}\|_2^2} &\stepa{=} \E\qth{\|\widehat{\theta}^{\mathrm{S}} - \theta\|_2^2} - \E\qth{\|\widehat{\theta}^{\mathrm{PI}} - \theta\|_2^2} \le \E\qth{\|\widehat{\theta}^{\mathrm{S}} - \theta\|_2^2} \stepb{\le} \E\qth{\|X - \theta\|_2^2} = n, 
\end{align*}
where (a) is a simple orthogonality relation established in \cite{greenshtein2009asymptotic}, and (b) uses the defining property that the separable oracle $\widehat{\theta}^{\mathrm{S}}$ achieves the smallest MSE among all separable decision rules, which contain the decision rule $\widehat{\theta}(X)=X$. The second upper bound is more challenging, and we break the proof into three steps.

\paragraph{Step I: introducing an auxiliary vector $Z$.} The first, and perhaps a bit surprising, step is to make use of the divisibility of the normal distribution and introduce an auxiliary random vector $Z$ to interpolate between $\theta$ and $X$. Specifically, let $W, W'\sim \calN(0, \frac{1}{2}I_n)$ be independent random vectors, we define
\begin{align*}
Z = \theta + W, \qquad X = \theta + W + W'. 
\end{align*}
Note that the joint distribution of $(\theta, X)$ remains the same, and the auxiliary vector $Z$ interpolates between $\theta$ and $X$. Let $\{(X_i,\theta_i,W_i,W_i',Z_i)\}_{i=1}^n$ be iid observations generated from the above model. The main identity we prove in this step is as follows: 
\begin{align}\label{eq:interpolation_Gaussian}
\E\qth{Z_1 | X_1} - \E\qth{Z_1 | X} = \frac{1}{2}\pth{ \E\qth{\theta_1 | X_1} - \E\qth{\theta_1 | X} }. 
\end{align}
To see \eqbr{eq:interpolation_Gaussian}, note that
\begin{align*}
\E\qth{Z_1 | X_1} &= \E\qth{\theta_1 + W_1 | \theta_1 + W_1 + W_1'} \\
&\stepa{=} \frac{\E\qth{\theta_1 + W_1 | \theta_1 + W_1 + W_1'} + \E\qth{\theta_1 + W_1' | \theta_1 + W_1' + W_1}}{2} \\
&=  \frac{\E\qth{2\theta_1 + W_1 + W_1' | \theta_1 + W_1 + W_1'}}{2} = \frac{\E\qth{\theta_1|X_1} + X_1}{2}, 
\end{align*}
where (a) follows from the exchangeability $(\theta, W, W') \overset{\textrm{d}}{=} (\theta, W', W)$. An entirely similar argument gives
\begin{align*}
\E\qth{Z_1 | X} = \frac{\E\qth{\theta_1 | X} + X_1}{2}, 
\end{align*}
and a combination of the above displays completes the proof of \eqbr{eq:interpolation_Gaussian}. From now on we drop $\theta$ and work on the joint distribution of $(Z,X)$ instead, where the key benefit of replacing $\theta$ by $Z$ will be apparent in the third step below. 

\paragraph{Step II: an information-theoretic upper bound.} The second step is to relate the mean squared error to information-theoretic quantities such as the mutual information and KL divergence. This step is similar to the analysis in \cite[Lemma 23]{nie2023large}, with some extra steps to handle the concentration. By \eqbr{eq:interpolation_Gaussian} and symmetry, 
\begin{align}
\E\qth{\|\widehat{\theta}^{\mathrm{S}} - \widehat{\theta}^{\mathrm{PI}}\|_2^2} &= n\E\qth{\pth{\E\qth{\theta_1|X_1} - \E\qth{\theta_1|X} }^2} \nonumber \\
&= 4n\E\qth{\pth{\E\qth{Z_1|X_1} - \E\qth{Z_1|X} }^2} \nonumber\\
&= 4n\E\qth{\pth{\E\qth{Z_1-X_1|X_1} - \E\qth{Z_1-X_1|X} }^2}\nonumber\\
&\le 12n\E\qth{\pth{\E\qth{(Z_1-X_1)\indc{E}|X_1} - \E\qth{(Z_1-X_1)\indc{E}|X} }^2} \nonumber\\
&\qquad + 12n\E\qth{\pth{\E\qth{(Z_1-X_1)\indc{E^c}|X_1}}^2} + 12n \E\qth{\pth{\E\qth{(Z_1-X_1)\indc{E^c}|X}}^2}, \label{eq:Gaussian_triangle}
\end{align}
where the last step is due to the triangle inequality $(a+b+c)^2\le 3(a^2+b^2+c^2)$, and we define the event $E:=\sth{ |Z_1 - X_1|\le \sqrt{2\log n} }$. For the last two terms, note that
\begin{align*}
\P(E^c) = \P\pth{|W_1'|> \sqrt{2\log n}} = \P\pth{|\calN(0,1/2)|> \sqrt{2\log n}} \le 2\exp\pth{-\frac{(\sqrt{2\log n})^2}{2\cdot \frac{1}{2}}} = \frac{2}{n^2}, 
\end{align*}
and by repeated applications of Cauchy--Schwarz, 
\begin{align*}
\E\qth{\pth{\E\qth{(Z_1-X_1)\indc{E^c}|X_1}}^2} \le \E\qth{(Z_1-X_1)^2 \indc{E^c}} \le \E^{1/2}\qth{W'^4} \P^{1/2}(E^c) \le \frac{2}{n}. 
\end{align*}
A similar inequality also holds for the last term of \eqbr{eq:Gaussian_triangle}. Therefore, we continue \eqbr{eq:Gaussian_triangle} as
\begin{align*}
\E\qth{\|\widehat{\theta}^{\mathrm{S}} - \widehat{\theta}^{\mathrm{PI}}\|_2^2} &\le 12n\E\qth{\pth{\E\qth{(Z_1-X_1)\indc{E}|X_1} - \E\qth{(Z_1-X_1)\indc{E}|X} }^2} + 48 \\
&\stepa{\le} 12n\cdot 2(\sqrt{2\log n})^2 I((Z_1-X_1)\indc{E}; X_2^n | X_1) + 48 \\
&\stepb{\le} 48n\log n\cdot I(Z_1; X_2^n | X_1) + 48,   
\end{align*}
where (a) is due to Tao's inequality (a consequence of Pinsker's inequality, cf. e.g. \cite[Corollary 7.11]{polyanskiy2024information}) and the upper bound $|(Z_1-X_1)\indc{E}|\le \sqrt{2\log n}$, and (b) uses the data-processing inequality that given $X_1$, the event $\indc{E}$ is determined by $Z_1$. The mutual information is then upper bounded as
\begin{align*}
I(Z_1; X_2^n | X_1) &= h(Z_1|X_1) - h(Z_1 | X) \\
&\stepc{\le} h(Z_1|X_1) - \frac{h(Z|X)}{n} \\
&= h(Z_1) - \frac{h(Z)}{n} - \pth{I(Z_1;X_1) - \frac{I(Z;X)}{n}} \\
&\stepd{\le} h(Z_1) - \frac{h(Z)}{n} = \frac{1}{n}\DKL\pth{ P_Z \bigg\| \prod_{i=1}^n P_{Z_i} }, 
\end{align*}
where $h(\cdot)$ denotes the differential entropy (i.e., $h(Q) = -\int q\log q$ if $Q$ has density $q$), (c) follows from the subadditivity of entropy and symmetry $h(Z|X) \le \sum_{i=1}^n h(Z_i|X) = nh(Z_1|X)$, and (d) is the subadditivity of mutual information $I(Z;X) \le \sum_{i=1}^n I(Z_i; X_i) = nI(Z_1;X_1)$ when $P_{X|Z} = \prod_{i=1}^n P_{X_i|Z_i}$. In summary,
\begin{align}\label{eq:Gaussian_stepII}
\E\qth{\|\widehat{\theta}^{\mathrm{S}} - \widehat{\theta}^{\mathrm{PI}}\|_2^2} \le 48\log n\cdot \DKL\pth{ P_Z \bigg\| \prod_{i=1}^n P_{Z_i} } + 48. 
\end{align}

\paragraph{Step III: approximate independence of noisy permutation mixtures.} The final step is an upper bound on the KL divergence between the joint distribution $P_Z$ and the product $\prod_{i=1}^n P_{Z_i}$ of the marginal distributions in \eqbr{eq:Gaussian_stepII}. Before we proceed, we emphasize the significance of replacing $\theta$ by $Z$ in Step I: without the replacement, we would arrive at the KL divergence between $P_{\theta}$ and $\prod_{i=1}^n P_{\theta_i}$ which could be as large as $\Omega(n)$ in general. By contrast, the auxiliary vector $Z$ is a \emph{noisy} version of $\theta$, and recent advances \cite{tang2023capacity,han2024approximate} show that the \emph{noisy permutation mixture} $P_Z$ is quantitatively close to its i.i.d.~approximation. This phenomenon was first observed in \cite{tang2023capacity} assuming that 
(a) the number of distinct $\theta_i$'s are bounded, i.e.,  $|\{\theta_1, \dots, \theta_n\}| = O(1)$;
and (b) the channel $P_{Z|\theta}$ is discrete; both conditions are significantly relaxed in \cite{han2024approximate} through the development of new proof techniques. For instance, for the Gaussian channel, \cite[Corollary 1.3]{han2024approximate} shows that
\begin{align*}
\DKL\pth{ P_Z \bigg\| \prod_{i=1}^n P_{Z_i} } \le \log\pth{1+\chi^2\pth{ P_Z \bigg\| \prod_{i=1}^n P_{Z_i} } } = O(\mu^3), 
\end{align*}
which already provides an upper bound polynomial in $\mu$. However, to obtain the linear dependence on $\mu$, we need some extra steps and use the chain rule of the KL divergence. 

To this end, observe that $[-\mu,\mu]\subseteq \cup_{j=1}^m I_j$, where $m=\lceil 2\mu \rceil$, and
\begin{align*}
    I_j = [-\mu + j-1, -\mu+j). 
\end{align*}
Now define an auxiliary vector $J$ in addition to $(\theta,Z,X)$, where $J_i \in [m]$ is the index of the interval such that $\theta_i\in I_{J_i}$. Therefore, 
\begin{align}\label{eq:chain_rule}
\DKL\pth{ P_Z \bigg\| \prod_{i=1}^n P_{Z_i} } &\le \DKL\pth{ P_{J,Z} \bigg\| \prod_{i=1}^n P_{J_i,Z_i} }  \nonumber\\
&= \DKL\pth{ P_{J} \bigg\| \prod_{i=1}^n P_{J_i} } + \E_J\qth{ \DKL\pth{ P_{Z|J} \bigg\| \prod_{i=1}^n P_{Z_i|J_i} } }, 
\end{align}
by the chain rule of the KL divergence. We bound the first term of \eqbr{eq:chain_rule} via a method-of-types argument: For $j\in [m]$, let $n_j = \sum_{i=1}^n \indc{J_i = j} = \sum_{i=1}^n \indc{\theta_i^\star\in I_j}$ be the number of coordinates of $\theta^\star$ lying in $I_j$. Since the vector $\theta^\star$ is deterministic, $(n_1,\dots,n_m)$ is also a deterministic vector. Then $P_J$ is the uniform distribution over $[m]^n$ with a given type $(n_1, \dots, n_m)$ (i.e. $1$ appears $n_1$ times, $2$ appears $n_2$ times, etc), while each $P_{J_i}$ is the discrete distribution $\pth{\frac{n_1}{n}, \dots, \frac{n_m}{n}}$. Therefore,
\begin{align*}
\DKL\pth{ P_{J} \bigg\| \prod_{i=1}^n P_{J_i} }  = nH(J_1) - H(J) = \sum_{j=1}^m n_j\log \frac{n}{n_j} - \log \binom{n}{n_1, \dots, n_m}. 
\end{align*}
By Stirling's approximation $\sqrt{2\pi n}\pth{\frac{n}{e}}^n \le n!\le 2\sqrt{2\pi n}\pth{\frac{n}{e}}^n$ and the AM-GM inequality, we have
\begin{align*}
\log \binom{n}{n_1, \dots, n_m} &= \log(n!) - \sum_{j=1}^m \log(n_j!) \\
&\ge n\log\frac{n}{e} + \frac{1}{2}\log(2\pi n) - \sum_{j=1}^m \pth{ n_j\log\frac{n_j}{e} + \frac{1}{2}\log(8\pi n_j) } \\
&= \sum_{j=1}^m n_j\log \frac{n}{n_j} - \frac{m\log(8\pi) - \log(2\pi)}{2} - \frac{1}{2}\log\frac{\prod_{j=1}^m n_j}{n} \\
&\ge \sum_{j=1}^m n_j\log \frac{n}{n_j} - \frac{m\log(8\pi) - \log(2\pi)}{2} - \frac{1}{2}\log\frac{(n/m)^m}{n}, 
\end{align*}
we conclude that
\begin{align}\label{eq:Gaussian_first_term}
    \DKL\pth{ P_{J} \bigg\| \prod_{i=1}^n P_{J_i} }  = O(m\log n). 
\end{align}

As for the second term of \eqbr{eq:chain_rule}, we fix any realization of $J$, which must be of type $(n_1, \dots, n_m)$. For $j\in [m]$, let $A_j = \sth{i\in [n]: J_i = j}$ be the set of coordinates of $\theta$ lying in $I_j$, with $|A_j|=n_j$. We note that
\begin{align*}
P_{Z|J}(z) = \prod_{j=1}^m P_j(z_{A_j}), 
\end{align*}
where $z_{A_j} = (z_i: i\in A_j)$, and $P_j$ is the joint distribution of $z_{A_j}\sim \calN(\theta_{A_j}, \frac{1}{2}I_{n_j})$ where $\theta_{A_j}$ is a uniformly random permutation of $\sth{\theta_1^\star, \dots, \theta_n^\star}\cap I_j$ (counted with multiplicity). Similarly, 
\begin{align*}
\prod_{i=1}^n P_{Z_i | J_i}(z_i) = \prod_{j=1}^m Q_j(z_{A_j}), 
\end{align*}
where $Q_j$ is an i.i.d.~product with the marginal distribution being $\E\qth{\calN(\theta,\frac{1}{2})}$, with the scalar mean $\theta$ uniformly distributed in $\sth{\theta_1^\star, \dots, \theta_n^\star}\cap I_j$ (counted with multiplicity). In other words, given a set of distributions $\sth{\calN(\theta_i^\star, \frac{1}{2}): \theta_i^\star\in I_j}$, the distributions $P_j$ and $Q_j$ are the permutation mixture and its i.i.d.~approximation (in the language of \cite{han2024approximate}), respectively. Since
\begin{align*}
\chi^2\pth{ \calN\pth{u, \frac{1}{2}} \bigg\| \calN\pth{v, \frac{1}{2}} } = e^{2(u-v)^2} - 1 \le e^2 - 1
\end{align*}
for all $u, v\in I_j$, we conclude from \cite[Theorem 1.2]{han2024approximate} that $\DKL(P_j \| Q_j) \le \chi^2(P_j \| Q_j) = O(1) $ for every $j\in [m]$. Consequently, for any realization of $J$, 
\begin{align}\label{eq:Gaussian_second_term}
    \DKL\pth{ P_{Z|J} \bigg\| \prod_{i=1}^n P_{Z_i|J_i} } = \sum_{j=1}^m \DKL\pth{P_j \| Q_j} = O(m). 
\end{align}
Finally, combining \eqbr{eq:chain_rule}, \eqbr{eq:Gaussian_first_term}, \eqbr{eq:Gaussian_second_term}, and the choice of $m=\lceil 2\mu \rceil$ yields 
\begin{align*}
\DKL\pth{ P_Z \bigg\| \prod_{i=1}^n P_{Z_i} } = O(\mu \log n), 
\end{align*}
which combined with \eqbr{eq:Gaussian_stepII} proves \Cref{thm:PIoracle_Gaussian}.  \qed

\subsubsection{Proof of \Cref{prop:PIoracle_Gaussian_LB}}
Define a positive integer
\begin{align*}
    m = \left\lfloor \min\sth{ \frac{n}{2} , 1 + \frac{\mu-1}{4\sqrt{\log n}} }  \right\rfloor, 
\end{align*}
and set
\begin{align*}
\theta_{2i-1}^\star &= -1 + 8(i-1)\sqrt{\log n}, \quad \theta_{2i}^\star = 1 + 8(i-1)\sqrt{\log n}, \qquad i\in [m], \\
\theta_{2m+1}^\star &= \dots = \theta_n^\star = -8\sqrt{\log n}. 
\end{align*}
We also define $(m+1)$ intervals $I_0, \dots, I_m$ as follows: 
\begin{align*}
I_0 &= (-\infty, -4\sqrt{\log n}),\\
I_j &= [4(2j-3)\sqrt{\log n}, 4(2j-1)\sqrt{\log n}), \quad j\in [m-1], \\ 
I_m &= [4(2m-3)\sqrt{\log n}, \infty).
\end{align*}
Thanks to the assumption $\mu\ge 8\sqrt{\log n}$, these parameters are constructed in a way such that $I_0, \dots, I_m$ constitute a partition of $\mathbb{R}$, $\{\theta_{2j-1}^\star, \theta_{2j}^\star\} \subseteq I_j$ for $j\in [m]$, and $\theta_{2m+1}^\star, \dots, \theta_n^\star\in I_0$. In addition, $(\theta_{2j-1}^\star, \theta_{2j}^\star)$ for $j\in [m]$ are different translations of $(-1,1)$.  

To analyze the oracles in this scenario, we introduce auxiliary random variables $J_1, \dots, J_n\in [m]\cup\{0\}$ such that $\theta_i\in I_{J_i}$ for each $i\in [n]$. In other words, this is the membership information of the unknown vector $\theta$. We argue that the oracles $\E\qth{\theta_1 | X_1}$ and $\E\qth{\theta_1 | X}$ are close to the counterparts $\E\qth{\theta_1 | X_1, J_1}$ and $\E\qth{\theta_1 | X, J}$, respectively. Specifically, by Tao's inequality (cf. \cite[Corollary 7.11]{polyanskiy2024information}) and $|\theta_i| = O(n\sqrt{\log n})$ almost surely, we have
\begin{align*}
\E\qth{ \pth{\E\qth{\theta_1 | X} - \E\qth{\theta_1 | X, J} }^2 } = O(n^2\log n)\cdot I(\theta_1; J | X) \le O(n^2\log n)\cdot H( J | X).
\end{align*}
To upper bound the conditional entropy $H(J | X)$, we construct an estimator $\widehat{J}$ of $J$ based on $X$, where for each $i\in [n]$, $\widehat{J}_i$ is the index $j\in [m]\cup\{0\}$ such that $X_i\in I_j$. Based on our construction, we have the following implications of events: 
\begin{align*}
    \sth{J_i\neq \widehat{J}_i} \Rightarrow \sth{\theta_i\in I_j} \cap \sth{X_i\notin I_j} \text{ for some }j\in [m]\cup\{0\} \stepa{\Rightarrow} |\theta_i - X_i| \ge 4\sqrt{\log n} - 1. 
\end{align*}
Here (a) recalls that $\theta_i$ must be one of $\theta_1^\star, \dots, \theta_n^\star$. Therefore,
\begin{align*}
\P\pth{J_i\neq \widehat{J}_i} \le \P\pth{|\theta_i - X_i| \ge 4\sqrt{\log n} - 1} = \P\pth{|\calN(0,1)| \ge 4\sqrt{\log n} - 1} = O\pth{ \frac{1}{n^8} }, 
\end{align*}
and by Fano's inequality (cf. \cite[Theorem 3.12]{polyanskiy2024information}) applied to the Markov chain $J_i - X - \widehat{J}_i$,  
\begin{align*}
H(J_i | X) \le \P\pth{J_i\neq \widehat{J}_i} \log m+ h_2\pth{\P\pth{J_i\neq \widehat{J}_i}} = O\pth{\frac{\log n}{n^8}}, 
\end{align*}
where $h_2(p) = p\log \frac{1}{p} + (1-p)\log\frac{1}{1-p}$ is the binary entropy function. Finally, the subadditivity of entropy gives 
\begin{align*}
    H(J|X)\le \sum_{i=1}^n H(J_i|X) = O\pth{\frac{\log n}{n^7}}. 
\end{align*}
In summary, we have shown that
\begin{align}\label{eq:difference}
    \E\qth{ \pth{\E\qth{\theta_1 | X} - \E\qth{\theta_1 | X, J} }^2 } \le \frac{C\log^2 n}{n^5}, 
\end{align}
for some absolute constant $C<\infty$. By a similar argument, the same upper bound also holds for $\E\qth{ \pth{\E\qth{\theta_1 | X_1} - \E\qth{\theta_1 | X_1, J_1} }^2 }$. 

In view of \eqbr{eq:difference}, the inequality $(a+b+c)^2\ge \frac{a^2}{3}-(b^2+c^2)$ gives
\begin{align*}
\E\qth{\|\widehat{\theta}^{\mathrm{S}} - \widehat{\theta}^{\mathrm{PI}}\|_2^2} &= n\E\qth{\pth{ \E\qth{\theta_1|X_1} - \E\qth{\theta_1|X} }^2} \\
&\ge n\pth{ \frac{1}{3}\E\qth{\pth{ \E\qth{\theta_1|X_1,J_1} - \E\qth{\theta_1|X,J} }^2} - \frac{2C\log^2 n}{n^5}}. 
\end{align*}
Next we fix any realization of $J$ with $J_1 = j\neq 0$; without loss of generality we assume that $J_2$ is the other variable taking the same value $j$. In this case, the Bayes rule gives
\begin{align*}
\E\qth{\theta_1|X_1,J_1} &= \frac{(\mu_j + 1)\varphi(X_1 - \mu_j - 1) + (\mu_j-1)\varphi(X_1 - \mu_j + 1)}{\varphi(X_1 - \mu_j - 1) + \varphi(X_1 - \mu_j + 1)} = \mu_j + \tanh(X_1 - \mu_j), \\
\E\qth{\theta_1|X,J} &= \frac{(\mu_j + 1)\varphi(X_1 - \mu_j - 1)\varphi(X_2 - \mu_j+1) + (\mu_j-1)\varphi(X_1 - \mu_j + 1)\varphi(X_2 - \mu_j - 1)}{\varphi(X_1 - \mu_j - 1)\varphi(X_2 - \mu_j+1) + \varphi(X_1 - \mu_j + 1)\varphi(X_2-\mu_j-1)} \\
&= \mu_j + \tanh(X_1 - X_2), 
\end{align*}
where we let $\mu_j := 8(j-1)\sqrt{\log n}$ be the center of $I_j$, and $\varphi(\cdot)$ is the standard normal density. Conditioned on such a realization of $J$, the joint distribution of $(X_1 - \mu_j, X_2 - \mu_j)$ is a balanced two-component Gaussian mixture with means $\sth{(-1,1), (1,-1)}$ and identity covariance, independent of $J$. Denote by $(Z_1, Z_2)$ the random vector with the above law, we have
\begin{align*}
\E\qth{\pth{ \E\qth{\theta_1|X_1,J_1} - \E\qth{\theta_1|X,J} }^2} &\ge \E\qth{\pth{ \E\qth{\theta_1|X_1,J_1} - \E\qth{\theta_1|X,J} }^2\indc{J_1\neq 0}} \\
&=\E\qth{\pth{\tanh(Z_1) - \tanh(Z_1-Z_2)}^2 \indc{J_1\neq 0}} \\
&=\frac{2m}{n} \E\qth{\pth{\tanh(Z_1) - \tanh(Z_1-Z_2)}^2} =: \frac{c_0 m}{n}, 
\end{align*}
for some absolute constant $c_0 > 0$. Therefore, we have established that
\begin{align*}
\E\qth{\|\widehat{\theta}^{\mathrm{S}} - \widehat{\theta}^{\mathrm{PI}}\|_2^2} \ge n\pth{ \frac{c_0m}{3n} - \frac{2C\log^2 n}{n^5}}, 
\end{align*}
and the claimed result follows from the choice of $m = \Omega\pth{\min\sth{n, \mu/\sqrt{\log n}}}$ and the assumption $n\ge n_0$ for a large enough constant $n_0$. 

\subsection{Poisson compound decision problem: Proof of \prettyref{thm:PI-Poisson}}

Recall the setting of this theorem is as follows:
$(p_1^\star, \dots, p_k^\star)$ is the true pmf, $\theta_j = np_{\pi(j)}^\star$ with $\pi\sim \mathrm{Unif}(S_k)$ is the actual Poisson rate, and the observation count vector is $N=(N_1, \dots, N_k)$ with conditionally independent $N_j\sim \Poi(\theta_j)$. The following theorem is an equivalent statement to Theorem \ref{thm:PI-Poisson}\ifthenelse{\boolean{arxiv}}{}{ in the main text} (after applying a scaling factor $n^{-1}$): 
\begin{theorem}\label{thm:PIoracle_Poisson}
For some absolute constant $C<\infty$, it holds that
\begin{align*}
\sum_{i=1}^k\E\qth{ \pth{ \E\qth{\theta_j | N}\log \frac{\E\qth{\theta_j | N}}{\E\qth{\theta_j | N_j}} - \E\qth{\theta_j | N} + \E\qth{\theta_j | N_j} }} \le C\pth{\min\sth{k, n^{1/3}}\log^2 n + \log^3 n}. 
\end{align*}
\end{theorem}
We first note that the analogue of \eqbr{eq:GR-result} (established in \cite{greenshtein2009asymptotic}) in the Poisson model is insufficient to establish \Cref{thm:PIoracle_Poisson}: a central limit theorem $\Poi(\theta) \approx N(\theta,\theta)$ relating the Poisson and Gaussian models will effectively lead to $\mu$ on the order of $\mathrm{poly}(n)$ in \eqbr{eq:GR-result}, leading to an exponentially large bound. In contrast, our improvement via the ``direct sum'' view in the Gaussian setup turns out to be critical in the Poisson setting, where it could be shown that the number of ``clusters'' is precisely $\widetilde{O}(\min\sth{k,n^{1/3}})$ due to Poisson concentration and the normalization  $\sum_{j=1}^k p_j^\star = 1$.

Despite the above connection, the proof of \Cref{thm:PIoracle_Poisson} is considerably more involved and will be split into three parts. In the first part, we state and prove the counterpart of \Cref{thm:PIoracle_Gaussian} in the Poisson model, as our ``main inequality''. In the second part, we control the contributions of large counts $N_j = \Omega(\log n)$ in \eqbr{eq:large_counts} and relate the target quantity to our main inequality via Poisson concentration. In the third part, we handle the small counts $N_j = O(\log n)$ in \eqbr{eq:small_counts} via a relaxation of the permutation-invariant oracle to the natural oracle, a technique tracing back to \cite{orlitsky2015competitive}. Finally, Theorem \ref{thm:PIoracle_Poisson} is a direct consequence of \eqbr{eq:large_counts} and \eqbr{eq:small_counts}. 

\subsubsection{The main inequality}
In this section, we state and prove the following result: 
\begin{lemma}\label{lemma:main_inequality}
Under the setup of \Cref{thm:PIoracle_Poisson}, it holds that
\begin{align*}
\sum_{j=1}^k \E\qth{ \frac{(\E\qth{\theta_j | N} - \E\qth{\theta_j | N_j} )^2}{N_j} \indc{N_j \ge 100\log n} } \le C\min\sth{k,n^{1/3}}\log^2 n, 
\end{align*}
where $C<\infty$ is an absolute constant. 
\end{lemma}

Lemma \ref{lemma:main_inequality} is the counterpart of \Cref{thm:PIoracle_Gaussian} in the Poisson model, which adds the normalizing factor $N_j\approx \theta_j$ on the denominator to account for the dispersion of Poisson distribution, and the indicator $\indc{N_j\ge 100\log n}$ to operate in the subGaussian concentration regime for $N_j\sim \Poi(\theta_j)$. Similar to the proof of \Cref{thm:PIoracle_Gaussian}, the proof of Lemma \ref{lemma:main_inequality} also takes three steps. 

\paragraph{Step I: introducing an auxiliary vector $Z$.} Similar to the Gaussian case, we construct an auxiliary vector $Z=(Z_1,\dots,Z_k)$ to interpolate between $\theta$ and $N$. Our construction is as follows: for each $j\in [k]$, let $\theta_j - Z_j - N_j$ form a Markov chain, with the distributions $Z_j\sim \Poi(2\theta_j)$ given $\theta_j$, and $N_j\sim \mathrm{B}(Z_j, 1/2)$ given $Z_j$ (the binomial distribution with $Z_j$ trials and success probability $1/2$). By the Poison subsampling property, marginally $N_j\sim \Poi(\theta_j)$ follows the target distribution. 

Next we show that
\begin{align}\label{eq:interpolation_poisson}
\E\qth{Z_j | N_j} = N_j + \E\qth{\theta_j | N_j}. 
\end{align}
In fact, by Poisson subsampling property again, $N_j$ and $Z_j - N_j$ are independent $\Poi(\theta_j)$ random variables given $\theta_j$. By the tower property, 
\begin{align*}
\E\qth{Z_j - N_j | N_j} = \E\sth{ \E\qth{Z_j - N_j | N_j, \theta_j} | N_j } = \E\sth{ \E\qth{Z_j - N_j | \theta_j} | N_j } = \E\sth{\theta_j | N_j}, 
\end{align*}
which proves \eqbr{eq:interpolation_poisson}. Similarly we can prove that $\E\qth{Z_j | N} = N_j + \E\qth{\theta_j | N}$, and therefore
\begin{align}\label{eq:poisson_stepI}
\E\qth{\theta_j | N} - \E\qth{\theta_j | N_j} = \E\qth{Z_j | N} - \E\qth{Z_j | N_j}. 
\end{align}

\paragraph{Step II: an information-theoretic upper bound.} Based on \eqbr{eq:poisson_stepI} and the triangle inequality $(a+b+c)^2 \le 3(a^2+b^2+c^2)$, we have
\begin{align*}
&\E\qth{ \frac{(\E\qth{\theta_j | N} - \E\qth{\theta_j | N_j} )^2}{N_j} \indc{N_j \ge 100\log n} } \\
&= \E\qth{ \frac{(\E\qth{Z_j | N} - \E\qth{Z_j | N_j} )^2}{N_j} \indc{N_j \ge 100\log n} } \\
&= \E\qth{ \frac{(\E\qth{Z_j - 2N_j | N} - \E\qth{Z_j - 2N_j | N_j} )^2}{N_j} \indc{N_j \ge 100\log n} } \\
&\le 3\E\qth{ \frac{(\E\qth{(Z_j - 2N_j)\indc{E} | N} - \E\qth{(Z_j - 2N_j)\indc{E} | N_j} )^2}{N_j} \indc{N_j \ge 100\log n} } \\
&\qquad + 3\E\qth{ \frac{(\E\qth{(Z_j - 2N_j)\indc{E^c} | N})^2 + (\E\qth{(Z_j - 2N_j)\indc{E^c} | N_j})^2}{N_j} \indc{N_j \ge 100\log n} } \\
&=: 3(A_1 + A_2),  
\end{align*}
where the event $E$ is defined as $\sth{ |Z_j - 2N_j|\le 100\sqrt{N_j\log n} }$. We start from the main term $A_1$. Writing $X:=(Z_j - 2N_j)\indc{E}$, by the upper bound of $|X|\le 100\sqrt{N_j\log n}$ and Pinsker's inequality, 
\begin{align*}
    A_1 \le \E\qth{ \frac{(\E\qth{X|N} - \E\qth{X|N_j})^2}{N_j} } &\lesssim \log n\cdot \E\qth{\mathrm{TV}(P_{X|N},P_{X|N_j})^2} \\
    &\le \log n \cdot \E\qth{\DKL(P_{X|N} \| P_{X|N_j})} \\
    &= \log n\cdot I(X;N|N_j) \le \log n\cdot I(Z_j; N| N_j). 
\end{align*}
Noticing the Markov structure $\theta_j-Z_j-N_j$, following the same analysis in the Gaussian case gives the upper bound
\begin{align*}
A_1 \lesssim \frac{\log n}{k}\DKL\pth{P_{Z} \bigg\| \prod_{j=1}^k P_{Z_j}}. 
\end{align*}

As for the remainder term $A_2$, by repeated applications of Cauchy--Schwarz, we have
\begin{align*}
    A_2 &\le 2 \E\qth{ \frac{(Z_j-2N_j)^2\indc{E^c}}{N_j}\indc{N_j\ge 100\log n} } \\
    &\le \frac{1}{50\log n} \E\qth{ (Z_j-2N_j)^2 \indc{E^c, N_j\ge 100\log n} } \\
    &= \frac{1}{50k\log n} \sum_{\ell=1}^k \E\qth{ (Z_j-2N_j)^2 \indc{E^c, N_j\ge 100\log n} | \theta_j = \theta_\ell^\star} \\
    &\le \frac{1}{50k\log n} \sum_{\ell=1}^k \E^{1/2}\qth{ (Z_j-2N_j)^4 | \theta_j = \theta_\ell^\star} \P^{1/2}\pth{E^c, N_j\ge 100\log n | \theta_j = \theta_\ell^\star}, 
\end{align*}
where $\theta_\ell^\star := np_\ell^\star$ is deterministic. Since $\theta_\ell^\star \le n$ and $Z_j - 2N_j = (Z_j - N_j) - N_j$ is the difference of two $\Poi(\theta_\ell^\star)$ random variables conditioned on $\theta_j = \theta_\ell^\star$, we have
\begin{align*}
    \E^{1/2}\qth{ (Z_j-2N_j)^4 | \theta_j = \theta_\ell^\star} = O\pth{n^2}. 
\end{align*}
For the second term, we distinguish into two cases. If $\theta_\ell^\star \le 20\log n$, the Chernoff bound (cf. Lemma \ref{lemma:chernoff}) gives
\begin{align*}
    \P\pth{E^c, N_j\ge 100\log n | \theta_j = \theta_\ell^\star} &\le  \P\pth{N_j\ge 100\log n | \theta_j = \theta_\ell^\star} \\
    &\le \pth{\frac{e\theta_{\ell}^\star}{100\log n}}^{100\log n} = O\pth{\frac{(\theta_{\ell}^\star)^2}{n^{10}}}. 
\end{align*}
If $\theta_{\ell}^\star > 20\log n$, note that when $Z_j\le 8N_j$, the event $E^c$ implies that
\begin{align*}
\left|\sqrt{N_j} - \sqrt{\frac{Z_j}{2}} \right| = \frac{|N_j - Z_j/2|}{\sqrt{N_j} + \sqrt{Z_j/2}} \ge \frac{100\sqrt{N_j\log n}}{\sqrt{N_j} + 2\sqrt{N_j}} > 10\sqrt{\log n}. 
\end{align*}
This lower bound also holds trivially when $Z_j > 8N_j$ and $N_j \ge 100\log n$. Therefore, the Chernoff bound in Lemma \ref{lemma:chernoff} applied to $N_j\sim \mathrm{B}(Z_j, \frac{1}{2})$ gives
\begin{align*}
\P\pth{E^c, N_j\ge 100\log n | \theta_j = \theta_\ell^\star} \le \P\pth{ \left|\sqrt{N_j} - \sqrt{\frac{Z_j}{2}} \right| > 10\sqrt{\log n} \bigg| \theta_j = \theta_\ell^\star }\le \frac{2}{n^{10}}. 
\end{align*}
Combining the above displays yields
\begin{align*}
A_2 \lesssim \frac{n^2}{k\log n}\pth{\sum_{\ell: \theta_{\ell}^\star \le 20\log n} \frac{\theta_{\ell}^\star}{n^5} + \sum_{\ell: \theta_{\ell}^\star > 20\log n} \frac{1}{n^{5}}} = O\pth{\frac{1}{n^2k}}, 
\end{align*}
where the last step uses the identity $\sum_{\ell=1}^k \theta_{\ell}^\star = n$. Summarizing the upper bounds of $A_1$ and $A_2$ leads to the final upper bound of this step: 
\begin{align}\label{eq:poisson_stepII}
    \sum_{j=1}^k \E\qth{ \frac{(\E\qth{\theta_j | N} - \E\qth{\theta_j | N_j} )^2}{N_j} \indc{N_j \ge 100\log n} } = O\pth{\DKL\pth{P_{Z} \bigg\| \prod_{j=1}^k P_{Z_j}}\log n + \frac{1}{n^2}}. 
\end{align}

\paragraph{Step III: analysis of Poisson permutation mixture.} This step is again similar to the Gaussian case, with different definitions of the intervals $I_1, \dots, I_m$. As entries of $\theta^\star$ lie in $[0, n]$, we define the following quadratic grid motivated by the Poisson concentration: 
\begin{align*}
I_{\ell} = [(\ell-1)^2, \ell^2), \qquad \ell = 1, \dots, m = \lceil \sqrt{n+1} \rceil. 
\end{align*}
It is clear that $[0,n] \subseteq \cup_{\ell=1}^m I_\ell$. Similar to the proof in the Gaussian case, for $j\in [k]$ let $L_j\in [m]$ be the index where $\theta_j\in I_{L_j}$, and the same analysis in \eqbr{eq:chain_rule} leads to
\begin{align*}
\DKL\pth{ P_Z \bigg\| \prod_{j=1}^k P_{Z_i} } \le \DKL\pth{ P_{L} \bigg\| \prod_{j=1}^k P_{L_i} } + \E_L\qth{ \DKL\pth{ P_{Z|L} \bigg\| \prod_{j=1}^k P_{Z_j|L_j} } } =: B_1 + B_2. 
\end{align*}
We begin with the upper bound of $B_1$. For $\ell\in [m]$, let $k_\ell = \sum_{j=1}^k \indc{L_j = \ell}=\sum_{j=1}^k \mathbf{1}\{\theta_j^\star = \ell\}$ be the number of coordinates of $\theta^\star$ lying in $I_\ell$. Again, $(k_1,\dots,k_m)$ is a deterministic vector, and $P_L$ is the uniform distribution over $[m]^k$ with the given type $(k_1,\dots,k_m)$. Therefore,
\begin{align*}
    B_1 = kH(L_1) - H(L) = \sum_{j=1}^k k_j \log \frac{k}{k_j} - \log \binom{k}{k_1,\dots,k_m}. 
\end{align*}
To proceed, we shall need the following inequality: 
\begin{align}\label{eq:effective_support}
m_0 := \sum_{\ell=1}^m \indc{k_\ell > 0} = O\pth{\min\sth{k, n^{1/3}}}. 
\end{align}
The quantity $m_0$ represents the effective number of ``clusters'' in the Poisson model. To show \eqbr{eq:effective_support}, the first upper bound $m_0\le k$ is a direct consequence of $\sum_{\ell=1}^m k_{\ell} = k$. For the second upper bound $m_0 = O(n^{1/3})$, note that
\begin{align*}
n = \sum_{j=1}^k \theta_j^\star \ge \sum_{\ell=1}^m k_\ell (\ell-1)^2 \ge \sum_{\ell=1}^{m_0} (\ell-1)^2 = \Omega(m_0^3) \Longrightarrow m_0 = O\pth{n^{1/3}}. 
\end{align*}
Based on \eqbr{eq:effective_support}, we use Stirling's approximation to upper bound $B_1$ as
\begin{align*}
B_1 &\le \sum_{\ell=1}^m k_{\ell}\log \frac{k}{k_{\ell}} - k\log\frac{k}{e} - \frac{1}{2}\log(2\pi k) + \sum_{\ell: k_{\ell}>0} \pth{k_{\ell}\log\frac{k_\ell}{e} + \frac{1}{2}\log(8\pi k_\ell) } \\
&=-\frac{1}{2}\log(2\pi k)+ \frac{1}{2}\sum_{\ell: k_{\ell}>0}\log(8\pi k_{\ell}) \\
&\stepa{\le} -\frac{1}{2}\log(2\pi k)+ \frac{m_0\log(8\pi)}{2} + \frac{1}{2}\log(n^{m_0-1}k) \overset{\eqbr{eq:effective_support}}{=} O\pth{\min\sth{k,n^{1/3}}\log n}, 
\end{align*}
where (a) is a crude upper bound $\prod_{\ell: k_{\ell}>0} k_{\ell} \le n^{m_0-1}k$, for there are $m_0$ non-zero terms in $(k_{\ell})_{\ell=1}^m$, with $k_1\le k$ and $k_2,\dots,k_{m}\le n$. 

As for $B_2$, we employ the same proof in the Gaussian case to arrive at
\begin{align*}
\DKL\pth{ P_{Z|L} \bigg\| \prod_{j=1}^k P_{Z_j | L_j} } = \sum_{\ell=1}^m \DKL\pth{ P_{\ell} \| Q_{\ell} } \le \sum_{\ell=1}^m \chi^2\pth{ P_{\ell} \| Q_{\ell} }, 
\end{align*}
where $P_\ell$ and $Q_\ell$ are the permutation mixture and its i.i.d.~approximation for the set of distributions $\calP_{\ell} := \{\Poi(2\theta_j^\star): \theta_j^\star\in I_{\ell}\}$. In the proof of \cite[Corollary 1.3]{han2024approximate}, it was shown that $\mathsf{C}_{\chi^2}(\calP_1)\le 1-c$ for an absolute constant $c>0$, and $\mathsf{D}_{\chi^2}(\calP_\ell) = O(1)$ for every $\ell \ge 2$, where $\mathsf{C}_{\chi^2}, \mathsf{D}_{\chi^2}$ are the $\chi^2$ channel capacity and diameter defined in \cite{han2024approximate}. Therefore, \cite[Theorem 1.2]{han2024approximate} shows that $\chi^2\pth{ P_{\ell} \| Q_{\ell} }  = O(1)$ for all $\ell\in [m]$. In addition, at most $m_0$ of $\calP_1, \dots, \calP_m$ is non-empty, by the definition of $m_0$ in \eqbr{eq:effective_support}. Combining these observations leads to
\begin{align*}
    B_2 = O(m_0) = O\pth{ \min\sth{k, n^{1/3}} }. 
\end{align*}
A combination of the upper bounds of $B_1$ and $B_2$ yields the main result of this step: 
\begin{align*}
\DKL\pth{ P_Z \bigg\| \prod_{j=1}^k P_{Z_j} } = O\pth{\min\sth{k,n^{1/3}}\log n}.  
\end{align*}
Plugging this result into \eqbr{eq:poisson_stepII} completes the proof of Lemma \ref{lemma:main_inequality}. \qed

\subsubsection{Controlling the contributions of large counts}
The target inequality of this section is
\begin{align}\label{eq:large_counts}
\sum_{j=1}^k \E\qth{ T_j\indc{N_j\ge 100\log n} } = O\pth{\min\sth{k,n^{1/3}}\log^2 n}, 
\end{align}
where we use the shorthand
\begin{align*}
    T_j := \E\qth{\theta_j | N}\log \frac{\E\qth{\theta_j | N}}{\E\qth{\theta_j | N_j}} - \E\qth{\theta_j | N} + \E\qth{\theta_j | N_j}. 
\end{align*}
To this end, we define a good event $E_j := \sth{ \theta_j/2 \le N_j \le 2\theta_j }$. Note that by Bayes rule,
\begin{align*}
\E\qth{\theta_j | N_j} = \frac{ \sum_{\ell=1}^k \theta_{\ell}^\star \Poi(N_j; \theta_{\ell}^\star) }{\sum_{\ell=1}^k \Poi(N_j; \theta_{\ell}^\star)} \ge \frac{N_j}{4} \cdot \frac{\sum_{\ell: \theta_\ell^\star \ge N_j/4} \Poi(N_j; \theta_{\ell}^\star)}{\sum_{\ell=1}^k \Poi(N_j; \theta_{\ell}^\star)}. 
\end{align*}
We claim that the event $E_j \cap \sth{N_j\ge 100\log n}$ implies
\begin{align}\label{eq:Ej_implication}
\sum_{\ell: \theta_\ell^\star < N_j/4} \Poi(N_j; \theta_{\ell}^\star) \le \sum_{\ell: N_j/2 \le \theta_\ell^\star \le 2N_j} \Poi(N_j; \theta_{\ell}^\star), 
\end{align}
which, based on the expression of $\E\qth{\theta_j | N_j}$, further implies that
\begin{align*}
\E\qth{\theta_j | N_j} \ge \frac{N_j}{4} \cdot \Big(1 - \frac{\sum_{\ell: \theta_\ell^\star < N_j/4} \Poi(N_j; \theta_{\ell}^\star)}{\sum_{\ell=1}^k \Poi(N_j; \theta_{\ell}^\star)}\Big) \geq \frac{N_j}{4} \cdot \Big(1 - \frac{\sum_{\ell: \theta_\ell^\star < N_j/4} \Poi(N_j; \theta_{\ell}^\star)}{\sum_{\ell:\theta_\ell^\star < N_j/4} + \sum_{\ell:N_j/2 \leq \theta_\ell^\star \leq 2N_j} \Poi(N_j; \theta_{\ell}^\star)}\Big)\geq \frac{N_j}{8}.
\end{align*}
To show \eqbr{eq:Ej_implication}, first we note that the index set $\sth{\ell\in [k]: N_j/2 \le \theta_\ell^\star \le 2N_j}$ is nonempty thanks to $E_j$. Therefore, the right-hand side of \eqbr{eq:Ej_implication} is at least
\begin{align*}
A := \min\sth{ \Poi(N_j; N_j/2), \Poi(N_j; 2N_j) }. 
\end{align*}
For $\theta < N_j/4$, we have
\begin{align*}
\frac{\Poi(N_j; \theta)}{A} &= \max\sth{ e^{N_j/2 - \theta}\pth{\frac{2\theta}{N_j}}^{N_j}, e^{2N_j - \theta}\pth{\frac{\theta}{2N_j}}^{N_j} } \\
&= \exp\pth{ \max\sth{N_j\log\frac{2\theta}{N_j} + \frac{N_j}{2} - \theta, N_j\log\frac{\theta}{2N_j} + 2N_j - \theta} } \\
&\le \pth{\frac{2\theta}{N_j}}^{\frac{N_j}{8}}\exp\pth{ \max\sth{\frac{7N_j}{8}\log\frac{2\theta}{N_j} + \frac{N_j}{2} - \theta, \frac{7N_j}{8}\log\frac{\theta}{2N_j} + 2N_j - \theta} } \\
&\stepa{\le} \pth{\frac{2\theta}{N_j}}^{\frac{N_j}{8}}\exp\pth{ \max\sth{\frac{1}{4}-\frac{7\log 4}{8}, \frac{7}{4}-\frac{7\log 8}{8} }N_j } \stepb{\le} \frac{\theta}{n}, 
\end{align*}
where (a) uses that both quantities are increasing in $\theta\in (0,N_j/4)$ by differentiation, and (b) follows from the assumption $N_j\ge 100\log n$ and simple algebra. Therefore, 
\begin{align*}
\sum_{\ell: \theta_\ell^\star \le N_j/4 }\frac{\Poi(N_j; \theta_\ell^\star)}{A} \le \sum_{\ell: \theta_\ell^\star \le N_j/4 } \frac{\theta_\ell^\star}{n} \le 1, 
\end{align*}
where the last step follows from $\sum_{\ell=1}^k \theta_{\ell}^\star = n$. This proves \eqbr{eq:Ej_implication}. Now using $a\log\frac{a}{b}-a+b\le \frac{(a-b)^2}{b}$ for $a,b>0$, 
\begin{align}\label{eq:large_counts_good}
\sum_{j=1}^k \E\qth{ T_j\indc{E_j, N_j\ge 100\log n} } &\le \sum_{j=1}^k \E\qth{ \frac{(\E\qth{\theta_j | N} - \E\qth{\theta_j | N_j} )^2}{ \E\qth{\theta_j | N_j} } \indc{E_j, N_j \ge 100\log n} } \nonumber \\
&\le \sum_{j=1}^k \E\qth{ \frac{(\E\qth{\theta_j | N} - \E\qth{\theta_j | N_j} )^2}{ N_j/8 } \indc{E_j, N_j \ge 100\log n} } \nonumber \\
&\le 8\sum_{j=1}^k \E\qth{ \frac{(\E\qth{\theta_j | N} - \E\qth{\theta_j | N_j} )^2}{ N_j } \indc{N_j \ge 100\log n} } \nonumber \\
&= O\pth{ \min\sth{k,n^{1/3}}\log^2 n }, 
\end{align}
by our main inequality in Lemma \ref{lemma:main_inequality}. 

To finish the proof of \eqbr{eq:large_counts}, it remains to control the contribution from $E_j^c$. We first establish an upper bound on $\P(E_j^c, N_j\ge 100\log n)$. 
For $\theta_\ell^\star \ge 20\log n$, Chernoff bound (cf. Lemma \ref{lemma:chernoff}) gives 
\begin{align*}
\P(N_j\notin [\theta_\ell^\star/2, 2\theta_\ell^\star]| \theta_j = \theta_\ell^\star) \le 2\exp\pth{-\frac{(\theta_{\ell}^\star)^2}{8}} = O\pth{\frac{1}{n^{10}}}. 
\end{align*}
For $\theta_\ell^\star < 20\log n$, again Chernoff bound (cf. Lemma \ref{lemma:chernoff}) gives 
\begin{align*}
\P(N_j \ge 100\log n| \theta_j = \theta_\ell^\star) \le \pth{\frac{e\theta_{\ell}^\star}{100\log n}}^{100\log n} = O\pth{\frac{(\theta_{\ell}^\star)^2}{n^{10}}}. 
\end{align*}
Therefore,
\begin{align*}
\P(E_j^c, N_j\ge 100\log n) &= \frac{1}{k}\sum_{\ell=1}^k \P(N_j\notin [\theta_\ell^\star/2, 2\theta_\ell^\star], N_j \ge 100\log n| \theta_j = \theta_\ell^\star) \\
&\lesssim \frac{1}{k}\pth{ \sum_{\ell: \theta_\ell^\star \ge 20\log n} \frac{1}{n^{10}} + \sum_{\ell: \theta_\ell^\star < 20\log n} \frac{(\theta_\ell^\star)^2}{n^{10}}} = O\pth{\frac{1}{kn^{10}}\sum_{\ell=1}^k (\theta_\ell^\star)^2 }. 
\end{align*}
Next we prove a deterministic upper bound of $T_j$. Clearly, since $0\le \theta_\ell^\star \le n$ for all $\ell\in [k]$, we have $0\le \E\qth{\theta_j | N_j}, \E\qth{\theta_j | N}\le n$. We also show a lower bound for $\E\qth{\theta_j | N_j}$. Note that
\begin{align*}
\E\qth{\theta_j | N_j = y} = \frac{ \sum_{\ell=1}^k \theta_{\ell}^\star \Poi(y; \theta_{\ell}^\star) }{\sum_{\ell=1}^k \Poi(y; \theta_{\ell}^\star)} =  \frac{ \sum_{\ell=1}^k e^{-\theta_{\ell}^\star} (\theta_{\ell}^\star)^{y+1} }{\sum_{\ell=1}^k e^{-\theta_{\ell}^\star} (\theta_{\ell}^\star)^y}
\end{align*}
is non-decreasing in $y$ by Cauchy--Schwarz, for $y\ge 1$ we have
\begin{align}\label{eq:sep_lower}
\E\qth{\theta_j | N_j = y} \ge \E\qth{\theta_j | N_j = 1} = \frac{ \sum_{\ell=1}^k e^{-\theta_{\ell}^\star} (\theta_{\ell}^\star)^2 }{\sum_{\ell=1}^k e^{-\theta_{\ell}^\star} \theta_{\ell}^\star } \ge \frac{e^{-n} \sum_{\ell=1}^k (\theta_{\ell}^\star)^2 }{\sum_{\ell=1}^k \theta_{\ell}^\star} = \frac{1}{ne^n}\sum_{\ell=1}^k (\theta_\ell^\star)^2. 
\end{align}
Therefore, deterministically we have
\begin{align*}
    T_j \le n\log\frac{n^2 e^n}{\sum_{\ell=1}^k (\theta_\ell^\star)^2} + n = O\pth{n^2 + n\log \frac{n^2}{\sum_{\ell=1}^k (\theta_\ell^\star)^2}}. 
\end{align*}
A combination of the above displays then gives
\begin{align}\label{eq:large_counts_bad}
\sum_{j=1}^k \E\qth{ T_j\indc{E_j^c, N_j\ge 100\log n} } \lesssim \frac{\sum_{\ell=1}^k (\theta_\ell^\star)^2}{n^{10}}\pth{n^2 + n\log \frac{n^2}{\sum_{\ell=1}^k (\theta_\ell^\star)^2}} = O\pth{\frac{1}{n^6}}, 
\end{align}
where the last step follows from simple algebra and $\sum_{\ell=1}^k (\theta_\ell^\star)^2 \le (\sum_{\ell=1}^k \theta_\ell^\star)^2 = n^2$. Now \eqbr{eq:large_counts_good} and \eqbr{eq:large_counts_bad} give the target inequality \eqbr{eq:large_counts}. 

\subsubsection{Controlling the contributions of small counts}
This section aims to show that for all $0\le y\le 100\log n$, it holds that
\begin{align}\label{eq:small_counts}
\sum_{j=1}^k \E\qth{ \pth{\E\qth{\theta_j | N}\log \frac{\E\qth{\theta_j | N}}{\E\qth{\theta_j | N_j}} - \E\qth{\theta_j | N} + \E\qth{\theta_j | N_j}} \indc{N_j = y} } = O\pth{ \log^2 n }. 
\end{align}
To analyze the challenging term $\E\qth{\theta_j | N}$, we make use of the convexity of $x\mapsto x\log x$ to get
\begin{align*}
&\sum_{j=1}^k \E\qth{ \pth{\E\qth{\theta_j | N}\log \frac{\E\qth{\theta_j | N}}{\E\qth{\theta_j | N_j}} - \E\qth{\theta_j | N} + \E\qth{\theta_j | N_j}} \indc{N_j = y} } \\
&\le \sum_{j=1}^k \E\qth{ \pth{\E\qth{\theta_j | N, S_y}\log \frac{\E\qth{\theta_j | N, S_y}}{\E\qth{\theta_j | N_j}} - \E\qth{\theta_j | N, S_y} + \E\qth{\theta_j | N_j}} \indc{N_j = y} }, 
\end{align*}
where 
\begin{align}\label{eq:S_y}
    S_y := \sum_{j=1}^k \theta_j \indc{N_j = y}
\end{align}
is the sum of Poisson rates for symbols exactly appearing $y$ times. The key benefit of introducing $S_y$ is that the quantity $\E\qth{\theta_j | N, S_y}$ now admits a simple expression. Specifically, by symmetry, if $N_j = N_{j'} = y$, we have $\E\qth{\theta_j | N, S_y} = \E\qth{\theta_{j'} | N, S_y}$. Therefore, using the shorthand
\begin{align}\label{eq:Phi_y}
\Phi_y := \sum_{j=1}^k \indc{N_j = y}
\end{align}
to denote the number of symbols appearing exactly $y$ times, for $N_j=y$ we have
\begin{align*}
\E\qth{\theta_j | N, S_y} = \frac{1}{\Phi_y}\sum_{\ell: N_{\ell} = y} \E\qth{\theta_\ell | N, S_y} = \frac{1}{\Phi_y} \E\qth{S_y | N, S_y} = \frac{S_y}{\Phi_y}. 
\end{align*}
This is precisely the natural oracle introduced in \cite{orlitsky2015competitive}, and the quantity $\Phi_y$ is called the \textit{profile} (see also \prettyref{eq:modified-Good--Turing}\ifthenelse{\boolean{arxiv}}{}{ of the main text}). For the symmetric oracle, when $N_j = y$, the quantity 
\begin{align}\label{eq:b_y}
\E\qth{\theta_j | N_j} = \frac{ \sum_{\ell=1}^k \theta_{\ell}^\star \Poi(y; \theta_{\ell}^\star) }{\sum_{\ell=1}^k \Poi(y; \theta_{\ell}^\star)}  =: b_y
\end{align}
is deterministic and denoted by $b_y$. Consequently, 
\begin{align*}
&\sum_{j=1}^k \E\qth{ \pth{\E\qth{\theta_j | N, S_y}\log \frac{\E\qth{\theta_j | N, S_y}}{\E\qth{\theta_j | N_j}} - \E\qth{\theta_j | N, S_y} + \E\qth{\theta_j | N_j}} \indc{N_j = y} } \\
&= \sum_{j=1}^k \E\qth{\pth{ \frac{S_y}{\Phi_y}\log \frac{S_y}{b_y\Phi_y} - \frac{S_y}{\Phi_y} + b_y}\indc{N_j=y}} = \E\qth{ S_y\log\frac{S_y}{b_y\Phi_y} - S_y + b_y\Phi_y }, 
\end{align*}
so \eqbr{eq:small_counts} is established if we show that
\begin{align}\label{eq:small_counts_target}
    (\star) := \E\qth{ S_y\log\frac{S_y}{b_y\Phi_y} - S_y + b_y\Phi_y } = O\pth{\log^2 n}, \qquad 0\le y\le 100\log n. 
\end{align}
Note that the joint distribution of $(S_y, \Phi_y)$ is invariant with the permutation of $\theta$, in the sequel we may work under the scenario where $\theta_j = \theta_j^\star$ and $N_j\sim \Poi(\theta_j^\star)$. We use the shorthand $q_j = \Poi(y; \theta_j^\star)$ and prove \eqbr{eq:small_counts_target} by distinguishing into four different cases. 

\paragraph{Case I: $\sum_{j=1}^k q_j < 100\log n$ and $\sum_{j=1}^k \theta_j^\star q_j < \frac{1}{n^{10}}$.} Since $\E\qth{S_y - b_y\Phi_y}=0$ by \eqbr{eq:S_y}--\eqbr{eq:b_y}, in this case
\begin{align*}
(\star) &= \E\qth{ S_y\log\frac{S_y}{b_y\Phi_y}} = \E\qth{ S_y\log\frac{S_y}{b_y\Phi_y}\indc{\Phi_y>0}} \\
&\le \E\qth{ S_y}\log\frac{n}{b_y} = \pth{\sum_{j=1}^k \theta_j^\star q_j}\log \frac{n\sum_{j=1}^k q_j }{\sum_{j=1}^k \theta_j^\star q_j} = O\pth{\frac{\log n}{n^{10}}}. 
\end{align*}

\paragraph{Case II: $\sum_{j=1}^k q_j < 100\log n$ and $\sum_{j=1}^k \theta_j^\star q_j \ge \frac{1}{n^{10}}$.} Using $a\log\frac{a}{b}-a+b\le \frac{(a-b)^2}{b}$ for $a,b>0$, 
\begin{align*}
(\star) \le \E\qth{ \frac{(S_y - b_y\Phi_y)^2}{b_y\Phi_y} } \le \frac{1}{b_y}\E\qth{ \pth{S_y - b_y\Phi_y}^2 }. 
\end{align*}
Since $\E\qth{S_y - b_y\Phi_y}=0$, we continue the above bound as
\begin{align*}
(\star) &\le \frac{1}{b_y}\var\pth{S_y - b_y\Phi_y} = \frac{1}{b_y} \sum_{j=1}^k \var\pth{ (\theta_j^\star - b_y)\indc{N_j = y} } \stepa{\le} \frac{1}{b_y}\sum_{j=1}^k (\theta_j^\star - b_y)^2 q_j \\
&\stepb{\le} \frac{1}{b_y}\sum_{j=1}^k (\theta_j^\star)^2 q_j = \frac{\sum_{j=1}^k (\theta_j^\star)^2 q_j }{\sum_{j=1}^k \theta_j^\star q_j}\sum_{j=1}^k q_j \le (100\log n)\cdot \frac{\sum_{j=1}^k (\theta_j^\star)^2 q_j }{\sum_{j=1}^k \theta_j^\star q_j}, 
\end{align*}
where (a) uses the independence of $(N_j)_{j=1}^k$, (b) uses the definition \eqbr{eq:b_y} of $b_y$ that it is the weighted average of $\theta_j^\star$ with weights proportional to $q_j$. To upper bound the last ratio, note that since $y\le 100\log n$, the Chernoff bound (cf. Lemma \ref{lemma:chernoff}) gives that $q_j = \Poi(y; \theta_j^\star) \le n^{-20}$ if $\theta_j^\star \ge 200\log n$. Therefore,
\begin{align*}
\frac{\sum_{j=1}^k (\theta_j^\star)^2 q_j }{\sum_{j=1}^k \theta_j^\star q_j}\le \frac{1}{\sum_{j=1}^k \theta_j^\star q_j}\pth{\sum_{j: \theta_j^\star < 200\log n} (200\log n)\theta_j^\star q_j + \sum_{j: \theta_j^\star \ge 200\log n} \frac{n\theta_j^\star}{n^{20}} } = O\pth{\log n}, 
\end{align*}
and consequently $(\star) = O\pth{\log^2 n}$. 

\paragraph{Case III: $\sum_{j=1}^k q_j \ge 100\log n$ and $\sum_{j=1}^k \theta_j^\star q_j < \frac{1}{n^{10}}$.} By the Chernoff bound (cf. Lemma \ref{lemma:chernoff}), 
\begin{align*}
\P\pth{ \Phi_y \ge \frac{\E\qth{\Phi_y}}{2}} \ge 1 - \exp\pth{-\frac{1}{8}\sum_{j=1}^k q_j}, \quad \text{ where } \frac{\E\qth{\Phi_y}}{2} = \frac{1}{2}\sum_{j=1}^k q_j.
\end{align*}
Call the above event $E$, then distinguishing into $E$ and $E^c$ gives
\begin{align*}
(\star) &= \E\qth{ S_y\log \frac{S_y}{b_y\Phi_y}\indc{E} } + \E\qth{ S_y\log \frac{S_y}{b_y\Phi_y}\indc{E^c} } \\
&\le \E\qth{S_y}\log \frac{2n}{\sum_{j=1}^k \theta_j^\star q_j} + \E\qth{S_y\indc{E^c}}\log \frac{n}{b_y} \\
&\le \pth{\sum_{j=1}^k \theta_j^\star q_j}\log \frac{2n}{\sum_{j=1}^k \theta_j^\star q_j} + \E^{1/2}\qth{S_y^2}\P^{1/2}\pth{E^c}\log \frac{n\sum_{j=1}^k q_j}{\sum_{j=1}^k \theta_j^\star q_j} \\
&\le \pth{\sum_{j=1}^k \theta_j^\star q_j}\log \frac{2n}{\sum_{j=1}^k \theta_j^\star q_j} + \pth{n\sum_{j=1}^k \theta_j^\star q_j}^{1/2}\exp\pth{-\frac{1}{16}\sum_{j=1}^k q_j}\log \frac{n\sum_{j=1}^k q_j}{\sum_{j=1}^k \theta_j^\star q_j} \\
&= O\pth{\frac{\log n}{n^{10}}}, 
\end{align*}
where the last step follows from both assumptions and simple algebra. 

\paragraph{Case IV: $\sum_{j=1}^k q_j \ge 100\log n$ and $\sum_{j=1}^k \theta_j^\star q_j \ge \frac{1}{n^{10}}$.} Using $a\log\frac{a}{b}-a+b\le \frac{(a-b)^2}{b}$ for $a,b>0$,
\begin{align*}
(\star) \le \E\qth{ \frac{(S_y - b_y\Phi_y)^2}{b_y\Phi_y} }. 
\end{align*}
Define the same event $E$ in Case III, with $\P(E^c) \le \exp(-\frac{1}{8}\sum_{j=1}^k q_j)$. Under $E$, 
\begin{align*}
\E\qth{ \frac{(S_y - b_y\Phi_y)^2}{b_y\Phi_y} \indc{E}}\le \frac{2}{\sum_{j=1}^k \theta_j^\star q_j}\E\qth{ \pth{S_y - b_y\Phi_y}^2 }.
\end{align*}
Proceeding in the same way as Case II, the above term is at most $O(\log n)$. Under $E^c$, note that
\begin{align*}
\frac{(S_y - b_y\Phi_y)^2}{b_y\Phi_y}\le \frac{\pth{\max\sth{n,  \frac{b_y}{2}\sum_{j=1}^k q_j}}^2}{b_y} = \frac{\pth{\max\sth{n, \frac{1}{2}\sum_{j=1}^k \theta_j^\star q_j}}^2}{\sum_{j=1}^k \theta_j^\star q_j}\sum_{j=1}^k q_j \le n^{12}\sum_{j=1}^k q_j. 
\end{align*}
Therefore,
\begin{align*}
    \E\qth{ \frac{(S_y - b_y\Phi_y)^2}{b_y\Phi_y} \indc{E^c}}\le n^{12}\pth{\sum_{j=1}^k q_j}\P(E^c)\le n^{12}\pth{\sum_{j=1}^k q_j} \exp\pth{-\frac{1}{8}\sum_{j=1}^k q_j} = O(1). 
\end{align*}
Adding up the contributions from $E$ and $E^c$ shows that $(\star) = O(\log n)$. 

%% file: GT_lower_bound.tex
\section{Proof of Good-Turing lower bound (\prettyref{thm:Good--Turing-LB})}
\label{sec:GT-proof}


Recall that the modified Good-Turing estimator is defined to be:
\begin{align}\label{eq:modified-Good-Turing}
\widehat{p}^{\MGT}_j(N;y_0) = \frac{\bar{p}^{\MGT}_j(N;y_0)}{Z^{\MGT}}, \quad j=1,\ldots,k
\end{align}
where, with $\Phi_y = \sum_{j=1}^k \bm{1}\{N_j = y\}$,
\begin{align*}
\bar{p}_{j}^{\mathrm{MGT}}(N; y_0) = \begin{cases}
    \frac{y+1}{n}\cdot \frac{\Phi_{y+1}+1}{\Phi_y} &\text{if } N_j = y \le y_0, \\
    \frac{N_j}{n} & \text{if } N_j = y > y_0,
\end{cases} \quad Z^{\MGT} = \sum_{j=1}^k  \bar{p}^{\MGT}_j(N;y_0).
\end{align*} 



We distinguish two cases based on the value of threshold $y_0$. 
In each case, we construct a hard instance of the true distribution $p^\star$ under which the regret of the modified Good--Turing estimator is provably larger than $\Omega((n\log n)^{-1/2})$. 
Let $y^\star := C_0\sqrt{n\log n}$ for a large enough absolute constant $C_0>0$ to be specified later.

\subsection{Case I: $y_0 \le y^\star$} 
For small $y_0$, consider
\begin{align*}
p^\star = \pth{ \frac{y_0+y^\star}{n}, \dots, \frac{y_0+y^\star}{n}, 0, \dots, 0 }, 
\end{align*}
where without loss of generality we assume that $m:= \frac{n}{y_0+y^\star}$ is an integer, with $m\le \sqrt{n}\le k$. Let $E = \cap_{j=1}^m \sth{N_j> y_0}$ be the ``good event'' that all first $m$ symbols have large counts, the Chernoff bound (cf. Lemma \ref{lemma:chernoff}) gives that
\begin{align*}
\P(E^c) \le m \P\pth{ \Poi\pth{y_0+y^\star} \le y_0 } \le m\P\pth{ \Poi\pth{y_0+y^\star} \le \frac{y_0+y^\star}{2} } \le m\exp\pth{-\frac{\sqrt{n}}{8}}. 
\end{align*}

Next we lower bound the KL risk of the modified Good--Turing estimator \eqbr{eq:modified-Good-Turing} and upper bound the KL risk of the permutation-invariant oracle. For $\widehat{p}^{\mathrm{MGT}}$, note that on the event $E$, we have $\Phi_0 = k-m$ and $\Phi_1 = 0$ which implies $Z^{\MGT} = \sum_{i=1}^m N_i/n + \sum_{i=m+1}^k \frac{1}{n(k-m)} = \frac{1}{n}(\sum_{i=1}^m N_i + 1)$, hence
\begin{align*}
\E\qth{ \DKL(p^\star \| \widehat{p}^{\mathrm{MGT}}) } &\stepa{\ge} 2\pth{\E\qth{ \mathrm{TV} (p^\star,  \widehat{p}^{\mathrm{MGT}})}}^2 \\
&\ge 2\pth{\E\qth{ \mathrm{TV} (p^\star,  \widehat{p}^{\mathrm{MGT}}) \indc{E}} }^2 \\
&\ge \frac{1}{2}\pth{ \sum_{j=1}^m \E\qth{|p_j^\star - \widehat{p}_j^{\mathrm{MGT}}| \indc{E} } }^2 \\
&=\frac{1}{2}\pth{ \sum_{j=1}^m \E\qth{\Big|p_j^\star - \frac{N_j}{\sum_{\ell=1}^m N_\ell + 1}\Big| \indc{E} } }^2\\
&\geq \frac{1}{2}\pth{ \sum_{j=1}^m \E\qth{\Big|p_j^\star - \frac{N_j}{n}\Big| - \frac{N_j}{n}\Big|1 - \frac{n}{\sum_{\ell=1}^m N_\ell + 1} \Big| } - 2\P(E^c)  }_+^2 \\
&\stepb{=} \Omega\pth{\sqrt{\frac{m}{n}} - O(n^{-1/2}) - O\pth{me^{-\sqrt{n}/8}} }^2  =  \Omega\pth{\frac{m}{n}}, 
\end{align*}
where (a) is due to Pinsker's inequality and Cauchy--Schwarz, and (b) uses $\E|\Poi(\theta) - \theta| = \Omega(\sqrt{\theta})$ for $\theta\ge 1$, the upper bound of $\P(E^c)$, and
\begin{align*}
\E\Big[\sum_{j=1}^m \frac{N_j}{n}\Big|1 - \frac{n}{\sum_{\ell=1}^m N_\ell + 1} \Big|\Big] = \E\qth{\frac{N|N+1-n|}{n(N+1)}} \le \frac{\E|N+1-n|}{n} \le \frac{\sqrt{n}+1}{n},
\end{align*}
for $N := \sum_{\ell=1}^m N_\ell \sim \Poi(n)$. As for the permutation-invariant oracle, we define another permutation-invariant decision rule: 
\begin{align*}
\widehat{p}_j = \begin{cases}
    \frac{y_0+y^\star}{n} &\text{if } N_j>0 \text{ and } \sum_{\ell=1}^k \indc{N_\ell>0} = m, \\
    0 &\text{if } N_j=0 \text{ and } \sum_{\ell=1}^k \indc{N_\ell>0} = m, \\
    \frac{1}{k} &\text{otherwise}. 
\end{cases}
\end{align*}
It is clear that $\widehat{p} = p^\star$ when $E$ holds, therefore 
\begin{align*}
\E\qth{ \DKL(p^\star \| \widehat{p}^{\mathrm{PI}}) } \le \E\qth{ \DKL(p^\star \| \widehat{p}) } = \E\qth{ \DKL(p^\star \| \widehat{p}) \indc{E^c}} \le \P(E^c)\log k \le m\exp\pth{-\frac{\sqrt{n}}{8}}\log k. 
\end{align*}
Consequently, since $m = \Omega(\sqrt{n/\log n})$ and $k \leq n^C$, we have
\begin{align*}
\reg(\widehat{p}^{\mathrm{MGT}}) &= \E\qth{ \DKL(p^\star \| \widehat{p}^{\mathrm{MGT}}) } - \E\qth{ \DKL(p^\star \| \widehat{p}^{\mathrm{PI}}) } \\
&\ge \Omega\pth{\frac{m}{n}} -  m\exp\pth{-\frac{\sqrt{n}}{8}}\log k = \Omega\pth{\frac{1}{\sqrt{n\log n}}}, 
\end{align*}
as claimed. 

\subsection{Case II: $y_0 > y^\star$}
For large $y_0$, the hard instance of $p^\star$ is constructed in a different manner. In the sequel $C_1, C_2, \dots$ denote different absolute constants.

Let $x_\ell := (C_1\ell^2\log n)/{n}$, where $C_1$ is a large enough constant to be specified later. Then for $\ell_0 := \lceil (n/(C_2\log n))^{1/4}\rceil$ with a large enough constant $C_2$ (depending on $C_1$), we define $\ell_1\in \naturals$ to be the solution to
$\sum_{\ell=\ell_0}^{\ell_1} \ell x_\ell = 1$. (Here for simplicity we assume $\ell_1$ to be an integer; in general, one may proceed similarly by letting $\ell_1$ be the largest integer such that $\sum_{\ell=\ell_0}^{\ell_1} \ell x_\ell \leq 1$ and then renormalizing $\{x_\ell\}_{\ell=\ell_0}^{\ell_1}$ such that $p^\star$ defined below is a probability vector).
In addition, by choosing $C_1$ and $C_2$ properly, such a solution satisfies
$\ell_1 = O\pth{(n/\log n)^{1/4}}$ and $\ell_1\ge 2\ell_0$; in particular, $n\cdot x_{\ell_1}\le y^\star = C_0\sqrt{n\log n}$ for a large enough constant $C_0$ depending only on $(C_1, C_2)$. 

Consider
\begin{align*}
p^\star = \pth{\underbrace{x_{\ell_0}, \dots, x_{\ell_0}}_{\ell_0-\text{times}}, \underbrace{x_{\ell_0+1}, \dots, x_{\ell_0+1}}_{(\ell_0+1)-\text{times}}, \dots, \underbrace{x_{\ell_1},\ldots, x_{\ell_1}}_{\ell_1-\text{times}}, 0, \dots, 0}, 
\end{align*}
where for $\ell_0\le \ell\le \ell_1$, the term $x_\ell$ appears exactly $\ell$ times. Since $\sum_{\ell=\ell_0}^{\ell_1} \ell x_\ell = 1$, the vector $p^\star$ is a valid probability vector. In addition, the support size of $p^\star$ is $\sum_{\ell=\ell_0}^{\ell_1} \ell \le \ell_1^2 = O(\sqrt{n/\log n})$, which is smaller than $k\ge \sqrt{n}$ for $n$ large enough. 

We first upper bound the KL risk for the permutation-invariant oracle. To this end, we define another permutation-invariant decision rule $\widehat{p}$ as follows: for $\ell_0\le \ell\le \ell_1$, let 
\begin{align*}
I_{\ell} := \left[ \frac{n(x_{\ell-1}+x_\ell)}{2}, \frac{n(x_\ell+x_{\ell+1})}{2} \right),  
\end{align*}
and $E_1$ be the event that there are precisely $\ell$ symbols with counts belonging to $I_\ell$ for all $\ell_0\le \ell\le \ell_1$, and rest of the symbols all being unseen. The decision rule $\widehat{p}$ is then defined as
\begin{align*}
\widehat{p}_j = \begin{cases}
    x_{\ell} & \text{if } E_1\text{ holds and }N_j\in I_{\ell}, \\
    0 & \text{if }E_1\text{ holds and }N_j=0,\\
    \frac{1}{k} & \text{if }E_1\text{ does not hold}, 
\end{cases}
\end{align*}
which is obviously permutation invariant. In particular, $\widehat{p} = p^\star$ if the following event $E$ holds: for all $j\in [k]$ and $\ell\in [\ell_0, \ell_1]$, if $p_j^\star = x_\ell$, then $N_j\in I_\ell$. Using the union bound and Chernoff bound (cf. Lemma \ref{lemma:chernoff}), simple algebra gives that $\P(E^c) = O\pth{n^{-10}}$ for a large enough $C_1>0$. Therefore,
\begin{align}\label{eq:PI_risk}
\E\qth{ \DKL(p^\star \| \widehat{p}^{\mathrm{PI}}) } \le \E\qth{ \DKL(p^\star \| \widehat{p}) } = \E\qth{ \DKL(p^\star \| \widehat{p}) \indc{E^c}} \le \P(E^c)\log k = O\pth{\frac{\log k}{n^{10}}}. 
\end{align}

Next we lower bound the KL risk of the modified Good--Turing estimator $\widehat{p}^{\mathrm{MGT}}$. We first lower bound the KL risk by the sum of contributions from symbols with different probabilities: 
\begin{align*}
\E\qth{ \DKL(p^\star \| \widehat{p}^{\mathrm{MGT}}) } &= \sum_{j=1}^k \E\qth{ p_j^\star \log \frac{p_j^\star}{\widehat{p}_j^{\mathrm{MGT}}} - p_j^\star + \widehat{p}_j^{\mathrm{MGT}}  } \\
&\ge \sum_{\ell=\ell_0}^{\ell_1} \sum_{j: p_j^\star = x_\ell} \E\qth{ p_j^\star \log \frac{p_j^\star}{\widehat{p}_j^{\mathrm{MGT}}} - p_j^\star + \widehat{p}_j^{\mathrm{MGT}}  } =: \sum_{\ell=\ell_0}^{\ell_1} S_{\ell}. 
\end{align*}
To lower bound each term $S_\ell$, let
\begin{align*}
\underline{I}_\ell = [nx_\ell - C_3\sqrt{nx_\ell}, nx_\ell - C_3^{-1}\sqrt{nx_\ell} - 1], \quad \overline{I}_\ell= [nx_\ell + C_3^{-1}\sqrt{nx_\ell} + 1, nx_\ell + C_3\sqrt{nx_\ell}],
\end{align*}
where $C_3>1$ is a large absolute constant. For each $j\in [k]$ with $p_j^\star = x_\ell$, we define two events
\begin{align*}
    E_j &:= \sth{ N_j \in \underline{I}_\ell, \Phi_{N_j+1} = 0, \Phi_{N_j} = 1}, \\
    F_j &:=  \sth{ N_j \in \overline{I}_\ell, \Phi_{N_j+1} = 0, \Phi_{N_j} = 1},
\end{align*}
where we recall that $\Phi_y = \sum_{j=1}^k \indc{N_j = y}$ is the $y$th profile. Using $Z$ as a shorthand for $Z^{\MGT}$ in the sequel, we will show that 
\begin{align}\label{eq:claim_joint_prob}
\Prob(\{Z \geq 1\} \cap E_j) \vee \Prob(\{Z \leq 1\} \cap F_j) =  \Omega(1).
\end{align}
We establish \eqbr{eq:claim_joint_prob} by showing that $\Prob(Z \geq 1) \geq 1/2$ would imply $\Prob(\{Z \geq 1\} \cap E_j) = \Omega(1)$; the other implication from $\Prob(Z \leq 1) \geq 1/2$ to $\Prob(\{Z \leq 1\} \cap F_j) = \Omega(1)$ is entirely symmetric. Since trivially $\P(Z\ge 1) \vee \P(Z\le 1) \ge 1/2$, the above implications would imply \eqbr{eq:claim_joint_prob}.

Now we assume that $\Prob(Z \geq 1) \geq 1/2$. Note that
\begin{align*}
\Prob(\{Z \geq 1\} \cap E_j) &= \Prob\pth{ \sth{Z\ge 1} \cap \sth{N_j \in \underline{I}_\ell} \cap \sth{\Phi_{N_j+1} = 0, \Phi_{N_j} = 1}}\\
&\geq \Prob\pth{\{Z\geq 1\} \cap \sth{N_j \in \underline{I}_\ell}} - \Prob\pth{\sth{N_j \in \underline{I}_\ell}
\cap \sth{\Phi_{N_j+1} = 0, \Phi_{N_j} = 1}^c}.
\end{align*}
We will lower bound the first probability and upper bound the second. For the second probability, we have
\begin{align*}
\Prob\pth{\sth{N_j \in \underline{I}_\ell}
\cap \sth{\Phi_{N_j+1} = 0, \Phi_{N_j} = 1}^c} &= \sum_{y\in \underline{I}_\ell} \Prob(\sth{\Phi_{N_j+1} \neq 0} \cup \sth{\Phi_{N_j} \neq 1} | N_j = y)\Prob(N_j = y)\\
&\leq \sum_{y\in \underline{I}_\ell} \big(\Prob(\Phi_{y+1} \neq 0|N_j = y) + \Prob(\Phi_{y} \neq 1 | N_j = y)\big) \Prob(N_j = y).
\end{align*}
Thanks to the mutual independence of $N = (N_1, \dots, N_k)$, by the union bound we have
\begin{align*}
\P(\Phi_{y} \neq 1 | N_j = y) = \P(N_{j'} = y \text{ for some }j' \neq j) &\leq \sum_{j'\neq j} \P(N_{j'} = y). 
\end{align*}
To upper bound the probability $\P(N_{j'} = y)$, we distinguish into three cases: 
\begin{enumerate}
    \item If $p_{j'}^\star = 0$, then clearly $\P(N_{j'} = y) = \P(\Poi(0)=y) = 0$; 
    \item If $p_{j'}^\star \neq x_\ell$, then $|np_{j'}^\star -y|\ge \min_{\ell'\neq \ell} |nx_{\ell'}-nx_\ell|-C_3\sqrt{nx_\ell}\ge \sqrt{C_1nx_\ell/2}$ if $C_1>0$ large enough (depending on $C_3$). Consequently, the Chernoff bound (cf. Lemma \ref{lemma:chernoff}) gives $\P(N_{j'}= y) = O(n^{-10})$ again for $C_1>0$ large enough; 
    \item If $p_{j'}^\star = x_\ell$, we simply use 
    \begin{align*}
    \max_{x\in \naturals} \P(\Poi(\theta) = x) =\P(\Poi(\theta) = \lfloor \theta \rfloor) \le \frac{e^{-\theta} \theta^{\lfloor \theta \rfloor} }{ \sqrt{2\pi \lfloor \theta \rfloor}(\lfloor \theta \rfloor/e)^{\lfloor \theta \rfloor} } \le \frac{2}{\sqrt{\theta}}
    \end{align*}
    for all $\theta\ge 1$ to conclude that $\P(N_{j'}= y) \le \frac{2}{\sqrt{nx_\ell}} \le \frac{2}{\sqrt{C_1}\ell}$. 
\end{enumerate}
Combining the above cases and noticing that $x_\ell$ appears $\ell$ times in $p^\star$, we get
\begin{align*}
\P(\Phi_{y} \neq 1 | N_j = y)\leq O(n^{-9}) + \frac{2}{\sqrt{C_1}} \leq 0.005
\end{align*}
as long as both $n$ and $C_1$ are large enough. A similar argument yields that $\Prob(\Phi_{y+1} \neq 0 | N_j = y) \leq 0.005$ as well, so 
\begin{align*}
\Prob\pth{\sth{N_j \in \underline{I}_\ell}
\cap \sth{\Phi_{N_j+1} = 0, \Phi_{N_j} = 1}^c} \le 0.01\sum_{y\in \underline{I}_\ell} \Prob(N_j = y) \leq 0.01. 
\end{align*}

Next we lower bound $\Prob(\{Z\geq 1\} \cap \sth{N_j \in \underline{I}_\ell})$. Let
\begin{align*}
A_\ell := \Big\{\sum_{j:p_j^\star \in I_\ell}\indc{N_j \in \underline{I}_\ell} \geq \frac{\ell}{5}\Big\}. 
\end{align*}
Since $\Prob(\Poi(\theta) \in [\theta - C_3\sqrt{\theta}, \theta - C_3^{-1}\sqrt{\theta} - 1]) \geq 0.3$ for both $(\theta,C_3)$ large enough, the random sum in the event $A_\ell$ follows a binomial distribution with $\ell$ independent trials and success probability at least $0.3$. By Chernoff inequality, we have $\P(A_\ell) \geq 1 - O(n^{-10})$ for $n$ large enough, which further implies $\Prob(A_\ell | Z \geq 1) \geq \Prob(A_\ell \cap \{Z\geq 1\})\geq \Prob(Z\geq 1) - \Prob(A_\ell^c) \geq 1/4$, again for $n$ large enough. By Markov's inequality, we have
\begin{align*}
\frac{1}{4} &\leq \Prob(A_\ell | Z \geq 1) = \Prob\Big(\sum_{j:p_j^\star \in I_\ell}\indc{N_j \in \underline{I}_\ell} \ge \frac{\ell}{5} \Big| Z \geq 1\Big) \\
&\leq \frac{5}{\ell} \cdot \E\Big[\sum_{j:p_j^\star \in I_\ell}\indc{N_j \in \underline{I}_\ell}\Big|Z\geq 1\Big] = 5 \cdot \P(N_j \in \underline{I}_\ell | Z\geq 1),
\end{align*}
implying that $\P(N_j \in \underline{I}_\ell | Z\geq 1) \geq 1/20$. Since it is assumed that $\P(Z\ge 1)\ge 1/2$, we conclude that $\Prob(\{Z\geq 1\} \cap \sth{N_j \in \underline{I}_\ell}) = \Prob(N_j \in \underline{I}_\ell |Z\geq 1) \Prob(Z\geq 1) \geq 1/40$. Therefore, 
\begin{align*}
\Prob(\{Z \geq 1\} \cap E_j) \ge \frac{1}{40} - 0.01 > 0.01 = \Omega(1), 
\end{align*}
establishing the claim \eqbr{eq:claim_joint_prob}. 

Finally, if $\Prob(\{Z \geq 1\} \cap E_j) = \Omega(1)$, using $a\log(a/b) - a + b \geq 0$ for $a,b>0$, we have 
\begin{align*}
S_\ell &= \sum_{j: p_j^\star = x_\ell} \E\qth{ p_j^\star \log \frac{p_j^\star}{\widehat{p}_j^{\mathrm{MGT}}} - p_j^\star + \widehat{p}_j^{\mathrm{MGT}}  } \\
&\ge \sum_{j: p_j^\star = x_\ell} \E\qth{ \pth{p_j^\star \log \frac{p_j^\star}{\widehat{p}_j^{\mathrm{MGT}}} - p_j^\star + \widehat{p}_j^{\mathrm{MGT}} }\indc{E_j \cap \{Z \geq 1\}} } \\
&\stepa{=} \sum_{j: p_j^\star = x_\ell} \E\qth{ \pth{p_j^\star \log \frac{p_j^\star}{\bar{p}_j^{\mathrm{MGT}}} - p_j^\star + \bar{p}_j^{\mathrm{MGT}} }\indc{E_j \cap \{Z \geq 1\}} }\\
&\stepb{\ge} \sum_{j: p_j^\star = x_\ell} \pth{ x_\ell \log \frac{x_\ell}{x_{\ell} - C_3^{-1}\sqrt{x_\ell/n} } - x_\ell + x_{\ell} - \frac{1}{C_3}\sqrt{\frac{x_\ell}{n}}  }\P(E_j \cap \{Z\geq 1\}) \\
&\stepc{\ge} \sum_{j: p_j^\star = x_\ell} \frac{\Omega(1)}{2n} \stepd{=} \Omega\pth{\frac{\ell}{n}}, 
\end{align*}
where (a) and (b) note that on the event $E_j \cap \{Z\geq 1\}$, we have $$\widehat{p}^{\MGT}_j \le \bar{p}_j^{\MGT} = (N_j+1)/n\le x_\ell - C_3^{-1}\sqrt{\frac{x_\ell}{n}},$$ as well as the decreasing property of $b\mapsto a\log(a/b) - a + b$ on $0\le b\le a$, (c) uses the inequality $\log\frac{1}{1-x}\ge x+\frac{x^2}{2}$ for $x\ge 0$, and (d) follows from $|\{j: p_j^\star = x_\ell\}| = \ell$. The case of $\Prob(\{Z \le 1\} \cap F_j) = \Omega(1)$ is handled using an analogous manner. Consequently, \eqbr{eq:claim_joint_prob} implies that
\begin{align}\label{eq:MGT-risk}
    \E\qth{ \DKL(p^\star \| \widehat{p}^{\mathrm{MGT}}) } = \sum_{\ell=\ell_0}^{\ell_1} S_{\ell} = \Omega\pth{\frac{1}{n}}\sum_{\ell = \ell_0}^{\ell_1} \ell = \Omega\pth{\frac{\ell_0^2}{n}} = \Omega\pth{\frac{1}{\sqrt{n\log n}}}. 
\end{align}
Now the desired claim of the theorem follows from \eqbr{eq:PI_risk} and \eqbr{eq:MGT-risk}.

%% file: preliminary.tex
\section{Equivalent definitions of regret and PI oracle}\label{sec:PI_equivalence}
In the compound decision literature, 
the regret can be defined by competing against either the oracle that knows the true parameter but is restricted to be permutation-invariant (PI), or 
the oracle that knows the true parameter up to a permutation. For completeness, in this section, we verify the equivalence of these definitions in the context of distribution estimation 
and give an expression for the oracle.

Recall from \eqbr{def:intro_regret} \ifthenelse{\boolean{arxiv}}{}{in the main tex} that the regret of a distribution estimator $\widehat p$ is defined as the worst-case excess risk over the PI oracle that knows the true distribution $p^\star$, that is,
\begin{align}\label{eq:reg_original}
\notag\reg(\widehat{p}) &:= \sup_{p^\star \in \Delta_k} [r_n(p^\star, \widehat p) - \min_{\widetilde{p} \text{: PI}} r_n(p^\star, \widetilde{p})]\\ 
&= \sup_{p^\star \in \Delta_k} [r_n(p^\star, \widehat p) - r_n(p^\star, \widehat{p}^{\PI})],
\end{align}
where we recall the following definitions:
\begin{itemize}
    \item For any estimator $\widetilde{p} = \widetilde{p}(N_1,\ldots,N_k) 
\in \Delta_k$, $r_n(p^\star,\widetilde{p}) = \E_{p^\star}[\DKL(p^\star||\widetilde{p})]$ denotes its average risk measured by the Kullback-Leibler (KL) divergence, where
$\E_{p^\star}$ is over the empirical counts 
$N=(N_1,\ldots,N_k)$ which are independently distributed as $\Poi(n p^\star_i)$ under the Poisson sampling model. 
\item
The PI oracle $\widehat{p}^{\PI}$ is defined by 
\begin{align}\label{eq:PI_oracle_SI_def1}
\widehat p^{\PI} = \argmin  \E[\DKL(p^\star \| \widetilde{p})],
\end{align}
where the minimum is taken over all PI estimators $\widetilde p: \mathbb{N}^k  \rightarrow \Delta_k$ satisfying  $\pi(\widetilde{p}(N)) = \widetilde{p}(\pi(N))$ for 
for any permutation $\pi$ of $[k]$ and 
any $N = (N_1,\ldots,N_k)$. 
We emphasize that $\widehat{p}^{\PI}$ is called an oracle since it has knowledge to the true distribution $p^\star$ but must satisfy the symmetry constraint of permutation invariance.
\end{itemize}

Next, we show that $\reg(\widehat{p})$ admits the following equivalent formulation:
\begin{align}\label{eq:reg_SI_equivalent}
\reg(\widehat{p}) = \sup_{p^\star\in \Delta_k}[r_n(p^\star,\widehat{p}) - \min_{\widetilde{p}} \max_{\pi\in S_k} r_n(\pi(p^\star), \widetilde{p})].
\end{align}
Here the inner minimization is taken over all (oracle) estimators $\widetilde p$ which may depend on the ground truth  $p^\star$ up to a permutation (or equivalently, the empirical distribution thereof.)
Furthermore, both minima (oracle risks) in \prettyref{eq:PI_oracle_SI_def1}
and \prettyref{eq:reg_SI_equivalent} coincide, attained by  $\widehat{p}^{\PI}$ given by the explicit formula \eqbr{eq:PI_oracle_intro} and \eqbr{eq:PI-oracle-full} \ifthenelse{\boolean{arxiv}}{}{in the main text}, namely:
\begin{equation}\label{eq:PI_oracle_SI}
\widehat p^{\PI}_i = \E[p_i|N] = 
\frac{
\sum_{\pi \in S_k} p^\star_{\pi(i)} \prod_{j=1}^k 
(p^{\star}_{\pi(j)})^{N_j}
}{\sum_{\pi \in S_k}  \prod_{j=1}^k 
(p^{\star}_{\pi(j)})^{N_j} }, \quad i = 1,\ldots, k.
\end{equation}
Here the expectation is taken over the following data generating process in which the true distribution is randomly relabeled: for fixed $p^\star \in \Delta_k$, first draw a uniform random permutation $\pi$ on $[k]$ and define the permuted version  
$p := \pi(p^\star)$, then generate the counts $N_1,\ldots,N_k$
according to $p$, which are independent $\Poi(np_i) $  conditioned  on $p$. In other words, under this Bayesian setting, the PI oracle is given by the posterior mean $\widehat p^{\PI} = \E[p|N]$.
Moreover, this interpretation makes it clear that \prettyref{eq:PI_oracle_SI} is permutation invariant.

As a historical remark, we note that \prettyref{eq:reg_SI_equivalent} was first used as the definition of compound regret \cite{robbins1951asymptotically, zhang2003compound}; see also \cite{orlitsky2015competitive} for  the specific problem of distribution estimation under 
Poisson sampling. 
In the Gaussian sequence model, \eqbr{eq:PI_oracle_SI_def1} and its further Bayesian interpretation \eqbr{eq:PI_oracle_SI} is introduced in \cite{greenshtein2009asymptotic} (see Proposition 1.1 therein). 

For completeness, we next show the equivalence between 
\eqbr{eq:reg_original} and \eqbr{eq:reg_SI_equivalent}. To see the minimum in \eqbr{eq:PI_oracle_SI_def1} is achieved by \eqbr{eq:PI_oracle_SI}, we note that for any permutation invariant estimator $\widetilde{p}:\mathbb{N}^k\rightarrow\Delta_k$, 
\begin{align*}
\E_{p^\star}[\DKL(p^\star\| \widetilde{p})] - \E_{p^\star}[\DKL(p^\star\| \widehat{p}^{\PI})] &\stackrel{(i)}{=} \E[\DKL(p\| \widetilde{p})] - \E[\DKL(p\| \widehat{p}^{\PI})] \\
&= \E\sum_{i=1}^k p_i\log\frac{\widehat p^{\PI}_i}{\widetilde p_i} \stackrel{(ii)}{=} \E\sum_{i=1}^k \widehat{p}^{\PI}_i\log\frac{\widehat p^{\PI}_i}{\widetilde p_i} = \E[\DKL(\widehat p^{\PI}\| \widetilde{p})] \geq 0,
\end{align*}
where in $(i)$ we apply the permutation invariance of $\widetilde{p}$ and $\widehat{p}^{\PI}$, and in $(ii)$ we take first the conditional expectation over $N$ and apply $\E[p_i|N] = \widehat{p}^{\PI}_i$. To see the minimum in \eqbr{eq:reg_SI_equivalent} is also achieved by \eqbr{eq:PI_oracle_SI}, we have for any (not necessarily PI)
estimator $\widetilde{p}: \mathbb{N}^k \rightarrow \Delta_k$, 
\ifthenelse{\boolean{arxiv}}{\begin{align*}
&\max_{\pi\in S_k}\E_{\pi(p^\star)}\qth{\DKL(\pi(p^\star) \| \widetilde{p})} - \max_{\pi\in S_k}\E_{\pi(p^\star)}\qth{\DKL(\pi(p^\star) \| \widehat{p}^{\PI})} \\
&\stackrel{(i)}{=} \max_{\pi\in S_k}\Big\{\E_{\pi(p^\star)}\qth{\DKL(\pi(p^\star) \| \widetilde{p})} - \E_{\pi(p^\star)}\qth{\DKL(\pi(p^\star) \| \widehat{p}^{\PI})}\Big\}\\
&\stackrel{(ii)}{\geq} \E\Big[\DKL(p \| \widetilde{p}) - \DKL(p \| \widehat{p}^{\PI})\Big] = \E\sum_{i=1}^k p_i\log\frac{\widehat{p}^{\PI}_i}{\widetilde{p}_i} = \E\sum_{i=1}^k \widehat{p}^{\PI}_i\log\frac{\widehat p^{\PI}_i}{\widetilde p_i} = \E[\DKL(\widehat p^{\PI}\| \widetilde{p})] \geq 0, 
\end{align*}}{
\begin{align*}
&\max_{\pi\in S_k}\E_{\pi(p^\star)}\qth{\DKL(\pi(p^\star) \| \widetilde{p})} - \max_{\pi\in S_k}\E_{\pi(p^\star)}\qth{\DKL(\pi(p^\star) \| \widehat{p}^{\PI})} \stackrel{(i)}{=} \max_{\pi\in S_k}\Big\{\E_{\pi(p^\star)}\qth{\DKL(\pi(p^\star) \| \widetilde{p})} - \E_{\pi(p^\star)}\qth{\DKL(\pi(p^\star) \| \widehat{p}^{\PI})}\Big\}\\
&\stackrel{(ii)}{\geq} \E\Big[\DKL(p \| \widetilde{p}) - \DKL(p \| \widehat{p}^{\PI})\Big] = \E\sum_{i=1}^k p_i\log\frac{\widehat{p}^{\PI}_i}{\widetilde{p}_i} = \E\sum_{i=1}^k \widehat{p}^{\PI}_i\log\frac{\widehat p^{\PI}_i}{\widetilde p_i} = \E[\DKL(\widehat p^{\PI}\| \widetilde{p})] \geq 0, 
\end{align*}}
where $(i)$ applies the fact that $\E_{\pi(p^\star)}\qth{\DKL(\pi(p^\star) \| \widehat{p}^{\PI})}$ is independent of $\pi$, and $(ii)$ lower bounds the maximal risk (in $\pi$) by the average risk with a uniform $\pi$.

\section{Preliminaries on NPMLE}
\label{sec:NPMLE-prelim}
In this section we summarize some well-known properties of the nonparametric maximum likelihood estimator (NPMLE) in the one-dimensional Poisson model \cite{simar1976maximum,lindsay1983geometry1}. 
We emphasize these results are deterministic and hold for any configuration of the data.
For results in general one-dimensional exponential families such as Gaussian model, we refer the reader to the monograph \cite{lindsay1995mixture}.

\subsection{Basic properties}
\label{sec:NPMLE-basics}
Recall that 
$\Poi(\theta)$ denotes the Poisson distribution with mean $\theta\geq0$, whose probability mass function (pmf) is given by
$f_\theta(y) := \Poi(\theta;y)
:= \frac{e^{-\theta} \theta^y}{y!}, y\in
\integers_+:=\{0,1,\ldots\}$.
For any probability distribution (prior) $G$ on $\reals_+$, 
\begin{equation}    
f_G(y) := \Expect_{\theta \sim G}[f_\theta(y)] =
\frac{1}{y!} \int G(\d\theta) e^{-\theta} \theta^y, \quad y\geq 0
\label{eq:poisson-mixture-fG}
\end{equation}
denotes the pmf of the Poisson mixture with mixing distribution $G$. In other words, if $\theta \sim G$ and $Y \sim \Poi(\theta)$ conditioned on $\theta$, then the marginal distribution of $Y$ is $f_G$.

Given a sample $(N_1,\ldots,N_k) \in \integers_+^k =\{0,1,\ldots\}^k$, 
we seek to fit a   Poisson mixture that best explains this data by maximizing the total likelihood:
\begin{equation}
\max_{G\in \calP(\reals_+)} \underbrace{\frac{1}{k} \sum_{i=1}^k \log f_G(N_i)}_{:=\ell(G)}
    \label{eq:NPMLE-opt}
\end{equation}
where $\calP(\reals_+)$ denotes the collection of all probability distributions on $\reals_+$. It is known that \prettyref{eq:NPMLE-opt}
has a unique optimizer, denoted by $\widehat G$, henceforth referred to as the NPMLE. Furthermore, 
\begin{itemize}
	\item $\widehat G$ is discrete, whose support is contained in the interval $[N_{\min},N_{\max}]$ with cardinality at most the number of distinct values in the sample $N_1,\ldots,N_k$. Therefore, deterministically, the support size of $\widehat G$ cannot exceed the domain size $k$. In practice, $\widehat G$ typically has much fewer atoms than this worst-case bound.   
    For example, in the Poisson sampling model where the counts $N_i$'s are drawn independently from $\Poi(n p^\star_i)$, we have the following dimension-free bound:\footnote{To see this, it suffices to show that the number of distinct values in $N_1,\ldots,N_k$ is at most $O\Big(\sqrt{N_{\rm tot}}\Big)$, where 
    $N_{\rm tot} = \sum_{i=1}^k N_i \sim \Poi(n)$, which is at most $O(n)$ with probability tending to one.     
    Indeed, the profile $\Phi_y = \sum_{i=1}^k \Indc\{N_i=y\}$ denotes the number of symbols occurring exactly $y$ times (the $y$th profile).
    Then $q=\sum_{y\geq 0}^\infty \Indc\{\Phi_y\geq 1\}$ is the number of distinct values, and
$N_{\rm tot}=
\sum_{y\geq 0}^\infty y \Phi_y 
\geq \sum_{y\geq 0} y \Indc\{\Phi_y\geq 1\}
\geq \sum_{y\geq 0}^{q-1} y = q(q-1)/2$.    
    } with high probability, $|\supp(\widehat G)| = O(\sqrt{n})$, which does not depend on the underlying domain size $k$. This type of structural property of the NPMLE is known as self-regularization that is well-understood in the Gaussian model \cite{PW20-npmle}.

	\item $\widehat G$ is the optimal solution if and only if the the \textit{first-order optimality condition} is satisfied: 
	\begin{equation}
	D_{\widehat G}(\theta) \triangleq \frac{1}{k} \sum_{i=1}^k \frac{f_\theta(N_i)}{f_{\widehat G}(N_i)} \leq 1, \quad \forall \theta\geq 0.
	\label{eq:kkt}
	\end{equation}
    This general fact is a simple consequence of the convexity of the optimization problem \prettyref{eq:NPMLE-opt}, (i.e., $G\mapsto\ell(G)$ is concave.)
    Here $D_{\widehat G}(\theta) -1  = \lim_{\varepsilon\to0} \frac{\ell(\widehat G)-\ell((1-\varepsilon)\widehat G+\varepsilon \delta_\theta)}{\varepsilon}$ is the ``directional gradient'' of the log-likelihood objective. 	
	As a corollary, by averaging both sides over $\theta\sim\widehat G$, we have 
	$D_{\widehat G}(\theta) = 1$ for any $\theta \in \supp(\widehat G)$.
    Note that in the special case of Poisson model, for any $G$,  $D_G(\theta)$ is always a  polynomial in $\theta$ with degree $N_{\max}$.
\end{itemize}
Note that $\widehat G$ is determined by the empirical distribution of the data $\pi_{\sf emp} \triangleq \frac{1}{k} \sum_{i=1}^k \delta_{N_i}$,  or equivalently, by the profiles.

\subsection{Computation of the NPMLE}
\label{sec:NPMLE-computation}
The NPMLE can be computed using the Frank--Wolfe (F--W) algorithm \cite{frank1956algorithm,jaggi2013revisiting}, also known as the vertex direction method (VDM) algorithm in the literature of mixture models \cite{simar1976maximum,lindsay1995mixture,jana2022poisson} and experimental design 
\cite{fedorov1972theory,wynn1970sequential}.

Specifically, we apply the 
so-called \textit{fully-corrective} version of the F--W algorithm (see \cite[Algorithm 4]{jaggi2013revisiting}) that iterates the following steps: Given the current solution $G_t$, we find a global maximizer of the gradient:
\begin{equation}
\theta_{t+1} \in \argmax_{\theta}D_{G_t}(\theta).
    \label{eq:fw1}
\end{equation}
Then we add $\theta_{t+1}$ to the support of $G_t$ and re-fit the weights, namely, by solving the convex optimization problem
\begin{equation}
G_{t+1} = \argmax_{G\in \calP(\supp(G_t) \cup \{\theta_{t+1}\})} \ell(G).
    \label{eq:fw2}
\end{equation}
It is known that starting from any initialization $G_0$, the sequence $G_t$ converges to the NPMLE $\widehat G$ at a rate of $O(1/t)$ \cite{fedorov1972theory,lindsay1983geometry1}.
Our implementation is similar to that used in \cite{jana2022poisson}. One notable difference is the following: The gradient function in \prettyref{eq:fw1} is a degree-$N_{\max}$ polynomial in $\theta$  so that its maxima can be found by computing the roots of its derivative. This, however, turns out to be numerically unstable due to the high degree of the polynomial when the sample size is large. As such, we opt for a grid search by discretizing the support interval $[N_{\min},N_{\max}]$ into an equal grid of size $10 N_{\max}$.

\subsection{Two examples}
	\label{sec:NPMLE-example}

In general, there is no closed-form expression for the NPMLE.
Here, we give two examples in which the counts take two distinct values and the NPMLE solution can be computed analytically.
Later in \prettyref{sec:notau} we will use the first example to investigate the role of the regularization parameter $\tau$ (see \prettyref{eq:NPMLE-unnormalized}\ifthenelse{\boolean{arxiv}}{}{ in the main text}) in the NPMLE distribution estimator.
\begin{example}
    \label{ex:NPMLE1}
Suppose that the data has a special configuration, in which a single symbol, say, $k$, appeared $m\geq 2$ times and the other symbols are unseen. In other words, the counts are $(N_1,\ldots,N_k)=(0,\ldots,0,m)$, whose empirical distribution is $\pi_{\sf emp} = (1-\frac{1}{k})\delta_0 + \frac{1}{k}\delta_m$.

To find the NPMLE, thanks to its uniqueness, we simply need to ``guess'' the form of the solution and verify the first-order optimality condition \prettyref{eq:kkt}. Let
\begin{equation}
\widehat G = (1-\varepsilon) \delta_0 + \varepsilon \delta_b,
\label{eq:NPMLE-ex1}
\end{equation}
where $\varepsilon = \frac{1}{k(1-e^{-b})} \in (0,1)$, and $b\in (0,m)$ is the unique solution of the following equation:
\begin{equation}
b = (m-b)(e^b-1).
\label{eq:NPMLE-ex1b}
\end{equation}
We claim that \prettyref{eq:NPMLE-ex1} satisfies the optimality condition \prettyref{eq:kkt} and is thus the unique NPMLE solution.
To this end, we verify:
\begin{itemize}
	\item For $m\geq 2$, the equation $b = (m-b)(e^b-1)$ has a unique solution in $(0,m)$. (Note that this fails for $m=1$.) Indeed, the function $g(b)\equiv(m-b)(e^b-1)-b$ is increasing on $[0,m-1]$ and decreasing on $[m-1,\infty)$, satisfies $g(0)=0, g(m-1)>0$, and $g(m)<0$. 
	\item 
	For \prettyref{eq:NPMLE-ex1}, we have
	\[
	D_{\widehat G}(\theta) 
    = \int \pi_{\mathsf{emp}} (\d y)\frac{f_\theta(y)}{f_{\widehat G}(y)}
    = \frac{k-1}{k} \frac{e^{-\theta}}{1-\varepsilon(1-e^{-b})} + \frac{1}{k \varepsilon} \frac{\theta^m e^{-\theta}}{b^m e^{-b}}
	= e^{-\theta} + (1-e^{-b}) \frac{\theta^m}{b^m}e^{-\theta+b},
	\]
	the last identity thanks to the choice of $\varepsilon$. Then we have, simultaneously, $D_{\widehat G}(0) = D_{\widehat G}(b) = 1$.
	
	\item $D_{\widehat G}(\theta)$ has a global maximum at $\theta=b$ in $(0,m)$; see \prettyref{fig:NPMLE-ex1}. In particular, $D_{\widehat G}'(b) = 0$ and 
	$D_{\widehat G}''(b) < 0$. 
\end{itemize}

\begin{figure}[ht]%
\centering
\includegraphics[width=0.5\columnwidth]{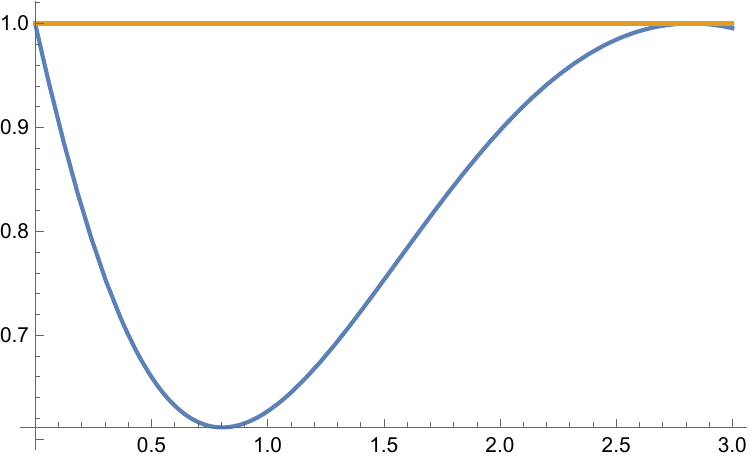}
\caption{The gradient function $\theta\mapsto D_{\widehat{G}}(\theta)$ (blue curve)
in \prettyref{ex:NPMLE1} for $m=3$.}%
\label{fig:NPMLE-ex1}%
\end{figure}

\end{example}

\begin{example}
    \label{ex:NPMLE2}
Suppose the data is such that each symbol either appears 0 times or once, so $N_i=0$ or $1$, but not all $0$ or all $1$.
From the basic properties summarized in \prettyref{sec:NPMLE-basics}, we know that NPMLE has at most two atoms contained in $[0,1]$.
In fact, the NPMLE is a simply a point mass at the sample mean:
\[
\widehat G = \delta_{q}, \quad q = \frac{1}{k}\sum_{i=1}^k N_i.
\]
Indeed, in this case one can verify that the function
\[
D_{\widehat G}(\theta) = (1-q+\theta) e^{-\theta+q}.
\]
satisfies $\max_{0\leq\theta\leq1}D_{\widehat G}(\theta)=1$, uniquely attained at $q$ in the interior; see \prettyref{fig:NPMLE-ex2}.
\begin{figure}[ht]%
\centering
\includegraphics[width=0.5\columnwidth]{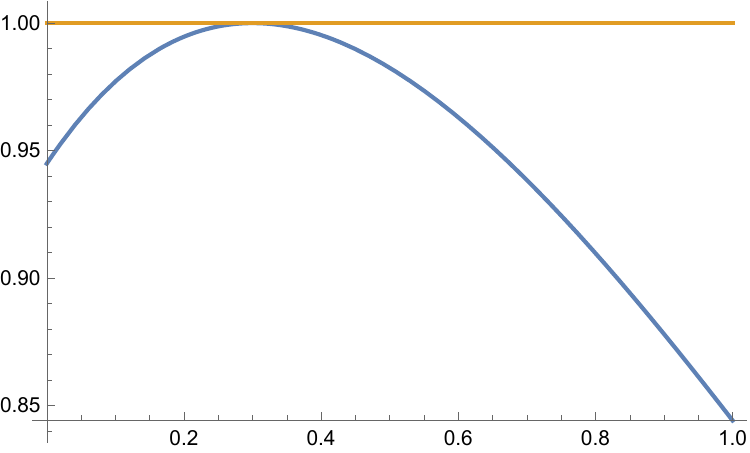}%
\caption{The gradient function $\theta\mapsto D_{\widehat{G}}(\theta)$ in \prettyref{ex:NPMLE2}
with $q=0.3$.}%
\label{fig:NPMLE-ex2}%
\end{figure}
\end{example}

\begin{remark}
As mentioned in the Introduction\ifthenelse{\boolean{arxiv}}{}{ of the main text}, it is well-known that the original Good--Turing estimator
\begin{equation}
\widehat{p}_{i}^{\mathrm{GT}} \propto
(N_i+1) \frac{\Phi_{N_i+1}}{\Phi_{N_i}},
    \label{eq:GT-original-SI}
\end{equation}
may result in unreasonable estimates so that additional modification must be applied in practice. 
In fact, it may not even be well-defined.
The preceding two examples are useful to illustrate its pathological behavior. 
For simplicity, assume that the alphabet is $\{\texttt{a,b,c,d}\}$.
\begin{itemize}
    \item As in \prettyref{ex:NPMLE1}, assume that we sample 4 times and only observe symbol $\texttt{d}$. So $N_{\texttt{a}}=N_{\texttt{b}}=N_{\texttt{c}}=0$ and $N_{\texttt{d}}=4$ and the only non-zero profiles are  
    $\Phi_0=3, \Phi_4=1$.
    In this case, the right-hand side of \prettyref{eq:GT-original-SI} is always 0 and the Good-Turing estimator is undefined.

    \item As in \prettyref{ex:NPMLE2}, assume that we sample twice and $N_{\texttt{a}}=N_{\texttt{b}}=1$ and $N_{\texttt{c}}=N_{\texttt{d}}=0$, so that the only non-zero profiles are 
    $\Phi_0=2, \Phi_1=2$.
    The Good--Turing estimator assigns zero probabilities $\widehat{p}_{\texttt{a}}^{\mathrm{GT}}=\widehat{p}_{\texttt{b}}^{\mathrm{GT}}=0$ to the two singletons 
    and non-zero probabilities $\widehat{p}_{\texttt{c}}^{\mathrm{GT}}=\widehat{p}_{\texttt{d}}^{\mathrm{GT}}=\frac{1}{2}$ to the two unseen symbols.
    This is a more general phenomenon: By \prettyref{eq:GT-original-SI}, the Good--Turing estimator always assigns zero probability to the most frequently observed symbols.
    In practice, this is remedied by reverting to the empirical frequency for symbols that occur sufficiently often, as in the modified Good--Turing estimators used by \cite{orlitsky2015competitive}
    (see 
    \prettyref{eq:modified-Good--Turing} and \prettyref{eq:MGT-exp}\ifthenelse{\boolean{arxiv}}{}{ 
    of the main text}.)
\end{itemize}

\end{remark}

%% file: additional_results.tex
\section{Regret sub-optimality of unregularized NPMLE}
\label{sec:notau}

In this section we show that without appropriate regularization for unseen symbols, the NPMLE estimator cannot compete with the PI oracle.
Consider the NPMLE-based estimator 
in \eqbr{def:npmle_estimator}, but without regularization for unseen symbols (i.e., setting $\tau = 0$ in \eqbr{def:npmle_estimator}):
\begin{align}
\widehat p_i^{\UNPML} := \frac{\bar p_i^{\UNPML}}{Z^{\UNPML}}, \quad \text{ where } \bar{p}^{\UNPML}_i := \frac{\theta_{\widehat G^{\NPMLE}}(N_i)}{n}, \quad Z^{\UNPML} := \sum_{i=1}^k \bar p_i^{\UNPML}.
\end{align}
The following result shows that without this regularization for unseen symbols, the regret NPML-based estimator does not vanish in the worst case.
\begin{proposition}
For large enough $n$ and $2\leq k\leq n^C$ for some universal $C > 0$, it holds that $\reg(\widehat{p}^{\UNPML}) \geq c$ for some universal $c > 0$.
\end{proposition}
\begin{proof}
We shorthand $\widehat{p}^{\UNPML}$ as $\widehat{p}$. Take the construction $p^\star$ in Lemma \ref{lem:NPMLE_bad} such that $\E[\DKL(p^\star\|\widehat{p})] \geq c$ for some universal $c > 0$. It suffices to show that $\E[\DKL(p^\star\|\widehat{p}^{\PI})] \leq c/2$. To this end, consider the following permutation-invariant estimator 
\begin{align*}
\widetilde{p}_j = \widetilde{p}_j(N) = 
\begin{cases}
\frac{100}{nk} & \text{if } N_j \leq n/4 \text{ and }\sum_{\ell=1}^k \bm{1}\{N_\ell > n/4\} = 1,\\
1 - \frac{100(k-1)}{nk} & \text{if } N_j > n/4 \text{ and }  \sum_{\ell=1}^k \bm{1}\{N_\ell > n/4\} = 1,\\
1/k & \text{otherwise},
\end{cases}
\end{align*}
which satisfies $\sum_{j=1}^k \widetilde{p}_j = 1$ with probability one. Define the event $E = \{\sum_{\ell=1}^k \bm{1}\{N_\ell > n/4\} = 1\}$, which satisfies $\Prob(E^c) = O(n^{-10})$ by the Chernoff bound. Hence $\E[\DKL(p^\star\|\widetilde{p})\bm{1}\{E^c\}] \leq \frac{\log k}{n^{10}} \leq \frac{c}{4}$ for large enough $n$, and again by the Chernoff bound,
\begin{align*}
\E[\DKL(p^\star\|\widetilde{p})\bm{1}\{E\}] = \E\Big[\sum_{\ell=1}^{k-1} \Big(p_{\ell}^\star\log\frac{p_{\ell}^\star}{p_k^\star} + p_k^\star\log\frac{p_k^\star}{p_{\ell}^\star}\Big)\bm{1}\{N_\ell > n/4, N_k \leq n/4\}\Big] = O(k\exp(-\Omega(n))) \leq \frac{c}{4}
\end{align*}
for large enough $n$. We therefore conclude that $\E[\DKL(p^\star\|\widehat{p}^{\PI})] \leq \E[\DKL(p^\star\|\widetilde{p})] \leq c/2$, as desired.
\end{proof}

\begin{lemma}\label{lem:NPMLE_bad}
For any $k\geq 2$ and large enough $n$, there exists some $p^\star \in \Delta_k$ such that 
\[
\E [\DKL(p^\star\|\widehat{p}^{\UNPML})] \geq c
\]
for some universal $c > 0$. 
\end{lemma}
\begin{proof}
We shorthand $\widehat{p}^{\UNPML}$ as $\widehat{p}$. Let $p^\star = (p_1^\star, \ldots, p_1^\star, p_k^\star)$ where $p_1^\star = 100/(nk)$ and $p_k^\star = 1 - 100(k-1)/(nk)$, leading to $G_k = \frac{1}{k}\sum_{i=1}^k \delta_{\theta^\star_i} = (1-\frac{1}{k})\delta_{100/k} + \frac{1}{k}
\delta_{a}$ with $a = n - 100(k-1)/k \in [0.99n,n]$ for large enough $n$. Define the event
\begin{align*}
\cE := \{N_1 = \ldots = N_{k-1} = 0, 0.75 n \leq N_{k} \leq 2n\},
\end{align*}
which, for large enough $n$, holds with probability $\Prob(\cE) = \Prob(N_1=0)^{k-1}\Prob(0.75n \leq N_k \leq 2n) \geq e^{-\frac{100(k-1)}{k}}(1-\exp(-cn)) \geq c'$ for some universal $c' > 0$. Then on $\cE$ we have
\begin{align*}
\bar{p}_1 = \ldots = \bar{p}_{k-1} = \frac{\theta_{\widehat{G}}(0)}{n}, \quad \bar{p}_k = \frac{\theta_{\widehat{G}}(N_k)}{n}, 
\end{align*}
and $\widehat{p}_j = \bar{p}_j/Z$ with $Z = \sum_{j=1}^k \bar{p}_j$, which further implies that
\ifthenelse{\boolean{arxiv}}{\begin{align*}
\E[\DKL(p^\star\|\widehat{p})] &\geq \E [\DKL(p^\star\|\widehat{p})\indc{\cE}] \\
&= \E \Big[\Big((k-1)p_1^\star \log\frac{p^\star_1}{\theta_{\widehat{G}}(0)/(n/2)} + p_k^\star\log\frac{p^\star_k}{\theta_{\widehat{G}}(N_k)/(n/2)}\Big)\indc{\cE}\Big] + \E[\log (2Z)\indc{\cE}].
\end{align*}}{
\begin{align*}
\E[\DKL(p^\star\|\widehat{p})] \geq \E [\DKL(p^\star\|\widehat{p})\indc{\cE}] = \E \Big[\Big((k-1)p_1^\star \log\frac{p^\star_1}{\theta_{\widehat{G}}(0)/(n/2)} + p_k^\star\log\frac{p^\star_k}{\theta_{\widehat{G}}(N_k)/(n/2)}\Big)\indc{\cE}\Big] + \E[\log (2Z)\indc{\cE}].
\end{align*}}
On the event $\cE$, we apply
\prettyref{ex:NPMLE1} to conclude that the NPMLE is given by $\widehat{G} = (1-\eps)\delta_0 + \eps\delta_b$ where $b \in (0,N_k)$ is the unique solution to $N_k = (N_k-b)e^b$ and $\eps = \frac{1}{k(1-e^{-b})}$; in particular, $b \geq N_k - 2N_ke^{-N_k}$.
For such a $\widehat G$ with a point mass at 0, we have 
$\theta_{\widehat{G}}(y) = b$  for all $y\geq 1$.
As a result,
on the event $\cE$ for large enough $n$, we have
\begin{align*}
Z = \frac{(k-1)\theta_{\widehat{G}}(0) + \theta_{\widehat{G}}(N_k)}{n} \geq \frac{ \theta_{\widehat{G}}(N_k)}{n} = \frac{b}{n} 
\geq \frac{N_k - 2N_ke^{-N_k}}{n} \geq \frac{1}{2}.
\end{align*}
Therefore $\E[\log (2Z)\indc{\cE}] \geq 0$. Next we have $\theta_{\widehat{G}}(0) = f_{\widehat{G}}(1)/f_{\widehat{G}}(0)$, with $f_{\widehat{G}}(0) = (1-\eps) + \eps e^{-b} = 1-\frac1k$ and $f_{\widehat{G}}(1) = \eps be^{-b} = \frac{1}{k(e^b-1)}$. Using $b \in [\frac{N_k}{2}, N_k]\subseteq [\frac{n}{4},2n]$ on the event $\cE$,
\begin{align*}
(k-1)p_1^\star \log\frac{p^\star_1}{\theta_{\widehat{G}}(0)/(n/2)} = \frac{100(k-1)}{nk}\log\frac{f_{\widehat{G}}(0)}{f_{\widehat{G}}(1)\cdot 50k} = \frac{100(k-1)}{nk}\log\frac{(k-1)(e^b-1)}{50k}\geq 10,
\end{align*}
for $k\ge 2$ and large enough $n$. On the other hand, using $\theta_{\widehat{G}}(N_k) = b \le 2n$ on $\cE$ and $p_k^\star \in [0.9, 1]$, 
\begin{align*}
    p_k^\star\log \frac{p_k^\star}{\theta_{\widehat{G}}(N_k)/(n/2)} \geq \log\frac{0.9}{4} > -2
\end{align*} on $\cE$, so for large enough $n$ we conclude that $\E[\DKL(p^\star\|\widehat{p})] \geq 8\Prob(\cE) \geq c$ for some universal $c > 0$, as desired. 
\end{proof}

\section{Connections to Profile Maximum Likelihood (PML)}
\label{sec:PML}

We discuss in this section another likelihood-based estimator for distribution estimation, known as the \emph{profile maximum likelihood} (PML) \cite{PML,acharya2017unified,hao2019broad,han2021competitive}.
For distribution estimation, there are  three distinct maximum likelihood methods : the vanilla MLE (empirical frequency), the PML, and the NPMLE. 
The PML aims to maximize the total likelihood modulo a relabeling, which involves a non-convex polynomial optimization that is difficult to solve.
In \prettyref{sec:PML-relax}, we show that NPMLE turns out to be a convex relaxation of PML in an appropriate sense. Next, in \prettyref{sec:PML-regret}, we show that like NPMLE, the (more computationally expensive) PML also attains the near optimal regret. This result is a by-product of our regret analysis of NPMLE and the competitive analysis of PML in \cite{han2021competitive}.

As noted before (see \ifthenelse{\boolean{arxiv}}{\prettyref{eq:modified-Good--Turing} and 
\prettyref{eq:Phi_y}}{\prettyref{eq:modified-Good--Turing} in the main text and 
\prettyref{eq:Phi_y} in this supplement}), the profile is a further summary of the histogram counts (without label information), which records how many symbols occur exactly a given number of times. Specifically, let $N=(N_1,\ldots,N_k)$, where 
$N_i$ denotes the number of times the $i$th symbol appears, for $i=1,\ldots,k$. 
The profile of $N$ is 
$\Phi\equiv \Phi(N) = (\Phi_0, \Phi_1, \dots)$, where $\Phi_j = \sum_{i=1}^k \indc{N_i=j}$.  In other words, profiles are the histogram of histograms. For example, if the observed symbols are $(a,b,a,a,c)$ in a universe of symbols $\{a,b,c,d,e\}$, the histogram counts are $(3,1,1,0,0)$, and the profile is $\Phi = (2,2,0,1,\ldots)$, with the rest being $0$. 


The PML aims to maximize the probability of the observed profile $\Phi$, namely
\[
\widehat{p}^{\PML} = \argmax_{p\in S_k} F_{\PML}(p)
\]
where, with $S_k$ denoting the set of permutations on $[k]$,
\begin{equation}
F_{\PML}(p)
:= \frac{1}{k!} \sum_{\pi \in S_k} \prod_{i=1}^k p_{\pi(i)}^{N_i} \propto \sum_{N'\in \naturals^k: \Phi(N')=\Phi} \prod_{i=1}^k p_i^{N_i'}. 
    \label{eq:FPML}
\end{equation}
In other words, the objective $F_{\PML}(p)$ is proportional to the likelihood of observing a given profile $\Phi$, or equivalently the average likelihood over a uniform permutation of $p$, in both multinomial and Poisson sampling models. Indeed, the average likelihood in the multinomial model $(N_1,\dots,N_k)\sim \mathrm{Multi}(n; (p_1,\dots,p_k))$ is
\begin{align*}
\frac{1}{k!} \sum_{\pi \in S_k} \binom{n}{N_1, \dots, N_k} \prod_{i=1}^k p_{\pi(i)}^{N_i} \propto F_{\PML}(p), 
\end{align*}
and similarly the average likelihood in the Poisson model $(N_1,\dots,N_k)\sim \prod_{i=1}^k \Poi(np_i)$ is
\begin{align*}
\frac{1}{k!} \sum_{\pi \in S_k} \prod_{i=1}^k e^{-np_{\pi(i)}}\frac{(np_{\pi(i)})^{N_i}}{N_i!} = \frac{n^{\sum_{i=1}^k N_i}e^{-n}}{N_1!\cdots N_k!} \cdot \frac{1}{k!}\sum_{\pi\in S_k} \prod_{i=1}^k p_{\pi(i)}^{N_i} \propto F_{\PML}(p). 
\end{align*}
Since it is clear that $F_{\PML}(p)$ is oblivious to the relabeling of $p$, the PML distribution $\widehat{p}^{\PML
}$ estimates $p^\star$ up to a permutation, or equivalently, the sorted version of $p^\star$. We refer to \cite{acharya2017unified,han2021competitive} for the state-of-the-art statistical analysis of the PML.

\subsection{NPMLE as a convex relaxation of PML}
\label{sec:PML-relax}
The PML maximization is computationally difficult as the objective $F_{\PML}(p)$ in \eqbr{eq:FPML} is the permanent of the matrix $(p_j^{N_i})_{i,j=1}^k$ \cite{valiant1979complexity}.
In addition, 
$F_{\PML}(p)$ is a symmetric function (polynomial), so that it only depends on the empirical distribution of $p$. This motivates the following relaxation: 
Using the fact $\sum_{j=1}^k p_j=1$, we have
\ifthenelse{\boolean{arxiv}}{\begin{align}
 F_{\PML}(p)\propto
\frac{1}{k!} \sum_{\text{distinct } j_1,\ldots,j_k \in [k]} \prod_{i=1}^k e^{-np_{j_i}}
p_{j_i}^{N_i}
&\leq \frac{1}{k!} \sum_{j_1,\ldots,j_k \in [k]} \prod_{i=1}^k  e^{-np_{j_i}} p_{j_i}^{N_i}
\nonumber\\
&= \frac{1}{k!} 
\prod_{i=1}^k
\sum_{j \in [k]} 
 e^{-np_j} p_j^{N_i}
\propto \prod_{i=1}^k
f_{G_p}(N_i), 
\label{eq:PML-relax}
\end{align}}{
\begin{align}
 F_{\PML}(p)\propto
\frac{1}{k!} \sum_{\text{distinct } j_1,\ldots,j_k \in [k]} \prod_{i=1}^k e^{-np_{j_i}}
p_{j_i}^{N_i}
\leq \frac{1}{k!} \sum_{j_1,\ldots,j_k \in [k]} \prod_{i=1}^k  e^{-np_{j_i}} p_{j_i}^{N_i}
= \frac{1}{k!} 
\prod_{i=1}^k
\sum_{j \in [k]} 
 e^{-np_j} p_j^{N_i}
\propto \prod_{i=1}^k
f_{G_p}(N_i), 
\label{eq:PML-relax}
\end{align}}
where, per the notation \prettyref{eq:poisson-mixture-fG}, $f_{G_p}$ is the marginal pmf of the Poisson mixture with mixing distribution $G_p = \frac{1}{k}\sum_{j=1}^k \delta_{np_j}$:  
\begin{align*}
f_{G_p}(y) 
= \int \Poi(y; \theta) \mathrm{d}G_p(\theta) = \frac{1}{k}\sum_{j=1}^k \Poi(y; np_j) 
= \frac{1}{k}\sum_{j=1}^k 
\frac{e^{-np_j} (np_j)^y}{y!}.
\end{align*}
Finally, optimizing the rightmost upper bound 
in \prettyref{eq:PML-relax} and 
relaxing the decision variable $G_p$ from being a $k$-point empirical distribution to an arbitrary distribution $G$, we arrive at 
\[
\max_{G \in \calP(\reals_+)} \prod_{i=1}^k
f_G(N_i),
\]
which is precisely the NPMLE optimization 
\prettyref{eq:NPMLE-opt}.

\subsection{Regret analysis of PML}
\label{sec:PML-regret}
The PML solution $\widehat{p}^\PML = (\widehat{p}_1^\PML, \dots, \widehat{p}_k^\PML)$ naturally gives an estimator 
\begin{align}\label{def:PML_prior}
\widehat{G}^\PML := \frac{1}{k}\sum_{j=1}^k \delta_{n\widehat{p}_j^\PML}
\end{align}
for the true empirical distribution $G_k = \frac{1}{k}\sum_{j=1}^k \delta_{np_j^\star}$. Consequently, we can use the estimator $\widehat{G}_\PML$ in place of the NPMLE $\widehat{G}$ and construct an estimator $\widehat{p}$ for $p$ as
\begin{align}\label{eq:PML_final_est}
\widehat{p}_j = \frac{\theta_{\widehat{G}^\PML}(N_j; \rho)}{Z^{\PML}}, \quad \text{ with } Z^{\PML} = \sum_{j=1}^k \theta_{\widehat{G}^\PML}(N_j; \rho),
\end{align}
where $\theta_{G}(y; \rho)$ is the regularized Bayes estimator defined in \eqbr{eq:regularized_bayes}.

The following result shows that PML attains the optimal regret $n^{-2/3}$ with a subpolynomial factor $n^{o(1)}$.
\begin{theorem}\label{thm:PML}
Let $\varepsilon\in (0,1/4)$ be any fixed constant, $\rho = (nk)^{-10}$, and assume that $k\le n^{A_0}$ for some constant $A_0>0$. Then under the Poissonized model, the following regret bound holds for the estimator $\widehat{p}$ in \eqbr{eq:PML_final_est}:   
\begin{align*}
\reg\pth{ \widehat{p} } \le C\pth{ \frac{k}{n} \wedge n^{-2/3 + \varepsilon} }\log^5 n, 
\end{align*}
where $C > 0$ depends only on $(A_0, \varepsilon)$. 
\end{theorem}

\begin{proof}
By \Cref{thm:upper_bound_general}, it suffices to show that for a constant $c>0$, 
\begin{align}\label{eq:PML_hellinger}
\P\pth{ H^2\pth{f_{G_k}, f_{\widehat{G}^\PML}} \ge \frac{cn^{1/3+\varepsilon}}{k}  } = \exp\pth{-\Omega\pth{ n^{1/3+\varepsilon}} }
\end{align}
for the PML estimate $\widehat{G}^\PML$ in \eqbr{def:PML_prior}. The proof of \eqbr{eq:PML_hellinger} relies on the following lemma, which is a slight variant of the competitive analysis of the PML in \cite[Theorem 1]{han2021competitive}: 
\begin{lemma}\label{lemma:PML_competitive}
Suppose there exists another estimator $\widehat{G}$ depending only on the profile of $N$ such that
\begin{align}\label{eq:PML_assumption}
\sup_{p^\star} \P\pth{ H^2\pth{f_{G_k}, f_{\widehat{G}}} \ge \frac{n^{1/3+\varepsilon}}{k}  } \le \delta,
\end{align}
then 
\begin{align}\label{eq:PML_result}
\sup_{p^\star} \P\pth{ H^2\pth{f_{G_k}, f_{\widehat{G}^\PML}} \ge \frac{4n^{1/3+\varepsilon}}{k} + r_{n/2}(A)  } \le \delta^{1-\kappa}\exp\pth{C(A, \kappa) n^{1/3+\kappa}} + 2\exp\pth{-\frac{n}{8}},
\end{align}
where $A, \kappa>0$ are any fixed constants, the constant $C(A, \kappa) > 0$ depends only on $A$ and $\kappa$, and the remainder term $r_m(A)$ is defined as
\begin{align*}
r_m(A) := \sup\sth{ H^2\pth{ f_{G_1}, f_{G_2} }: G_1 = \frac{1}{k}\sum_{j=1}^k \delta_{np_j}, G_2 = \frac{1}{k}\sum_{j=1}^k \delta_{nq_j}, q \text{ is } (m^{-A}, 3m^{-A/2}) \text{-close to }p  }. 
\end{align*}
Here for pmfs $p$ and $q$ in $\Delta_k$, $q$ is $(\alpha,\beta)$-close to $p$ means that $p_j \le \alpha \Longrightarrow q_j\le \alpha$, and $p_j>\alpha \Longrightarrow \frac{p_j}{1+\beta}\le q_j\le p_j $ for every $j\in [k]$.
\end{lemma}

The only difference between \cite[Theorem 1]{han2021competitive} and Lemma \ref{lemma:PML_competitive} is that Lemma \ref{lemma:PML_competitive} generalizes the multinomial model in \cite[Theorem 1]{han2021competitive} to the Poisson model. We defer the proof of Lemma \ref{lemma:PML_competitive} to the end of this section. By choosing $\kappa$ small enough and $A$ large enough (both depending on $(A_0, \varepsilon)$), it is clear that Lemma \ref{lemma:PML_competitive} implies \eqbr{eq:PML_hellinger} if we establish the following two claims:
\begin{enumerate}
    \item there exists an estimator $\widehat{G}$ such that \eqbr{eq:PML_assumption} holds with $\delta = \exp(-\Omega(n^{1/3+3\varepsilon/2}))$; 
    \item $r_{n/2}(A) = O\pth{k^{-1}}$ for a large enough constant $A>0$ (e.g. $A = A_0+1$ suffices). 
\end{enumerate}

The first claim simply follows from taking $\widehat{G}$ to be the NPMLE (which depends on $N$ only through its profile) and invoking Lemma \ref{lem:compound_density}. As for the second claim, note that when $q$ is $(n^{-A}, 3n^{-A/2})$-close to $p$, 
\begin{align*}
H^2\pth{f_{G_1}, f_{G_2}} &= H^2\pth{ \frac{1}{k}\sum_{j=1}^k \Poi(np_j), \frac{1}{k}\sum_{j=1}^k \Poi(nq_j) } \\
&\stepa{\le} \frac{1}{k}\sum_{j=1}^k H^2\pth{\Poi(np_j), \Poi(nq_j)} \\
&\stepb{\le} \frac{1}{k}\sum_{j=1}^k n\pth{ \sqrt{p_j} - \sqrt{q_j}}^2 \stepc{\le} \frac{1}{k}\sum_{j=1}^k n\cdot \frac{3}{(n/2)^A} = \frac{3\cdot 2^A}{n^{A-1}}. 
\end{align*}
Here (a) uses the joint convexity of the squared Hellinger distance, (b) follows from $H^2(\Poi(\theta), \Poi(\mu)) = 2(1-e^{-(\sqrt{\theta}-\sqrt{\mu})^2/2}) \le (\sqrt{\theta}-\sqrt{\mu})^2$, and (c) follows from the definition of $(m^{-A}, 3m^{-A/2})$-closeness with $m=n/2$. Since $k\le n^{A_0}$, clearly $r_{n/2}(A) = O\pth{k^{-1}}$ for $A = A_0+1$. 
\end{proof}

\begin{proof}[Proof of Lemma \ref{lemma:PML_competitive}]
Recall that a Poissonized model is simply a multinomial model with a random sample size $\Poi(n)$. An estimator $\widehat{G}$ under the Poissonized model can be written as a sequence of estimators $(\widehat{G}_m)_{m=0}^\infty$ where $\widehat{G}(N) = \widehat{G}_m(N)$ whenever the sample size is $\sum_{i=1}^k N_i = m$; since $\sum_{i=1}^k N_i = \sum_{j=0}^\infty j\Phi_j$ is a function of the profile $\Phi$, each $\widehat{G}_m$ can be made to depend only on the profile $\Phi$. Let $\widehat{G}_m$ be the estimator under the multinomial model with sample size $m$, then \eqbr{eq:PML_assumption} implies that
\begin{align*}
\sum_{m=0}^\infty \P_m\pth{ H^2\pth{f_{G_k}, f_{\widehat{G}_m}} \ge \frac{n^{1/3+\varepsilon}}{k}  } \Poi(m;n) =  \P_{\Poi(n)}\pth{ H^2\pth{f_{G_k}, f_{\widehat{G}}} \ge \frac{n^{1/3+\varepsilon}}{k}  } \le \delta.
\end{align*}
Here we use $\P_m$ to denote the probability under the multinomial model with sample size $m$, and similarly $\P_{\Poi(n)}$ for the Poissonized model. In particular,
\begin{align}\label{eq:PML_multinomial}
    \sup_{p^\star} \P_m\pth{ H^2\pth{f_{G_k}, f_{\widehat{G}_m}} \ge \frac{n^{1/3+\varepsilon}}{k}  } \le \frac{\delta}{\Poi(m;n)}, \quad \forall m\in \naturals.
\end{align}
In addition, since the PML distributions $P_\PML$ take the same form under multinomial and Poissonized models, we obtain
\begin{align*}
&\P_{\Poi(n)}\pth{ H^2\pth{f_{G_k}, f_{G_\PML}} \ge \frac{4n^{1/3+\varepsilon}}{k} + r_{n/2}(A) } \\
&= \sum_{m=0}^\infty \P_m \pth{ H^2\pth{f_{G_k}, f_{G_\PML}} \ge \frac{4n^{1/3+\varepsilon}}{k} + r_{n/2}(A) } \Poi(m;n) \\
&\stepa{\le} \sum_{m=n/2}^{2n} \P_m \pth{ H^2\pth{f_{G_k}, f_{G_\PML}} \ge \frac{4n^{1/3+\varepsilon}}{k} + r_m(A) } \Poi(m;n) + \P(\Poi(n) \notin [n/2,2n]) \\
&\stepb{\le} \sum_{m=n/2}^{2n} \pth{\frac{\delta}{\Poi(m;n)}}^{1-\kappa}\exp\pth{C_0(A,\kappa)m^{1/3+\kappa}} \Poi(m;n) + \P(\Poi(n) \notin [n/2,2n]) \\
&\stepc{\le} \delta^{1-\kappa}\cdot 2n \exp\pth{C_0(A,\kappa)(2n)^{1/3+\kappa}} + 2\exp\pth{-\frac{n}{8}} \\
&=  \delta^{1-\kappa} \exp\pth{O\pth{n^{1/3+\kappa}}} + 2\exp\pth{-\frac{n}{8}},
\end{align*}
which completes the proof of \eqbr{eq:PML_result}. Here (b) uses the upper bound \eqbr{eq:PML_multinomial} and applies \cite[Theorem 1]{han2021competitive} to the multinomial model with sample size $m$ and $L(G_1, G_2) = 2d(G_1,G_2) = H^2(f_{G_1}, f_{G_2})$; note that the triangle inequality $d(G_1,G_2)\le L(G_1,G_3) + L(G_2,G_3)$ required in \cite[Eq.~(6)]{han2021competitive} is true. For the other steps, (a) uses the decreasing property of $m\mapsto r_m(A)$, and (c) is simply the Chernoff bound (cf. Lemma \ref{lemma:chernoff}). 
\end{proof}

\section{Details on the experiments}
\label{sec:experiment-details}

\subsection{Data processing}
\label{sec:dataproc}



We describe the pre-processing of real datasets used in \ifthenelse{\boolean{arxiv}}{\prettyref{subsec:experiment}}{the Experiments section of the main text}.
\begin{itemize}
    \item Text data: We use Python's built-in string module to strip all punctuations and  convert all upper cases to lower cases.

\item 
2000 census surname dataset:
The spreadsheet containing 
frequently occurring surnames (100 or more times) in the 2010 census can be found as File B in \cite{USCensus2010Surnames}.
The last row corresponding to ``all other names'' is removed.

\item 
2020 Census Detailed Demographic and Housing Characteristics File A (Detailed DHC-A) \cite{DHC}:
The raw Detailed DHC-A contains a large number of overlapping groups. 
For the purpose of our experiment, this dataset is processed as follows:
\begin{itemize}
\item \emph{Drop ``$\dots$ alone'' categories.} We keep only categories labeled ``$\dots$ alone \emph{or in any combination}''. For example, ``Albanian alone'' is nested within ``Albanian alone or in any combination'', so the former is removed. Because the dataset does not reveal how racial identities overlap, this step slightly inflates the total population by 33.8 million (due to those who  identify as having two or more races).
\item \emph{Remove regional aggregates.} Regional summaries such as ``European alone or in any combination'' (which subsumes groups like ``Albanian alone or in any combination'') and ``Polynesian'' (which subsumes ``Native Hawaiian'') are deleted to avoid double counting.
\item \emph{Exclude ethnicity groups.} We focus exclusively on race and therefore discard responses to Question 8 of the 2020 Census (``Is Person 1 of Hispanic, Latino, or Spanish origin?''), which would otherwise overlap with race categories.
\item \emph{Drop zero counts.} Groups with zero population are omitted. 
\end{itemize}

After cleaning we obtain a sample of 1446 groups with fewer overlaps. 

\end{itemize}

\subsection{Conditionally applying NPMLE}
\label{sec:experiment-details-conditional}

As noted in \prettyref{sec:NPMLE-computation}, to compute the NPMLE \prettyref{eq:NPMLE-opt}, each iteration of the F-W algorithm requires finding the maximum of 
the gradient function \prettyref{eq:kkt}, a polynomial of degree equal to the maximal count $N_{\max}$.
In large datasets such as the 2000 census surname data, where many surnames appear more than $100,000$ times even after subsampling, 
this results in numerical instability and slows down computation.

To resolve this issue, we adopt a simple modification by applying the NPMLE estimator  \textit{conditionally}, only to those counts below a certain threshold. Specifically, 
\begin{enumerate}
    \item
    Let $S = \{i: N_i \leq\tau\}$ denote the set of low-count symbols which appear at most $\tau$ times.
    Apply NPMLE to these counts $\{N_i: i \in S\}$ to get an estimated probability distribution
    $\{\widehat q_i: i \in S\}$.
    These provide an estimate of the conditional distribution
    \[
    q_i :=\frac{p_i}{p_S}, \quad p_S := \sum_{i\in S} p_i.
    \]

    \item Estimate the total probability mass $p_S$ by
    \[
    \widehat{p_S} := \frac{N_i}{\sum_{i\in S} N_i}.
    \]
    
    \item Finally, we reweight and combine $\widehat q$ with the empirical distribution to produce the full estimate:
    \[
    \widehat p_i
    := \begin{cases}
        \widehat{p_S} \cdot \widehat q_i & i \in S\\
        \frac{N_i}{\sum_{i\in[k]} N_i} & i \notin S
    \end{cases}
    \]    
\end{enumerate}
For the 2000 census surname dataset (\prettyref{fig:census-surname}), we set the threshold to $\tau=20000$.

Empirically, this conditional estimator works comparably to the full NPMLE which is much more time-consuming to compute.
The intuition is that for high-count symbols, the ``signal-to-noise ratio'' is sufficiently high so that the empirical frequency is highly accurate. 
The real challenge lies with low-count symbols, precisely where the NPMLE is most effective.



\subsection{Omitted figures and tables}
\label{sec:all-figs}
In this section we collect plots and tables of KL risk and regret for the experiments\ifthenelse{\boolean{arxiv}}{on}{of the main text for} synthetic data. Specifically, 
\begin{itemize}
    \item 
Figs.~\ref{fig:SI_full_KLregrets1}--\ref{fig:SI_full_KLregrets2} are the full version of \prettyref{fig:KLregrets} in the
\ifthenelse{\boolean{arxiv}}{Introduction}{main text} for synthetic-data experiments.

\item \prettyref{fig:smooth-full} is the full version of \prettyref{fig:smooth} in the \ifthenelse{\boolean{arxiv}}{Introduction}{main text} on the monotonicity and smoothing effect of estimators.

\item \prettyref{tab:oos} contains the KL risks of the out-of-sample experiment on Shakespearean plays in \prettyref{fig:shakespeare-oos} of the \ifthenelse{\boolean{arxiv}}{Introduction}{main text}, as well as \textit{Lord of the Rings}.
\end{itemize}

\begin{figure}[!ht]
     \centering
     \begin{subfigure}[b]{0.45\columnwidth}
         \centering
         \includegraphics[width=\textwidth]{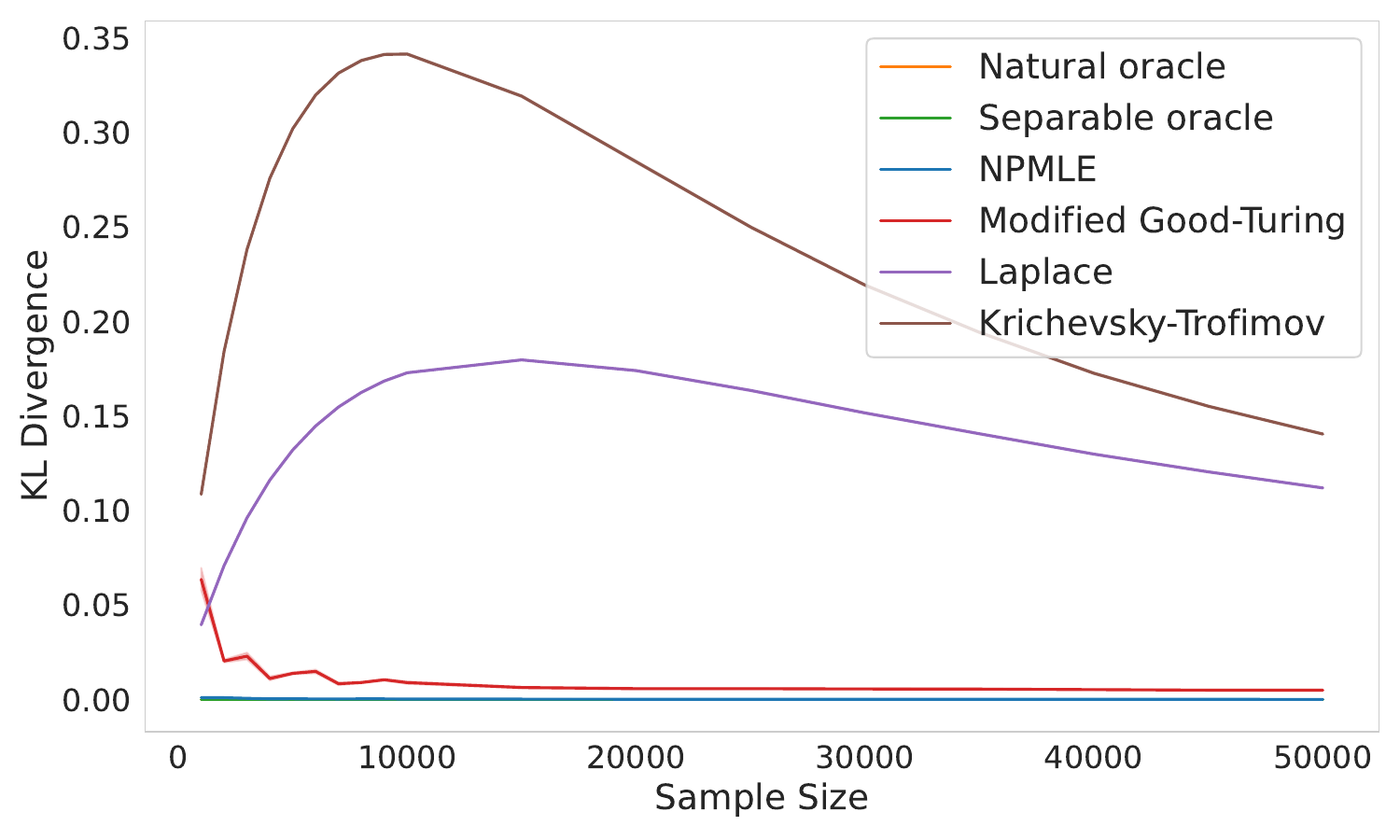}
         \caption{Uniform}
     \end{subfigure}
     \begin{subfigure}[b]{0.45\columnwidth}
         \centering
         \includegraphics[width=\textwidth]{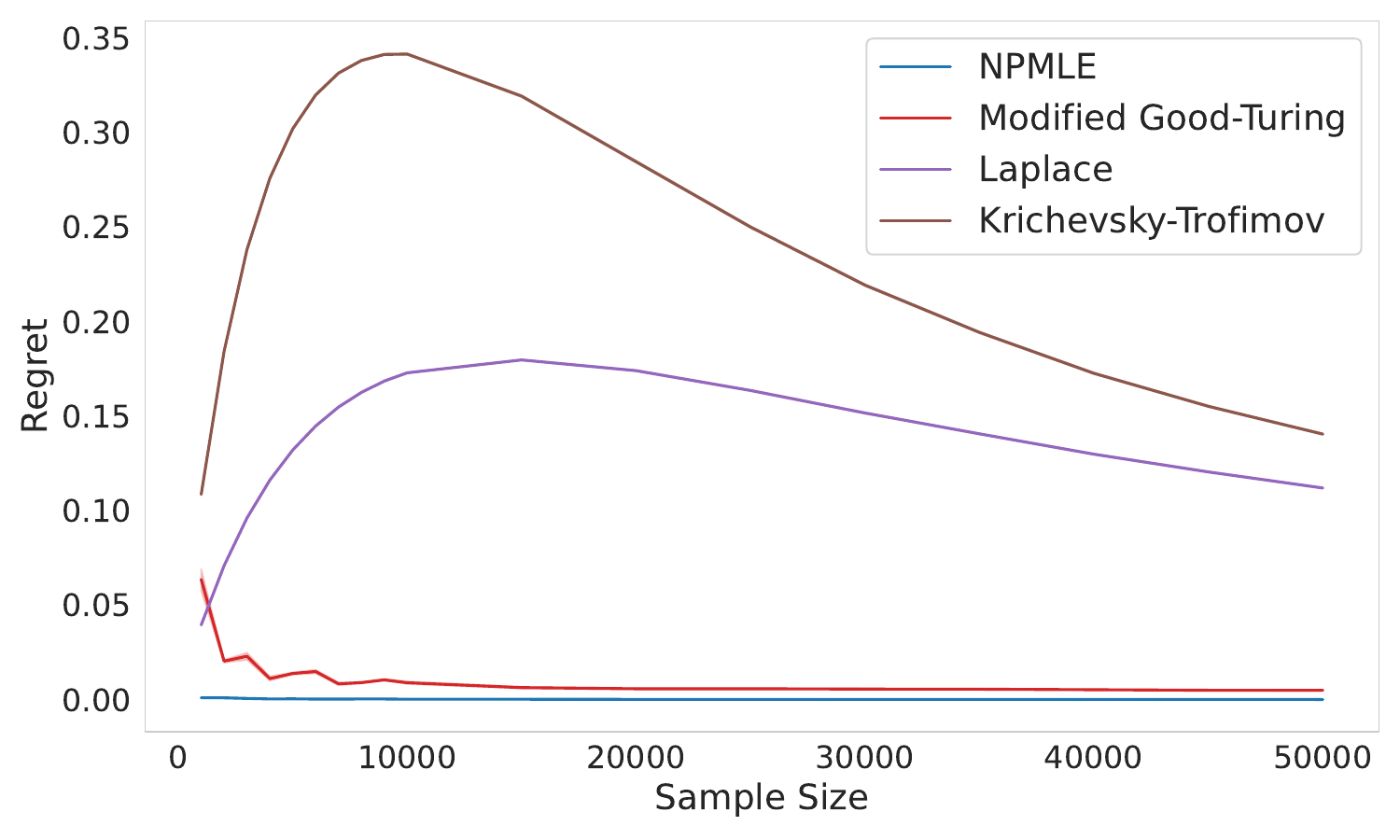}
         \caption{Uniform}
         \label{fig:SI_full_uniform}
     \end{subfigure}

     \begin{subfigure}[b]{0.45\columnwidth}
         \centering
         \includegraphics[width=\textwidth]{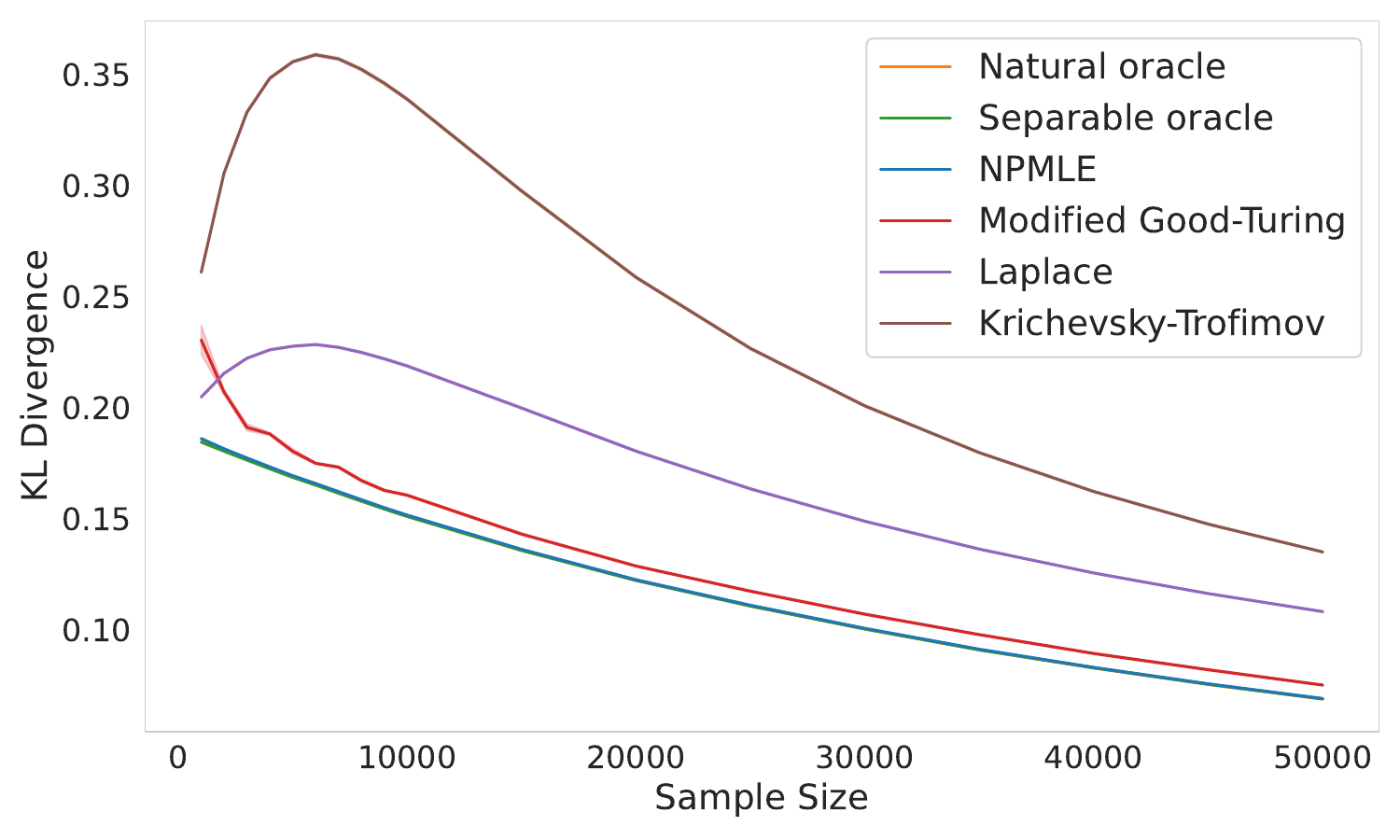}
         \caption{Step.}
     \end{subfigure}
     \begin{subfigure}[b]{0.45\columnwidth}
         \centering
         \includegraphics[width=\textwidth]{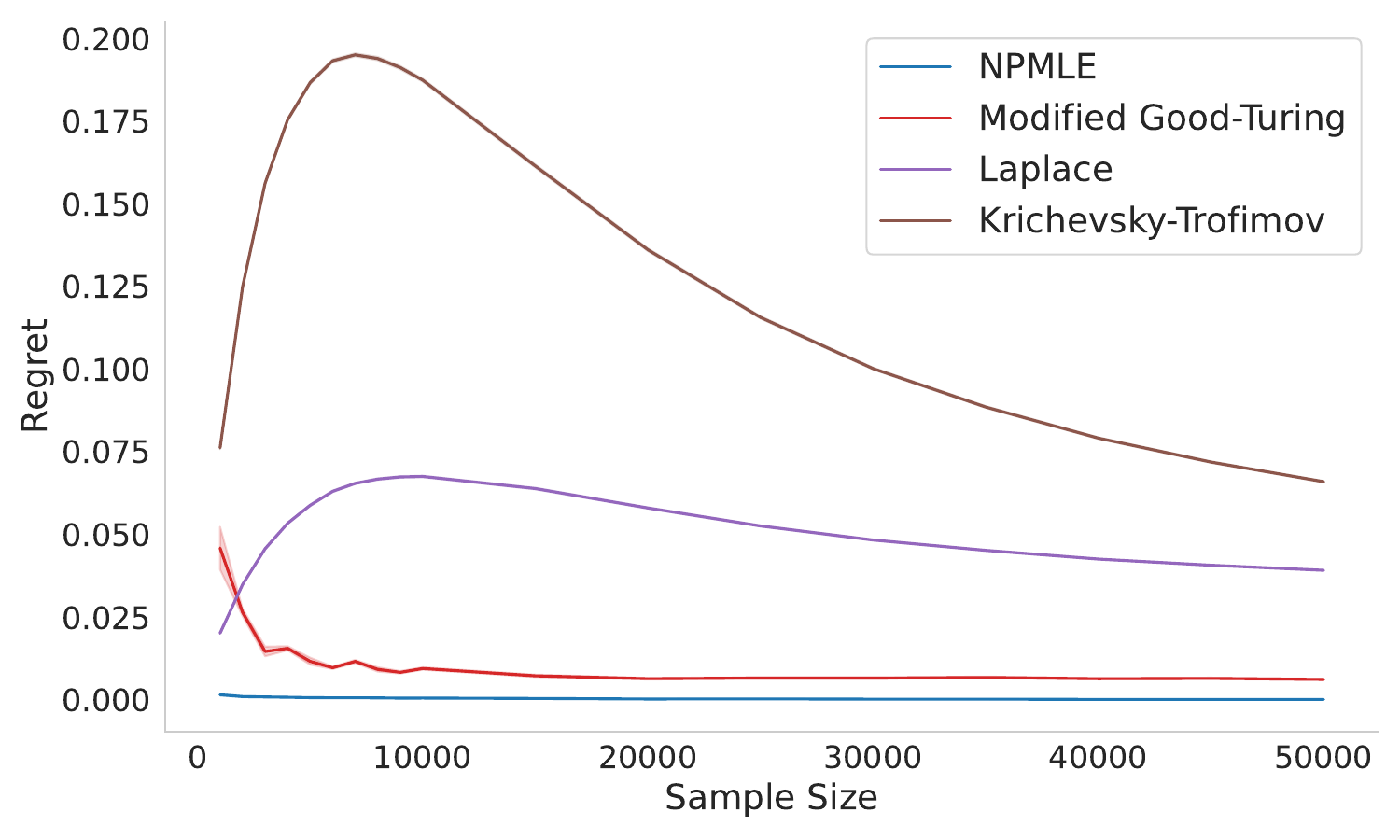}
         \caption{Step.}
         \label{fig:SI_full_step}
     \end{subfigure}

     \begin{subfigure}[b]{0.45\columnwidth}
         \centering
         \includegraphics[width=\textwidth]{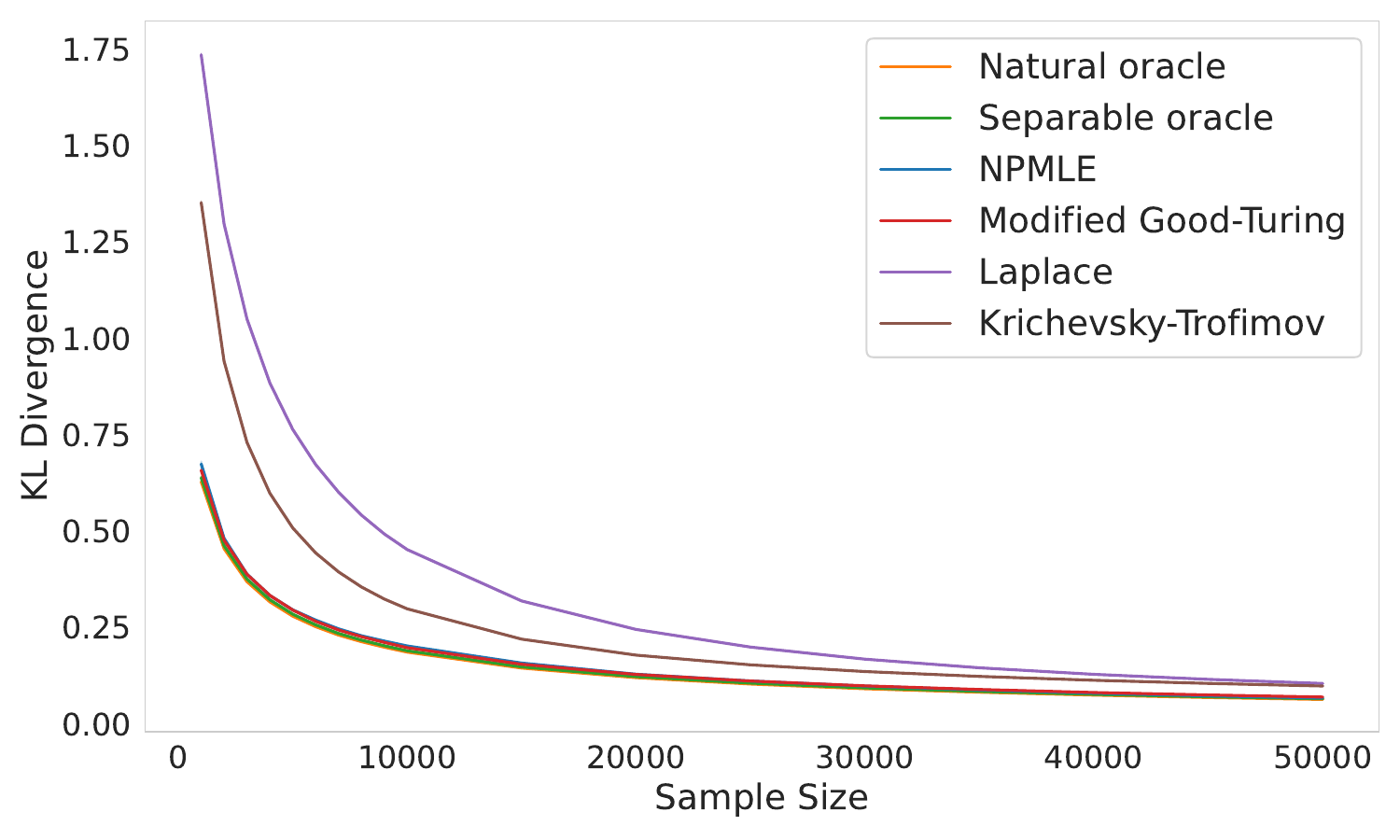}
         \caption{Zipf ($\alpha=1$).}
     \end{subfigure}
     \begin{subfigure}[b]{0.45\columnwidth}
         \centering
         \includegraphics[width=\textwidth]{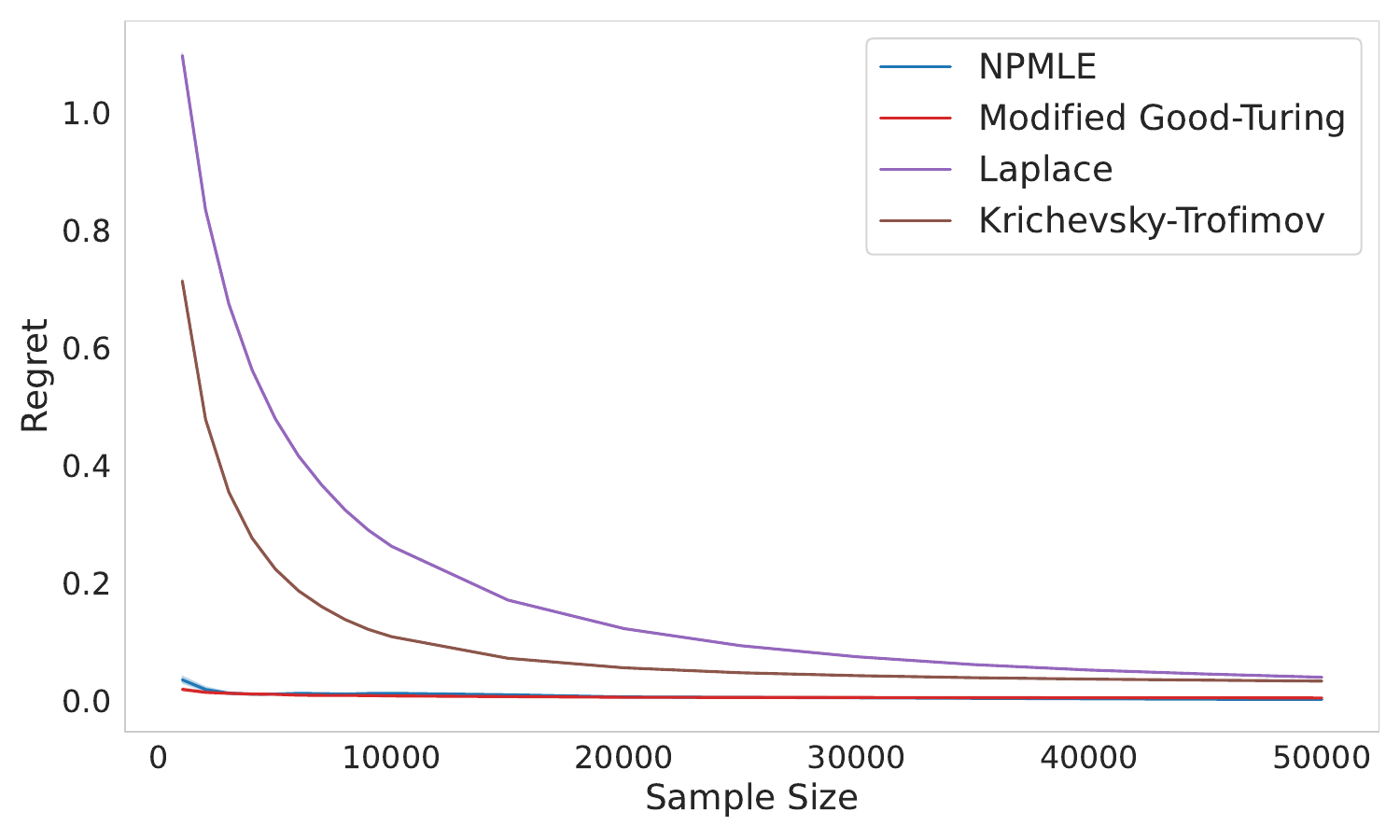}
         \caption{Zipf ($\alpha=1$).}
         \label{fig:SI_full_zipf1}
     \end{subfigure}

     \begin{subfigure}[b]{0.45\columnwidth}
         \centering
         \includegraphics[width=\textwidth]{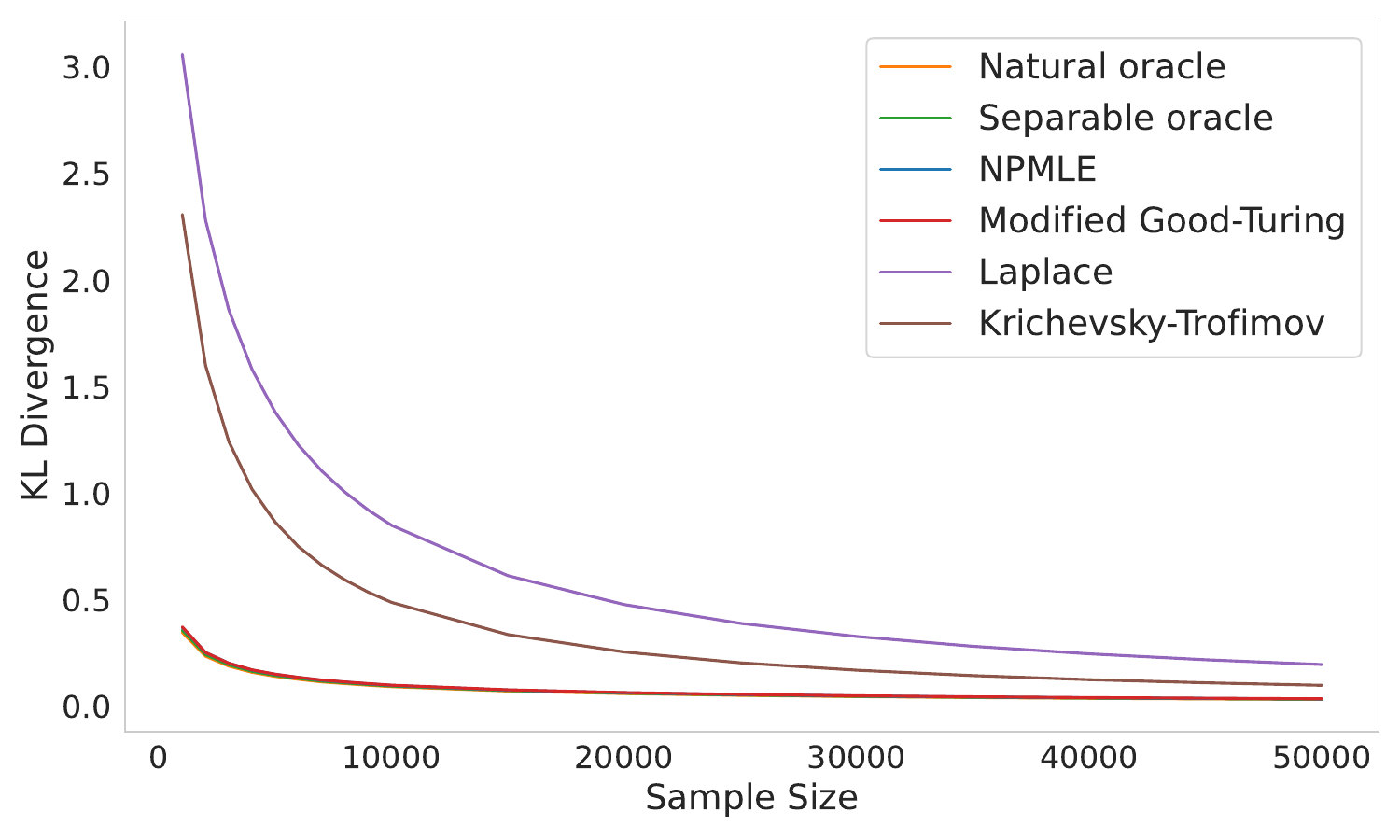}
         \caption{Zipf ($\alpha=1.5$).}
     \end{subfigure}
     \begin{subfigure}[b]{0.45\columnwidth}
         \centering
         \includegraphics[width=\textwidth]{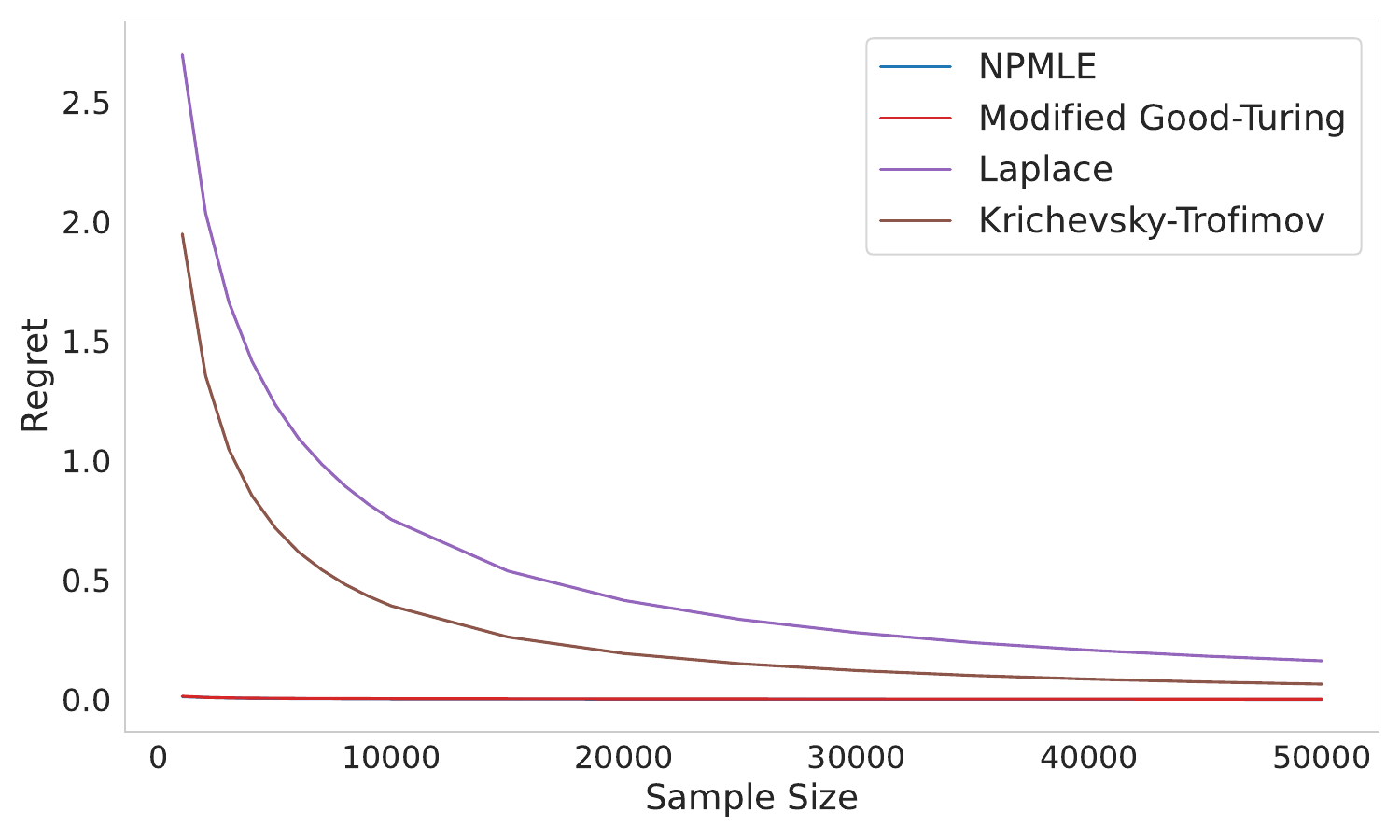}
         \caption{Zipf ($\alpha=1.5$).}
         \label{fig:SI_full_zipf15}
     \end{subfigure}

        \caption{KL risk and regret (over the separable oracle) for various distributions over $k=10000$ elements.}
        \label{fig:SI_full_KLregrets1}
\end{figure}

\begin{figure}[!ht]
     \begin{subfigure}[b]{0.45\columnwidth}
         \centering
         \includegraphics[width=\textwidth]{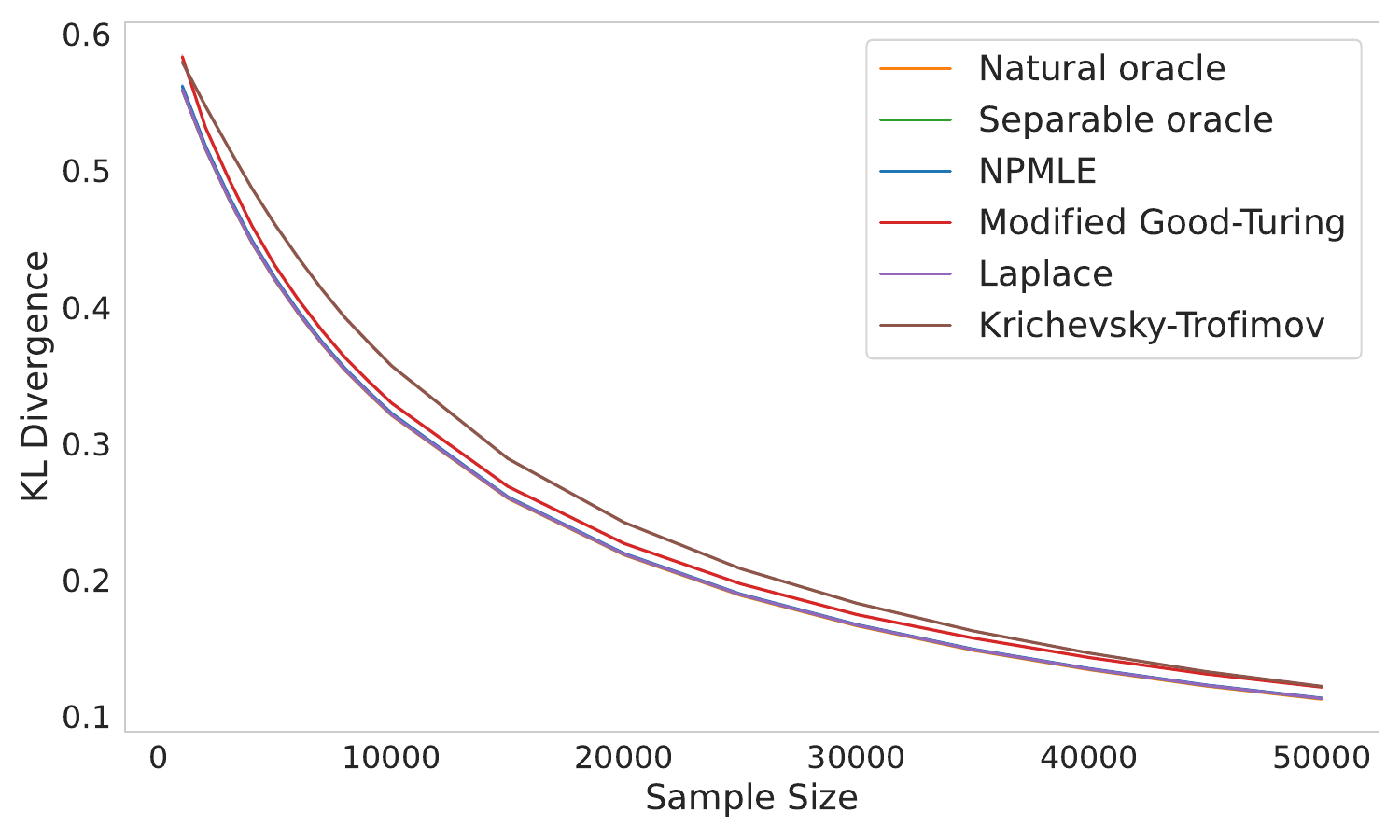}
         \caption{Dirichlet $(c=1)$}
     \end{subfigure}
     \begin{subfigure}[b]{0.45\columnwidth}
         \centering
         \includegraphics[width=\textwidth]{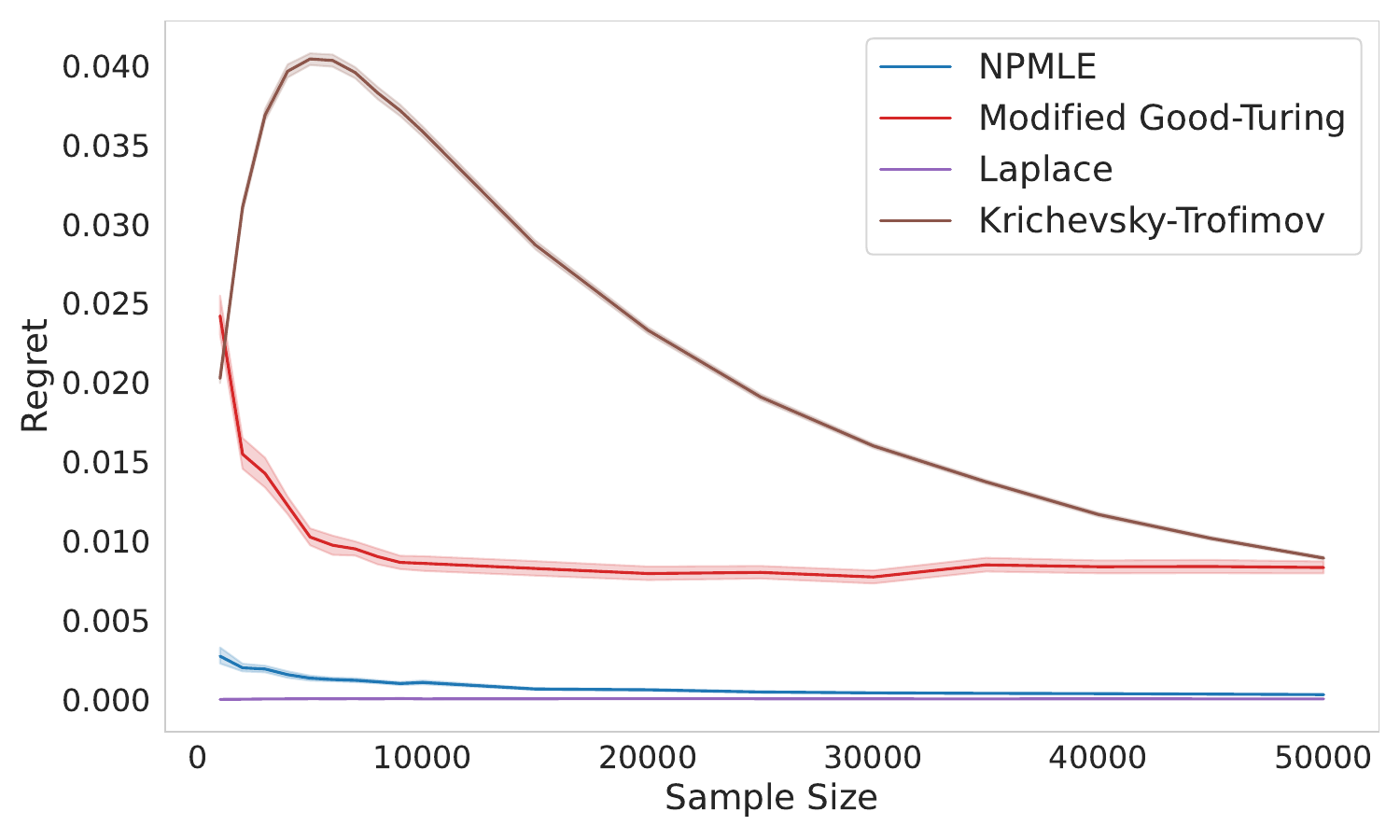}
         \caption{Dirichlet $(c=1)$}
         \label{fig:SI_full_dirichlet1}
     \end{subfigure}
     
     \begin{subfigure}[b]{0.45\columnwidth}
         \centering
         \includegraphics[width=\textwidth]{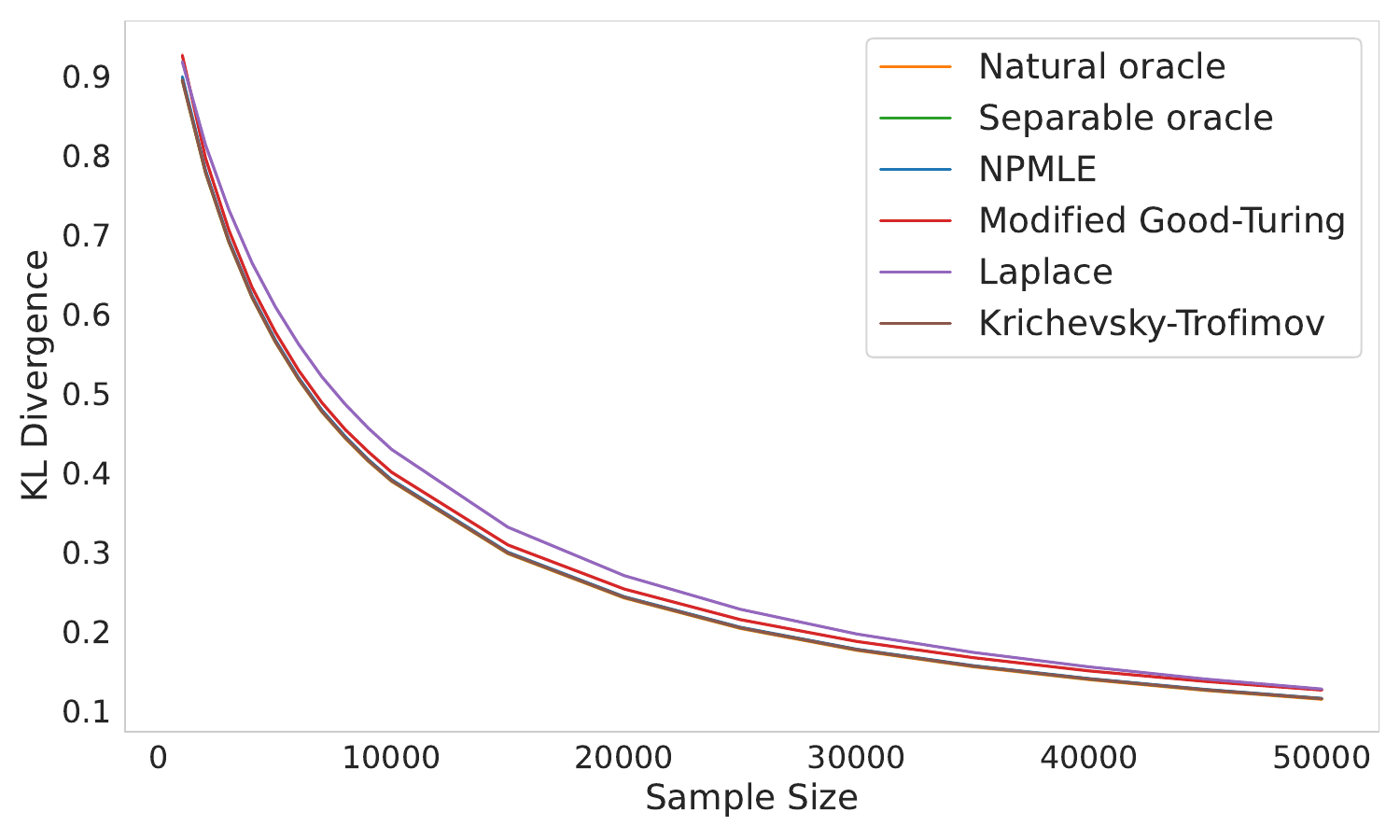}
         \caption{Dirichlet $(c=0.5)$}
     \end{subfigure}
     \begin{subfigure}[b]{0.45\columnwidth}
         \centering
         \includegraphics[width=\textwidth]{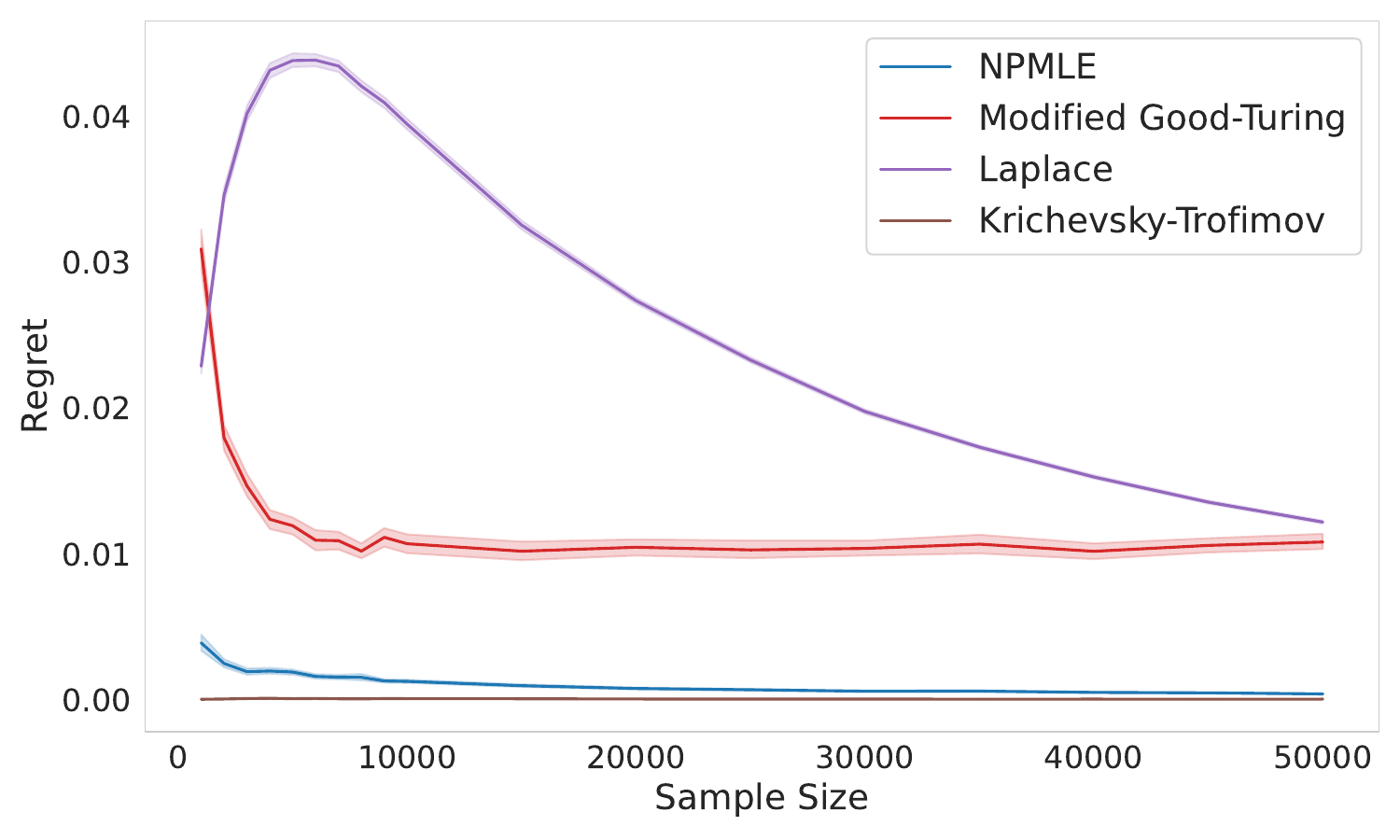}
         \caption{Dirichlet $(c=0.5)$}
         \label{fig:SI_full_dirichlet05}
     \end{subfigure}
     
        \caption{KL risk and regret (over the separable oracle) for various distributions over $k=10000$ elements (continued). 
        Here the Laplace and the Krichevsky estimators are the 
        \textit{exact} Bayes estimators under the Dirichlet priors with parameter $c=1$ and $c=1/2$ respectively and thus achieving zero regret. In both cases, NPMLE attains nearly 
        zero regret, much smaller than that of Good--Turing.
        }
        \label{fig:SI_full_KLregrets2}
\end{figure}

\begin{figure}
\centering
\centering
     \begin{subfigure}[b]{0.8\columnwidth}
         \includegraphics[width=\textwidth]{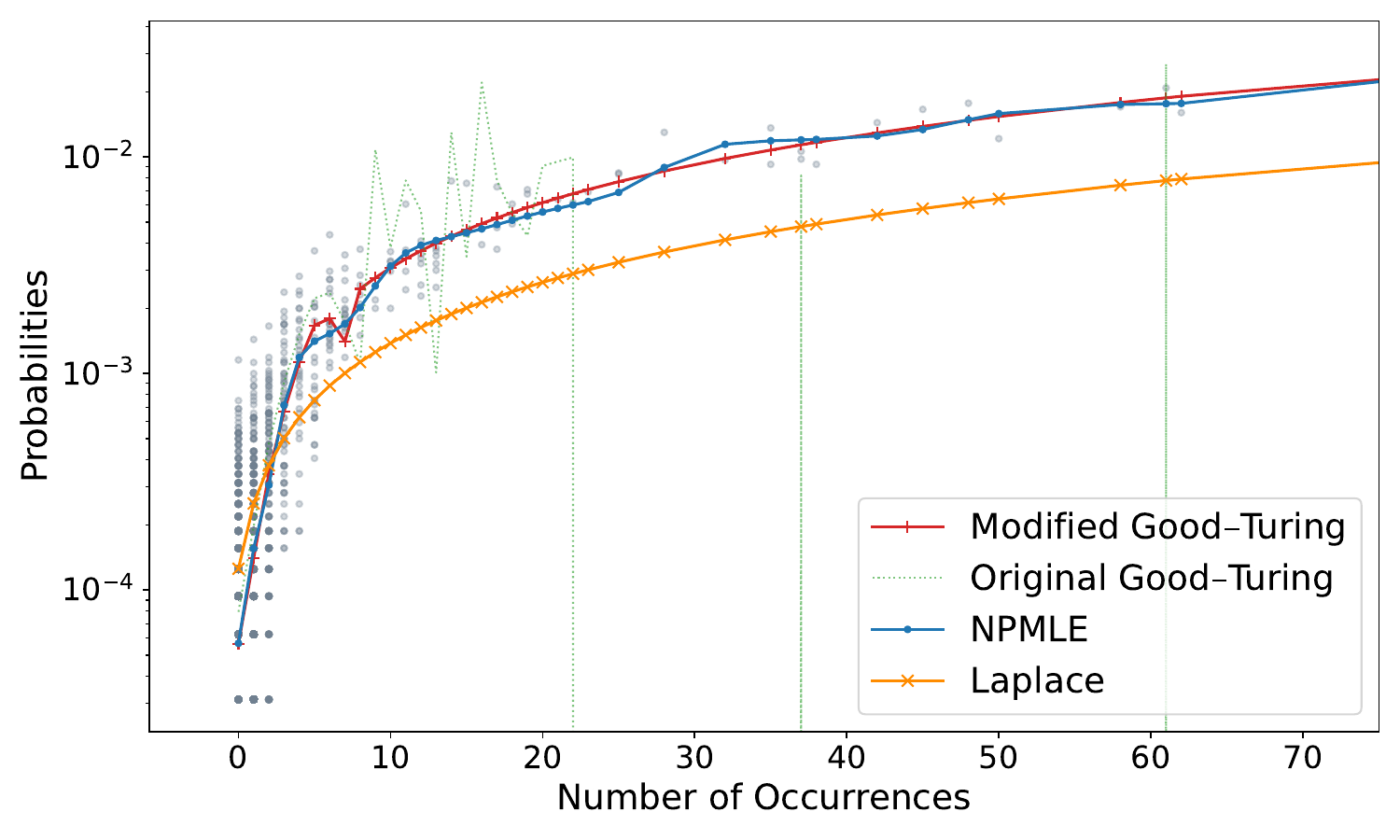}
         \caption{Hamlet, sampling ratio $10\%$.}
         \label{fig:smooth-hamlet}
     \end{subfigure}
     \begin{subfigure}[b]{0.8\columnwidth}
         \includegraphics[width=\textwidth]{figs_final/smoothing_zipf.pdf}
         \caption{Zipf distribution,  $\alpha=1/2, n=15000,
k=10000$.}
         \label{fig:smooth-zipf}
     \end{subfigure}
\caption{Smoothing effect of estimators. Estimated probabilities $\widehat{p}_i$ are plotted in solid against the empirical count $N_i$. Dots in gray are scatter plot of true probabilities $p^\star_i$ versus $N_i$. Unlike the Good--Turing estimators, the estimated probabilities by the NPMLE and Laplace estimators are always positive and monotonically increasing in the empirical count, thanks to their Bayes form.
}
\label{fig:smooth-full}
\end{figure}

\begin{table}[ht]
    \centering    
\begin{tabular}{lccc}
\toprule
\textbf{Text name} & \textbf{Modified Good-Turing} & \textbf{NPMLE} & \textbf{Pretrained Bayes} \\
\midrule
Alls Well That Ends Well & 0.0113 & 0.0060 & 0.0043 \\
Antony and Cleopatra & 0.0127 & 0.0067 & 0.0024 \\
As You Like It & 0.0129 & 0.0064 & 0.0056 \\
A Midsummer Nights Dream & 0.0145 & 0.0073 & 0.0176 \\
Coriolanus & 0.0111 & 0.0064 & 0.0020 \\
Cymbeline & 0.0113 & 0.0065 & 0.0023 \\
Edward III & 0.0126 & 0.0076 & 0.0101 \\
Henry IV Part 1 & 0.0119 & 0.0070 & 0.0033 \\
Henry IV Part 2 & 0.0134 & 0.0058 & 0.0064 \\
Henry V & 0.0103 & 0.0095 & 0.0028 \\
Henry VIII & 0.0111 & 0.0066 & 0.0019 \\
Henry VI Part 1 & 0.0124 & 0.0075 & 0.0077 \\
Henry VI Part 2 & 0.0115 & 0.0076 & 0.0040 \\
Henry VI Part 3 & 0.0125 & 0.0062 & 0.0053 \\
Julius Caesar & 0.0133 & 0.0066 & 0.0177 \\
King John & 0.0157 & 0.0068 & 0.0122 \\
King Lear & 0.0119 & 0.0074 & 0.0027 \\
Loves Labours Lost & 0.0125 & 0.0066 & 0.0070 \\
Macbeth & 0.0157 & 0.0073 & 0.0196 \\
Measure for Measure & 0.0130 & 0.0062 & 0.0054 \\
Much Ado About Nothing & 0.0138 & 0.0055 & 0.0069 \\
Othello & 0.0115 & 0.0057 & 0.0028 \\
Pericles Prince of Tyre & 0.0145 & 0.0074 & 0.0115 \\
Richard II & 0.0120 & 0.0074 & 0.0075 \\
Richard III & 0.0110 & 0.0073 & 0.0029 \\
Romeo and Juliet & 0.0137 & 0.0063 & 0.0067 \\
The Comedy of Errors & 0.0148 & 0.0064 & 0.0199 \\
The Merchant of Venice & 0.0141 & 0.0071 & 0.0069 \\
The Merry Wives of Windsor & 0.0121 & 0.0055 & 0.0039 \\
The Taming of the Shrew & 0.0119 & 0.0061 & 0.0072 \\
The Tempest & 0.0139 & 0.0067 & 0.0156 \\
The Two Gentlemen of Verona & 0.0145 & 0.0060 & 0.0157 \\
The Two Noble Kinsmen & 0.0127 & 0.0069 & 0.0034 \\
The Winter's Tale & 0.0129 & 0.0065 & 0.0025 \\
Timon of Athens & 0.0149 & 0.0060 & 0.0116 \\
Titus Andronicus & 0.0129 & 0.0069 & 0.0096 \\
Troilus and Cressida & 0.0125 & 0.0054 & 0.0095 \\
Twelfth Night & 0.0121 & 0.0061 & 0.0080 \\
Lord of the Rings & 0.0912 & 0.0897 & 0.2701\\
\bottomrule
\end{tabular}
\ifthenelse{\boolean{arxiv}}{\caption{Average KL risks of the out-of-sample experiment in \prettyref{fig:shakespeare-oos} of the Introduction. 
    Each risk is averaged over 200 independent trials.
    The pretrained Bayes estimator is trained on the entirety of \textit{Hamlet}, then applied to each of the other 38 Shakespearean plays, with sampling ratio 20\%, to estimate the word frequency.
    The additional last row shows that the Bayes estimator trained on \textit{Hamlet} does not  generalize well to \textit{Lord of the Rings}.}}{
\caption{Average KL risks of the out-of-sample experiment in \prettyref{fig:shakespeare-oos} of the main text. 
    Each risk is averaged over 200 independent trials.
    The pretrained Bayes estimator is trained on the entirety of \textit{Hamlet}, then applied to each of the other 38 Shakespearean plays, with sampling ratio 20\%, to estimate the word frequency.
    The additional last row shows that the Bayes estimator trained on \textit{Hamlet} does not  generalize well to \textit{Lord of the Rings}.}
    }
    \label{tab:oos}
\end{table}